\documentclass[a4paper,10pt,leqno]{amsart}
\title[On  Brown's Problem,s,and Nielsen Realization]{On  Brown's Problem,
   Poincar\'e models for the classifying spaces for proper actions and Nielsen Realization}
              \author{Wolfgang L\"uck}
        \address{Mathematical Institute of the University at Bonn\\
                Endenicher Allee 60\\
                53115 Bonn, Germany}
         \email{wolfgang.lueck@him.uni-bonn.de}
          \urladdr{http://www.him.uni-bonn.de/lueck}
         \date{January, 2022}

\makeatletter
\@namedef{subjclassname@2020}{%
  \textup{2020} Mathematics Subject Classification}
\makeatother

\usepackage{hyperref}
\usepackage{color}
\usepackage{pdfsync}
\usepackage{calc}
\usepackage{enumerate,amssymb}
\usepackage{graphicx}
\usepackage[arrow,curve,matrix,tips,2cell]{xy}
\usepackage{amssymb}
\SelectTips{eu}{10}\UseTips\UseAllTwocells
\DeclareMathAlphabet{\matheurm}{U}{eur}{m}{n}

\DeclareMathOperator{\aut}{aut}
\DeclareMathOperator{\Aut}{Aut}

\DeclareMathOperator{\cone}{cone}

\DeclareMathOperator{\id}{id}

\DeclareMathOperator{\inte}{int}

\DeclareMathOperator{\mor}{mor}

\DeclareMathOperator{\Out}{Out}

\DeclareMathOperator{\pr}{pr}

\DeclareMathOperator{\tors}{tors}

\DeclareMathOperator{\trf}{trf}
\DeclareMathOperator{\trunc}{trunc}

\DeclareMathOperator{\untw}{untw}

\DeclareMathOperator{\Wh}{Wh}

  \newcommand{\IZ}{\mathbb{Z}}

  \newcommand{\calc}{\mathcal{C}}

  \newcommand{\calf}{\mathcal{F}}

  \newcommand{\calm}{\mathcal{M}}

  \newcommand{\cals}{\mathcal{S}}
  
  \newcommand{\caltr}{\mathcal{TR}}

  \newcommand{\pt}{\{\bullet\}}

  \newcommand{\bub}[1]{\underline{B}#1}        
  \newcommand{\eub}[1]{\underline{E}#1}

  \newcommand{\OrGF}[2]{\matheurm{Or}_{#2}(#1)}

  \newcommand{\calfin}{{\mathcal F}{\mathcal I}{\mathcal N}}

     \newcounter{commentcounter}

     \theoremstyle{plain} \newtheorem{theorem}{Theorem}[section]
      \newtheorem{lemma}[theorem]{Lemma}
     
     \newtheorem{proposition}[theorem]{Proposition}
     \newtheorem{conjecture}[theorem]{Conjecture}
     \newtheorem{assumption}[theorem]{Assumption}
      \newtheorem*{theorem*}{Theorem}
     \newtheorem*{theoremA*}{Theorem A} \newtheorem*{theoremB*}{Theorem B}

     \theoremstyle{definition} \newtheorem{definition}[theorem]{Definition}
      \newtheorem{example}[theorem]{Example}
     \newtheorem{question}[theorem]{Question} \newtheorem{problem}[theorem]{Problem}
      \newtheorem{remark}[theorem]{Remark}
     \newtheorem{notation}[theorem]{Notation} 
      
     \newtheorem*{definition*}{Definition}

     \theoremstyle{remark}

     \makeatletter\let\c@equation=\c@theorem\makeatother

     \newcommand{\CATzero}{\operatorname{CAT(0)}}

     \newcommand{\version}[1] 
     {\begin{center} last edited on #1\\
         last compiled on \today\\
         name of tex-file: \jobname
       \end{center}}

\begin{document}

     \begin{abstract}
       There is the problem, whether for a given virtually torsionfree discrete group
       $\Gamma$ there exists a cocompact proper topological $\Gamma$-manifold, which is
       equivariantly homotopy equivalent to the classifying space for proper actions.  It
       is related to Nielsen's Realization and to the problem of Brown, whether there is a
       $d$-dimensional model for the classifying space for proper actions, if the
       underlying group has virtually cohomological dimension $d$.  Assuming 
       that the expected manifold model has a zero-dimensional singular set, we solve the
       problem in the Poincar\'e category and obtain new results about Brown's problem
       under certain conditions concerning the underlying group, for instance if it 
       is hyperbolic. In a sequel paper together with James Davis we will deal with this
       on the level of topological manifolds.
     \end{abstract}

\maketitle

\newlength{\origlabelwidth}\setlength\origlabelwidth\labelwidth


\typeout{------------------- Introduction -----------------}
\section{Introduction}%
\label{sec:introduction}


\subsection{Manifold models for the classifying space for proper actions}%
\label{subsec:Manifold_models_for_the_classifying_space_for_proper_actions_intro}

Let $\Gamma$ be a discrete group. One can associate to it the \emph{classifying space for
  proper $\Gamma$-actions} $\eub{\Gamma}$.  This is a $\Gamma$-$CW$-complex, whose
isotropy groups are finite and whose $H$-fixed point set $\eub{\Gamma}^H$ is contractible
for every finite subgroup $H \subseteq \Gamma$.  Two such models are $\Gamma$-homotopy
equivalent. These $\Gamma$-$CW$-complexes $\eub{\Gamma}$ are interesting in their own
right and have many applications to group theory and equivariant homotopy theory and play
a prominent role for the Baum-Connes Conjecture and the Farrell-Jones Conjecture. For a
survey on $\eub{\Gamma}$ we refer for instance to~\cite{Lueck(2005s)}.

A natural question is, whether there are nice models for $\eub{\Gamma}$. One may ask for
finiteness properties up to $\Gamma$-homotopy equivalence such as being finite or finite
dimensional. One may also try to address the much harder question, whether there is a
\emph{cocompact manifold model} for $\eub{\Gamma}$, i.e., whether there is a cocompact
proper $\Gamma$-manifold without boundary, which is homotopy equivalent to
$\eub{\Gamma}$. If the answer is yes, one of course encounters the problem, in which sense
such a manifold model is unique. In general it makes a difference, whether one is asking
this question in the smooth or in the topological category.  We will mainly deal with the
topological category.

In general there is no cocompact manifold model. There are some well-known prominent
examples, where cocompact manifold models exist. For instance, if $\Gamma$ acts properly,
isometrically and cocompactly on a Riemannian manifold $M$ with non-positive sectional
curvature, then $M$ is a cocompact model for $\eub{\Gamma}$. Let $L$ be a connected Lie
group.  Then $L$ contains a maximal compact subgroup $K$, which is unique up to
conjugation. If $\Gamma \subseteq L$ is a discrete subgroup such that $L/\Gamma$ is
compact, then $M = L/K$ is a smooth manifold with smooth $\Gamma$-action and a cocompact
model for $\eub{\Gamma}$.

Our long term goal is to answer the question about a manifold model in a specific but
interesting situation.  Namely, we will assume throughout this paper that $\Gamma$ is
virtually torsionfree, in other words that there is a group extension

\begin{equation}\label{group_extension_intro}
  1 \to \pi \xrightarrow{i} \Gamma \xrightarrow{p} G \to 1
\end{equation}
such that $\pi$ is torsionfree and $G$ is finite. Constructing a manifold model $M$ for
$\eub{\Gamma}$ is very hard, since one has to deal with equivariant surgery. We will
confine ourselves to the case, where $M$ is \emph{pseudo-free}, i.e., the singular set
$M^{> 1} = \{x \in M \mid \Gamma_x \not= \{1\}\}$ is discrete, or, equivalently,  the
$\Gamma$-space $M^{>1}$ is the disjoint union of its $\Gamma$-orbits. Here and in the
sequel $\Gamma_x$ denotes the isotropy group
$\{\gamma \in \Gamma \mid \gamma \cdot x = x\}$ of $x \in M$.

Let $d$ be the dimension of $M$.  Moreover, we will assume that $M^{> 1}$ nicely embeds
into $M$, namely, for every $x \in M^{>1}$ we assume the existence of a $d$-dimensional
disk $D_x$ with $\Gamma_x$-action and an $\Gamma_x$-equivariant embedding
$D_x \subseteq M$ such that the origin of $D_x$ is mapped to $x$ and the $\Gamma_x$-action
on $D_x$ is free outside the origin.  Let $S_x$ be the boundary of $D_x$, which is a
$(d-1)$-dimensional sphere with free $\Gamma_x$-action. We will assume  that $S_x$ carries
a $\Gamma_x$-$CW$-structure, which is automatically the case, if $d \not = 5$, just
apply~\cite[Section~9.2]{Freedman-Quinn(1990)} or~\cite[Theorem~2.2 in III.2 on
page~107]{Kirby-Siebenmann(1977)} to $S_x/\Gamma_x$. Note that this implies that
$(D_x,S_x)$ is a finite $\Gamma_x$-$CW$-pair, since $D_x$ is the cone over $S_x$.

Let $I \subset M^{>1}$ be a set-theoretic transversal of the projection $M^{>1} \to M^{>1}/\Gamma$. Put
\begin{eqnarray}
  \partial X
  &  := &
          \coprod_{x \in I} \Gamma \times_{\Gamma_x} S_x;
          \label{partial_X_intro}         
  \\
  C(\partial X)
  & := &
         \coprod_{x \in I} \Gamma \times_{\Gamma_x} D_x.
         \label{C(partial_X)_intro}                
\end{eqnarray}
Then we get a free cocompact proper $\Gamma$-manifold $X$ with boundary $\partial X$, if we
put
\begin{equation}
  X : = M \setminus \left(\coprod_{x \in I} \Gamma \times _{\Gamma_x} (D_x \setminus S_x)\right)
  \label{def_Of_X_intro}
\end{equation}
and $M$ becomes the $\Gamma$-pushout
\begin{equation}
  \xymatrix{\partial X = \coprod_{x \in I} \Gamma \times_{\Gamma_x}  S_x 
    \ar[r] \ar[d]
    &
    X \ar[d]
    \\
    C(\partial X) := \coprod_{x \in I} \Gamma \times_{\Gamma_x} D_x \ar[r]
    &
    M.
  }
  \label{Gamma_pushout_for_M_intro}
\end{equation}
Note that the existence of the disks $D_x$ and the existence of the
diagram~\eqref{Gamma_pushout_for_M_intro} is guaranteed, if $M$ is a smooth
$\Gamma$-manifold with discrete $M^{>1}$, since we can choose $C(\partial X)$ to be  a tubular
neighbourhood with boundary $\partial X$ of the zero-dimensional $\Gamma$-invariant smooth
submanifold $M^{>1}$ of $M$. Our main concern is

\begin{problem}[Main Problem]\label{prob:main_problem_intro}
  Does there exist a proper cocompact $\Gamma$-manifold $M$ of the shape described
  in~\eqref{Gamma_pushout_for_M_intro}, which is $\Gamma$-homotopy equivalent to
  $\eub{\Gamma}$?

  If yes, are two such $\Gamma$-manifolds $\Gamma$-homeomorphic?
\end{problem}

The collection $\{S_x \mid x \in I\}$ is an example of a so called \emph{free
  $d$-dimensional slice system}, and this notion will be analyzed in
Section~\ref{sec:free_slice-systems}. The pair $(X,\partial X)$ is an example of a so
called \emph{slice complement model for $\eub{\Gamma}$}, and this notion will be investigated in
Section~\ref{subsec:Slice-models}.


\subsection{Some necessary conditions}%
\label{subsec:Some_necessary_conditions_intro}

Next we figure out some necessary conditions for the existence of a $\Gamma$-manifold $M$,
which is $\Gamma$-homotopy equivalent to $\eub{\Gamma}$ and is of the shape described
in~\eqref{Gamma_pushout_for_M_intro}.

Firstly we prove that any closed topological manifold $M$ of the shape described
in~\eqref{Gamma_pushout_for_M_intro} has the $\Gamma$-homotopy type of a
$\Gamma$-$CW$-complex.

The pair $(C(\partial X),\partial X)$ is a finite proper $\Gamma$-$CW$-pair of dimension
$d$.  The pair $(X/\Gamma, \partial X/\Gamma)$ is a topological compact manifold with
boundary and hence is homotopy equivalent relative $\partial X$ to a finite relative
$CW$-complex $(Y,\partial X/\Gamma)$ of relative dimension $d$
see~\cite[Section~9.2]{Freedman-Quinn(1990)} or~\cite[Theorem~2.2 in III.2 on
page~107]{Kirby-Siebenmann(1977)}. Hence we can find a finite $d$-dimensional
$\Gamma$-$CW$-pair $(\overline{Y},\partial X)$, which is relatively free and is $\Gamma$-homotopy equivalent
relative $\partial X$ to $(X, \partial X)$. Take
$Z = C(\partial X) \cup_{\partial X} \overline{Y}$.  Then $Z$ is a finite proper
$d$-dimensional $\Gamma$-$CW$-complex with $Z^{>1} = M^{>1}$ such that $M$ and $Z$ are
$\Gamma$-homotopy equivalent.

Note that this implies that $M$ is $\Gamma$-homotopy equivalent to $\eub{\Gamma}$, if and
only if $M$ is contractible.

\begin{notation}
  Let $\calm$ be a complete system of representatives of the conjugacy classes of maximal
  finite subgroups of $\Gamma$.  Put
  \begin{eqnarray*}
    \partial E\Gamma & := & \coprod_{F \in \calm} \Gamma \times_F EF;
    \\
    \partial \eub{\Gamma} & := & \coprod_{F \in \calm} \Gamma /F;
    \\
    \partial B\Gamma & := & \coprod_{F \in \calm} BF;
    \\
    \bub{\Gamma} & := &\eub{\Gamma/\Gamma}.
  \end{eqnarray*}            
\end{notation}
  
In the sequel $H_*(Z) = H_*(Z,\IZ)$ for a space or $CW$-complex $Z$ denotes its singular
or cellular homology with coefficients in $\IZ$.
  
We will consider the following conditions for the group $\Gamma$, where
$H_d(B\Gamma,\partial B\Gamma)$ is  the homology
of the canonical map $\partial B\Gamma \to B\Gamma$.

\begin{notation}\label{not:(M)_and_(NM)_intro}\

  \begin{itemize}

  \item[(M)] Every non-trivial finite subgroup of $\Gamma$ is contained in a unique
    maximal finite subgroup;

  \item[(NM)] If $F$ is a maximal finite subgroup, then $N_{\Gamma}F = F$;

  \item[(H)] For the homomorphism $w \colon \Gamma \to \{\pm 1\}$ of~Notation~\ref{not:w}  the composite
      \begin{multline*}
        \quad \quad \quad H_d^{\Gamma}(E\Gamma,\partial E\Gamma;\IZ^w)
        \xrightarrow{\partial} H_{d-1}^{\Gamma}(\partial E\Gamma;\IZ^w)
      \xrightarrow{\cong} \bigoplus_{F \in \calm} H^F_{d-1}(EF;\IZ^{w|_F}) \\
      \to H^F_{d-1}(EF;\IZ^{w|_F})
    \end{multline*}
    of the boundary map, the inverse of the obvious isomorphism, and the projection to the
    summand of $F \in \calm$ is surjective for all $F \in \calm$.

  \end{itemize}

\end{notation}

For simplicity we will assume in the remainder of this Subsection~\ref{subsec:Some_necessary_conditions_intro}
that $\Gamma$-acts orientation
preservingly on $M$, or, equivalently, that the homomorphism
$w \colon \Gamma \to \{\pm 1\}$ of Notation~\ref{not:w}  is trivial,
and that $d = \dim(M)$ is even. These condition will be dropped in the
main body of the paper, for instance the case, where $d$ is odd and $d\ge 3$ is discussed in
Section~\ref{sec:The_case_of_odd_d}.  Note that then the composite appearing in condition (H) above reduces to
\[
H_d(B\Gamma,\partial B\Gamma) \xrightarrow{\partial} H_{d-1}(\partial B\Gamma)
\xrightarrow{\cong} \bigoplus_{F \in \calm} H_{d-1}(BF) \to H_{d-1}(BF).
\]

  \begin{lemma}\label{lem:necessary_condition_for_the_existence_of_M_intro}
    Suppose that the topological $\Gamma$-manifold $M$ of the shape described
    in~\eqref{Gamma_pushout_for_M_intro} is $\Gamma$-homotopy equivalent to $\eub{\Gamma}$
    and $\Gamma$ acts orientation preserving on $M$. Suppose that $d = \dim(M)$ is even and $d \ge 4$. 

    Then the following conditions are satisfied:

    \begin{enumerate}
    \item\label{lem:necessary_condition_for_the_existence_of_M_intro:eub(Gamma)} There is
      a finite $d$-dimensional $\Gamma$-$CW$-model for $\eub{\Gamma}$ such that
      $\eub{\Gamma}^{>1}$ is $0$-dimensional;

    \item\label{lem:necessary_condition_for_the_existence_of_M_intro:(M)_and_(NM)_intro}
      The group $\Gamma$ satisfies (M) and (NM). Moreover, we get a well-defined bijection
      \[
        I \xrightarrow{\cong} \calm
      \]
      by sending $x \in I$ to the element $F \in \calm$, which is conjugated to $\Gamma_x$;
    
    \item\label{lem:necessary_condition_for_the_existence_of_M_intro:Bpi} The quotient
      $M/\pi$ is an orientable aspherical closed $d$-dimensional manifold. In particular there is a
      finite $CW$-model for $B\pi$ and $H_d(B\pi)$ is infinite cyclic;

    \item\label{lem:necessary_condition_for_the_existence_of_M_intro:group_homology} The
      group $H_d(B\Gamma,\partial B\Gamma )$ is infinite cyclic. For each
      $F \in \calm$ the Tate cohomology of $F$ is $d$-periodic and $H_{d-1}(BF)$ is finite
      cyclic of order $|F|$;

    \item\label{lem:necessary_condition_for_the_existence_of_M_intro:(H)} The group
      $\Gamma$-satisfies (H).
    
    \end{enumerate}
  \end{lemma}
  \begin{proof}~\eqref{lem:necessary_condition_for_the_existence_of_M_intro:eub(Gamma)}
    This follows from the consideration in the beginning of this
    Subsection~\ref{subsec:Some_necessary_conditions_intro}.
    \\[1mm]~\eqref{lem:necessary_condition_for_the_existence_of_M_intro:(M)_and_(NM)_intro}
    Consider a non-trivial finite subgroup $H \subseteq \Gamma$. Then $H$ is contained in
    some maximal finite subgroup, since $\Gamma$ is virtually torsionfree.
  
    Let $F_0$ and $F_1$ be two maximal finite subgroups with $H \subseteq F_0$ and
    $H \subseteq F_1$.  Since $M^{F_0} \subseteq M^{H}$ and $M^{F_1} \subseteq M^{H}$
    holds and $M^H$, $M^{F_0} $ and $M^{F_1}$ are contractible zero-dimensional manifolds
    and hence consist of precisely one point $x$, we get
    $M^{F_0} = M^{H} = M^{F_0} =\{x\}$ and hence $F_0 \subseteq \Gamma_x$ and
    $F_1 \subseteq \Gamma_x$. Since $F_0$ and $F_1$ are maximal finite and $\Gamma_x$ is
    finite, we conclude $F_0 = F_1 = \Gamma_x$.  Hence (M) is satisfied, the isotropy
    group of any point in $M$ is either trivial or maximal finite, and any maximal finite
    subgroup occurs as isotropy group $\Gamma_x$ of some element $x$ in $M^{>1}$. Hence we
    get the desired bijection $I \xrightarrow{\cong} \calm$, which we will use in the
    sequel as an identification.
  
    Consider $F \in \calm$. Then $M^F$ consists of precisely one point $x$.  Since
    $N_\Gamma F$ leaves $M^F$ invariant, we conclude $N_\Gamma F \subseteq \Gamma_x =
    F$. Hence (NM) is satisfied.
    \\[1mm]~\eqref{lem:necessary_condition_for_the_existence_of_M_intro:Bpi} Since $\pi$
    is torsionfree and has finite index in $\Gamma$, it acts freely, properly, and
    cocompactly on $M$.  Hence $M/\pi$ is a closed orientable manifold and a finite
    $CW$-complex. This implies that $H_d(B\pi)$ is infinite cyclic.  Since $M$ is
    contractible, $M/\pi$ is a model for $B\pi$.
    \\[1mm]~\eqref{lem:necessary_condition_for_the_existence_of_M_intro:group_homology}
    and~\eqref{lem:necessary_condition_for_the_existence_of_M_intro:(H)} There is the
    following commutative diagram
    \[
      \xymatrix{H_d(M/\Gamma,\partial M/\Gamma) \ar[r]^-{\partial} \ar[d]^-{\cong} &
        H_d(\partial M/\Gamma) & \bigoplus_{F \in\calm} H_{d-1}(S_F/F) \ar[l]_-{\cong}
        \ar[d]
        \\
        H_d(B\Gamma,\partial B\Gamma)\ar[r]^-{\partial} \ar[d]^-{\cong} & H_{d-1}(\partial
        B\Gamma) & \bigoplus _{F \in\calm} H_{d-1}(BF) \ar[l]_-{\cong}
        \\
        H_d(\bub{\Gamma}, \partial \bub{\Gamma}) & &
        \\
        H_d(\bub{\Gamma}), \ar[u]_{\cong} & & }
    \]
    whose maps are given by the obvious maps on space level, boundary homomorphisms, or
    the classifying maps $M/\Gamma \to B\Gamma$ and $S_F/F \to BF$.  We will show using a
    transfer argument and conditions (M) and (NM) that the left vertical arrows are all
    bijective and that the inclusion $i \colon \pi \to \Gamma$ induces an injection of
    infinite cyclic groups $H_d(B\pi) \to H_d(\bub{\Gamma})$,
    see~\eqref{diagram_of_the_H_d-s} and Lemma~\ref{lem;everything_is_infinite_cyclic}.

    Since $F$ acts freely on the $(d-1)$-dimensional sphere $S_F$, the finite group $F$
    has periodic cohomology, see~\cite[page~154]{Brown(1982)}. Moreover $H_{d-1}(F)$ is
    cyclic of order $|F|$, see~\cite[Theorem~9.1 in Section~VI.9 on
    page~154]{Brown(1982)}. Each map $H_{d-1}(S_F/F) \to H_{d-1}(BF)$ is surjective, as
    $S_F$ is $(d-2)$-connected. Since $(M/\Gamma, \partial M/\Gamma)$ is an orientable connected 
    compact manifold, $H_d(M/\Gamma,\partial M/\Gamma)$  is infinite cyclic
    and the image of the  fundamental class in $H_d(M/\Gamma,\partial M/\Gamma)$
    under
    \[
      H_d(M/\Gamma,\partial M/\Gamma) \xrightarrow{\partial_d} H_{d-1}(\partial M/\Gamma)
      \cong \bigoplus_{C \in \pi_0(\partial M/\Gamma)} H_{d-1}(C)
     \]
     is given by the collection of the fundamental classes in $H_{d-1}(C)$ for each path component
     $C$ of $\partial M$. Hence the composite          
    \[H_d(M/\Gamma,\partial M/\Gamma) \xrightarrow{\partial_d} H_{d-1}(\partial M/\Gamma)
      \xrightarrow{\cong} \bigoplus_{F \in \calm} H_{d-1}(S_F/F) \to H_{d-1}(S_F/F)
    \]
    of the boundary map, the inverse of the obvious isomorphism, and the projection to the
    summand of $F \in \calm$ is surjective for all $F \in \calm$. This finishes the proof
    of Lemma~\ref{lem:necessary_condition_for_the_existence_of_M_intro}.
  \end{proof}

  \begin{example}\label{exa:kremer_and_Kirstein}
    The following example is due to Dominik Kirstein and Christian Kremer.  If
    $G$ is cyclic of order $p$ for a prime number $p$, and
    $H_k(B\pi;\IZ)_{(p)}$ vanishes for $1 \le k \le (d-1)$, then (H) is automatically
    satisfied. Its proof is left to the reader.
  \end{example}


  \subsection{Brown's problem}\label{subsec:Browns_problem_intro}

  The condition appearing in
  Theorem~\ref{lem:necessary_condition_for_the_existence_of_M_intro}~%
\eqref{lem:necessary_condition_for_the_existence_of_M_intro:eub(Gamma)} is hard to check
  and looks very restrictive.  In particular we have to find a finite $d$-dimensional
  $\Gamma$-$CW$-complex model for $\underline{E}\Gamma$.  This is a necessary condition,
  which is not at all obvious.  In view of
  Theorem~\ref{lem:necessary_condition_for_the_existence_of_M_intro}~%
\eqref{lem:necessary_condition_for_the_existence_of_M_intro:Bpi} the assumption that
  there is a finite $d$-dimensional model for $B\pi$ is reasonable and will be made.  So
  we are dealing with a special case of the following problem due to Brown~\cite[page~32]{Brown(1979)}.

\begin{problem}[Brown's problem]\label{prob:Browns_problem_intro}
  For which discrete groups $\Gamma$, which contain a torsionfree subgroup $\pi$ of finite
  index and have virtual cohomological dimension $\le d$, does there exist a
  $d$-dimensional $\Gamma$-$CW$-model for $\underline{E}\Gamma$?
\end{problem}

Meanwhile there are examples, where the answer is negative for Brown's
problem~\ref{prob:Browns_problem_intro}
see~\cite{Leary-Nucinkis(2003)},~\cite{Leary-Petrosyan(2017)}. So in order to prove that
the condition appearing in
Theorem~\ref{lem:necessary_condition_for_the_existence_of_M_intro}~%
\eqref{lem:necessary_condition_for_the_existence_of_M_intro:eub(Gamma)} is satisfied, we
must give a proof of a positive answer to Brown's problem~\ref{prob:Browns_problem_intro}
in our special case, and actually much more.  We will show

\begin{theorem}[Models for the classifying space for proper $\Gamma$-actions]%
\label{the:Models_for_the_classifying_space_for_proper_Gamma-actions_intro}
  Assume  that the following conditions are satisfied:

  \begin{enumerate}

    \item\label{the:Nielsen_Realization_Problem:d_intro}
      The natural number $d$ satisfies $d \ge 3$;

      \item\label{the:Nielsen_Realization_Problem:(M)_and)NM)_intro} The group $\Gamma$ satisfies
    conditions (M) and (NM), see Notation~\ref{not:(M)_and_(NM)_intro};

  \item\label{the:Nielsen_Realization_Problem:eub}

    The group $\Gamma$ satisfies one of the following conditions:

    \begin{enumerate}

    \item\label{the:Nielsen_Realization_Problem:eub_direct} There exists a finite
      $\Gamma$-$CW$-model for $\eub{\Gamma}$;

    \item\label{the:Nielsen_Realization_Problem:eub_hyperbolic} The group $\Gamma$ is
      hyperbolic;

    \item\label{the:Nielsen_Realization_Problem_eub_CATzero} The group $\Gamma$ acts
      cocompactly properly and isometrically on a proper $\CATzero$-space;

    \end{enumerate}
   
  \item\label{the:Nielsen_Realization_Problem_vcd} There is a finite $CW$-complex model
    of dimension $d$ for $B\pi$.

  \end{enumerate}

  Then there exists a finite $\Gamma$-$CW$-model $X$ for $\eub{\Gamma}$ of dimension $d$
  such that its singular $\Gamma$-subspace $X^{> 1}$ is $\coprod_{F\in\calm} \Gamma/F$.
\end{theorem}

The proof of
Theorem~\ref{the:Models_for_the_classifying_space_for_proper_Gamma-actions_intro} will be
given in Section~\ref{sec:Some_results_about_the_classifying_space_for_proper_actions}.

\begin{remark}\label{Theorem_models_and_necessary_conditions}
The conditions~\eqref{the:Nielsen_Realization_Problem:(M)_and)NM)_intro},~%
\eqref{the:Nielsen_Realization_Problem:eub_direct},
and~\eqref{the:Nielsen_Realization_Problem_vcd}  appearing in
Theorem~\ref{the:Models_for_the_classifying_space_for_proper_Gamma-actions_intro} are
necessary.  This follows for~\eqref{the:Nielsen_Realization_Problem:(M)_and)NM)_intro},
since the argument appearing in the proof of
Lemma~\ref{lem:necessary_condition_for_the_existence_of_M_intro:(M)_and_(NM)_intro}
carries over directly, and is obvious
for~\eqref{the:Nielsen_Realization_Problem:eub_direct} and
and~\eqref{the:Nielsen_Realization_Problem_vcd}.
\end{remark}

\begin{remark}\label{rem:fin_dom_intro} Suppose that
  assumptions~\eqref{the:Nielsen_Realization_Problem:(M)_and)NM)_intro}
  and~\eqref{the:Nielsen_Realization_Problem_vcd}
  appearing in Theorem~\ref{the:Models_for_the_classifying_space_for_proper_Gamma-actions_intro}
  hold and $\calm$ is
  finite.  Then we do get a finitely dominated $\Gamma$-$CW$-model for $\eub{\Gamma}$
  by the following argument.  We obtain  some finite-dimensional $\Gamma$-$CW$-model for
  $\eub{\Gamma}$ from~\cite[Theorem~2.4]{Lueck(2000a)}. We get a $\Gamma$-$CW$-model of
  finite type for $\eub{\Gamma}$ from~\cite[Theorem~4.2]{Lueck(2000a)}, since $W_{\Gamma}H$
  is finite for every non-trivial finite subgroup of $\Gamma$ and there are only finitely
  many conjugacy classes of finite subgroups in $\Gamma$ by condition (M) and (NM),
  see~\cite[Lemma~2.1]{Davis-Lueck(2022_manifold_models)}, and $B\Gamma$ has a 
  $CW$-model of finite type by~\cite[Lemma~7.2]{Lueck(1997a)} applied to the fibration
  $B\pi \to B\Gamma \to BG$. Hence we get a finitely dominated $\Gamma$-$CW$-model for
  $\eub{\Gamma}$ by~\cite[Proposition~14.9 on page~282]{Lueck(1989)}.

  If we want to turn this model into a finite $\Gamma$-$CW$-model, we have to compute its
  equivariant finiteness obstruction, see~\cite[Section~11]{Lueck(1989)}.  Since there
  exists a $\Gamma$-$CW$-model for $\eub{\Gamma}$ with finite $\eub{\Gamma}^{> 1}$,
  see Proposition~\ref{pro:constructing_underline(E)Gamma)}, only
  the top component of the equivariant finiteness obstruction associated to the trivial
  subgroup may be non-trivial. It takes values in $\widetilde{K}_0(\IZ\Gamma)$.  Hence
  there exists a finite $\Gamma$-$CW$-model for $\eub{\Gamma}$ if
  assumptions~\eqref{the:Nielsen_Realization_Problem:(M)_and)NM)_intro}
  and~\eqref{the:Nielsen_Realization_Problem_vcd} hold, $\calm$ is finite,
  and $\widetilde{K}_0(\IZ\Gamma)$ vanishes. Suppose that $\pi$ satisfies the Full
  Farrell-Jones Conjecture. Then the canoncial map
  $\bigoplus_{F \in \calm} \widetilde{K}_0(\IZ F) \xrightarrow{\cong}
  \widetilde{K}_0(\IZ\Gamma)$ is an isomorphism,
  see~\cite[Theorem~5.1~(d)]{Davis-Lueck(2003)}
  or~\cite[Theorem~5.1]{Davis-Lueck(2022_manifold_models)}.  Hence
  $\widetilde{K}_0(\IZ\Gamma)$ vanishes, if and only if $\widetilde{K}_0(\IZ F)$ vanishes
  for all $F \in \calm$.  If $F$ is cyclic of order $n$, then $\widetilde{K}_0(\IZ F)$
  vanishes, if and only if $n \le 11$ or $n \in \{13,14,17,19\}$.
\end{remark}

\begin{remark}[The torsionfree case]\label{rem:torsion_free_case_Brown_and_Poincare_intro}
  Now suppose that $\Gamma$ is torsionfree, there is a finite model for $B\pi$, and that
  $\pi$ satisfies the Full Farrell-Jones Conjecture in the sense
  of~\cite[Section~12.5]{Lueck(2022book)}.  Remark~\ref{rem:fin_dom_intro} implies that
  there is a finitely dominated $CW$-complex model for $B\Gamma$. Since
  $\widetilde{K}_0(\IZ\Gamma)$ vanishes by the Full Farrell-Jones Conjecture, there is a
  finite $CW$-model for $B\Gamma$. Moreover, we conclude
  from~\cite[Theorem~H]{Klein-Qin_Su(2019)} that there is a finite Poincar\'e complex
  homotopy equivalent to $B\Gamma$, if and only if there is a finite Poincar\'e complex
  homotopy equivalent to $B\pi$.
\end{remark}

In view of Remark~\ref{rem:torsion_free_case_Brown_and_Poincare_intro} we we will often
and tacitly assume throughout the remainder of this paper that $\Gamma$ is not
torsionfree.


\subsection{Poincar\'e models}\label{subsec:Poincare_models_intro}

Recall that we want to construct the desired $\Gamma$-manifold $M$ described
in~\eqref{Gamma_pushout_for_M_intro} such that $M$ is $\Gamma$-homotopy equivalent to
$\eub{\Gamma}$, or, equivalently such that $M$ is contractible.  For this purpose it
suffices to find a free proper cocompact $d$-dimensional $\Gamma$-manifold $X$ with
boundary $\partial X$ such that the space $M$ defined in~\eqref{Gamma_pushout_for_M_intro}
is contractible.  If we divide out the $\Gamma$-action, we get a compact manifold
$X/\Gamma$ with boundary $\partial X/ \Gamma$. Hence it suffices to construct a compact
$d$-dimensional manifold $Y$ with fundamental group $\Gamma$ and boundary
$\partial X/\Gamma$ such that $Y \cup_{\partial X/\Gamma} C(\partial X)/\Gamma$ is
aspherical, since then we can define $X$ to be the universal covering of $Y$. Recall that
any compact manifold with boundary is a finite Poincar\'e pair. Hence our first task is to
construct a finite $d$-dimensional Poincar\'e pair $(Y,\partial X/\Gamma)$ such that the
fundamental group of $Y$ is $\Gamma$ and $Y \cup_{\partial X/\Gamma} C(\partial X)/\Gamma$
is aspherical.

The next result will be a direct consequence of
Theorem~\ref{the:Models_for_the_classifying_space_for_proper_Gamma-actions_intro} and 
Theorem~\ref{the:existence_Poincare_slice_models}.

  \begin{theorem}[Poincar\'e models]\label{the:Poincare_models_intro}
    Suppose that the following conditions are satisfied:

    \begin{itemize}

    \item The natural number $d$ is even and satisfies $d \ge 4$;

     \item The group $\Gamma$ satisfies conditions (M), (NM), and (H), see
       Notation~\ref{not:(M)_and_(NM)_intro};

      \item The homomorphism $w \colon \Gamma \to \{\pm 1\}$ of Notation~\ref{not:w} has the
  property that $w|_F$ is trivial for every $F \in \calm$;

    \item One of the following assertions holds:

      \begin{itemize}

      \item There exists a finite $\Gamma$-$CW$-model for $\eub{\Gamma}$;

      \item The group $\pi$ is hyperbolic;

      \item The group $\Gamma$ acts cocompactly, properly, and isometrically on a proper
        $\CATzero$-space;

      \end{itemize}
      
    \item There is a finite $d$-dimensional Poincar\'e complex, which is homotopy
      equivalent to $B\pi$;

    \item There exists an oriented free $d$-dimensional slice system $\cals$,
      see Definition~\ref{def:free_d-dimensional_slide_system},
      satisfying condition (S), see Definition~\ref{def:condition_(S)}.
  
    \end{itemize}

    Put $\partial X = \coprod_{F \in \calm} \Gamma \times_F S_F$ and
    $C(\partial X) = \coprod_{F \in \calm} \Gamma \times_F D_F$ for $D_F$ the cone over
    $S_F$.

    Then there exists a finite free $\Gamma$-$CW$-pair $(X,\partial X)$ such that
    $X \cup_{\partial X} C(\partial X)$ is a model for $\eub{\Gamma}$ and
    $(X/\Gamma,\partial X/\Gamma)$ is a finite $d$-dimensional Poincar\'e pair.
  \end{theorem}

  We will also explain that for every $F \in \calm$ there is only one choice for $S_F$ up
  to $F$-homotopy and how we can determine this choice from $\Gamma$, see
  Section~\ref{subsec:Invariants_associated_to_slice_models}, and prove that the free
  $\Gamma$-$CW$-pair $(X,\partial X)$ is unique up to $\Gamma$-homotopy equivalence, see
  Theorem~\ref{the:Uniqueness_of_Poincare_slice_models}. Uniqueness up to
  simple $\Gamma$-homotopy equivalence is investigated in
  Section~\ref{sec:Simple_homotopy_classification_of_slice_models}.


  \subsection{Surgery Theory}\label{subsec:Surgery_Theory_intro}

  The second step is to promote the finite $d$-dimensional Poincar\'e pair
  $(Y,\partial X/\Gamma)$ of Subsection~\ref{subsec:Poincare_models_intro} up to homotopy
  to a compact manifold with boundary. This will be presented in a different paper, namely
  in~\cite{Davis-Lueck(2022_manifold_models)}.  In that paper surgery theory will come in, whereas in this
  paper our main techniques stem from algebraic topology and equivariant homotopy
  theory. In general further surgery obstructions, notably splitting obstructions taking
  values in UNil-groups, will appear.

  As an illustration we mention that our ultimate result in~\cite{Davis-Lueck(2022_manifold_models)}
  will imply the following result:

  \begin{theorem}[Manifold models]\label{the:special_case_of_ultiimate_theorem_intro}
    Suppose that the following conditions are satisfied:

    \begin{itemize}

    \item The natural number $d$ is even and satisfies $d \ge 6$;

     \item The group $\Gamma$ satisfies conditions (M), (NM), and (H), see
      Notation~\ref{not:(M)_and_(NM)_intro};
      
    \item The group $\pi$ is hyperbolic;

    \item There exists a closed $d$-dimensional manifold, which is homotopy equivalent to
      $B\pi$;

    \item The group $F$ is cyclic of odd order for every $F \in \calm$.

    \end{itemize}

    Then there exists a proper cocompact $d$-dimensional topological $\Gamma$-manifold $M$ of the
    shape described in~\eqref{Gamma_pushout_for_M_intro}, where each $S_x$ is homeomorphic
    to $S^{d-1}$, such that $M$ is $\Gamma$-homotopy  equivalent to $\eub{\Gamma}$.  
  \end{theorem}

  In the paper~\cite{Davis-Lueck(2022_manifold_models)} we will also discuss the uniqueness problem
  extending the results of~\cite{Connolly-Davis-Khan(2015)}.


  \subsection{Nielsen Realization}\label{subsec:Nielsen_Realization}

  One motivation for searching for manifold models for $\eub{\Gamma}$ comes from the
  following classical

\begin{question}[Nielsen Realization Problem]\label{que:Nielsen_realization_problem}
  Let $M$ be an aspherical closed manifold with fundamental group $\pi$. Let
  $j \colon G \to \Out(\pi)$ be an embedding of a finite group $G$ into the outer
  automorphism group of $\pi$.

  Is there an effective $G$-action $\rho \colon G \to \Aut(M)$ of $G$ on $M$ such that the
  composite of $\rho$ with the canonical map $\nu \colon \Aut(M) \to \Out(\pi)$ is $j$?
\end{question}

In the smooth category it is easy to give examples using exotic spheres, where the answer
is negative to Question~\ref{que:Nielsen_realization_problem}, see for
instance~\cite[page~22]{Block-Weinberger(2008)}.  Therefore we will only consider the
topological category.

\begin{remark}[Original version]\label{rem:original_version}
  Question~\ref{que:Nielsen_realization_problem} was originally formulated for closed
  orientable hyperbolic surfaces of genus $\ge 1$ by Nielsen and was proved by
  Kerckhoff~\cite{Kerckhoff(1983)}.  Subsequently, Tromba~\cite{Tromba(1996Nielsen)},
  Gabai~\cite{Gabai(1992Fuchsian)}, and Wolpert~\cite{Wolpert(1987)} gave new proofs.
\end{remark}
 
 \begin{remark}[Counterexamples]\label{rem:counterexamples}
   There do exist examples, where the answer to
   Question~\ref{que:Nielsen_realization_problem} is negative.  A necessary condition is
   that there exists an extension $1 \to \pi \to \Gamma \to G \to 1$ such that the
   conjugation action of $G$ on $\pi$ is the given map $j$, see for
   instance~\cite[Theorem~8.1 on page~138]{Lueck(1989)}.  This condition is automatically
   satisfied, if $\pi$ is centerless, see~\cite[Corollary~6.8 in Chapter~IV
   page~106]{Brown(1982)}, but not in general, see~\cite{Raymond-Scott(1977)}. Even for
   centerless $\pi$ there are examples, where the answer is negative for the
   Question~\ref{que:Nielsen_realization_problem}, see~\cite[Theorem~1.5 and
   Theorem~1.6]{Block-Weinberger(2008)}.
 \end{remark}

 Some positive results about Question~\ref{que:Nielsen_realization_problem} have been
 obtained by Farrell-Jones, see for instance~\cite[page~282ff]{Farrell(2002)}.
  
  \begin{remark}[Nielsen Realization for torsionfree $\Gamma$]\label{rem:torsionfree_Gamma_Nielsen_intro}
    Suppose that the group $\Gamma$ appearing in the
    extension~\eqref{group_extension_intro}
    is torsionfree, $\dim(M) \not= 3,4$, and $\pi$
    satisfies the Full Farrell-Jones Conjecture in the sense
    of~\cite[Section~12.5]{Lueck(2022book)}.  This means that the $K$-and $L$-theoretic
    Farrell-Jones Conjecture with coefficients in additive $\pi$-categories and with
    finite wreath products is satisfied for $\pi$. Examples for such $\pi$ are hyperbolic
    groups or CAT(0)-groups.  Then also $\Gamma$ satisfies the Full Farrell-Jones
    Conjecture. Supppose that $B\pi$ is homotopy equivalent to a closed manifold $M$.
    Since $\Gamma$ is   torsionfree,  $B\Gamma$ can be  realized as a finite Poincar\'e complex
    by Remark~\ref{rem:torsion_free_case_Brown_and_Poincare_intro}.  We conclude
    from~\cite[Theorem~1.2]{Bartels-Lueck-Weinberger(2010)}, that $B\Gamma$ has the
    homotopy type of a closed ANR-homology manifold $N$ with fundamental group
    $\Gamma$. We have the finite covering $\overline{N} \to N$ associated to
    $1 \to \pi \to \Gamma \to G \to 1$. Note that $\overline{N}$ comes with a
    $G$-action. It suffices to show that $\overline{N}$ is homeomorphic to $M$. Since
    $\overline{N}$ and $M$ have the same Quinn obstruction
    by~\cite[Theorem~5.28]{Ferry-Lueck-Weinberger(2019)} and $M$ is manifold, the Quinn
    obstruction of $\overline{N}$ is trivial and hence $\overline{N}$ is an aspherical
    closed manifold, whose fundamental group is identified with $\pi = \pi_1(M)$.  Thanks
    to the Borel Conjecture, which follows from the Farrell-Jones Conjecture,
    $\overline{N}$ is homeomorphic to $M$.
  \end{remark}

  To the authors' knowledge, there is no counterexample to the following 

  \begin{conjecture}[Nielsen Realization for aspherical closed manifolds with hyperbolic
    fundamental group]%
\label{con:Nielsen_Realization_for_aspherical_closed_manifolds_with_hyperbolic_fundamental_group}
    Let $M$ be an aspherical closed manifold with hyperbolic fundamental group $\pi$.  Let
    $j \colon G \to \Out(\pi)$ be an embedding of a finite group $G$ into the outer
    automorphism group of $\pi$.

    Then there is an effective topological $G$-action $\rho \colon G \to \Aut(M)$ of $G$
    on $M$ such that the composite of $\rho$ with the canonical map
    $\nu \colon \Aut(M) \to \Out(\pi)$ is $j$.
  \end{conjecture}

  \begin{remark}\label{rem:out(F)_finite} If in
    Conjecture~\ref{con:Nielsen_Realization_for_aspherical_closed_manifolds_with_hyperbolic_fundamental_group}
    the dimension of $M$ is greater or equal to $3$, then $\Out(\pi)$ is known to be finite,
    see~\cite[\S~5, 5.4.A]{Gromov(1987)}, and  hence one can take and it suffices to consider $G = \Out(\pi)$.
  \end{remark}
  
\begin{remark}\label{rem:connection:to:Nielsen}
  The proper cocompact $\Gamma$-manifold $M$ appearing in
  Problem~\ref{prob:main_problem_intro}
  yields a solution to the Nielsen Realization Problem for
  $1 \to \pi \to \Gamma \to G \to 1$, provided that the Borel Conjecture holds for $\pi$
  and $M$ is contractible.  Namely, $M/\pi$ is aspherical with fundamental group $\pi$ and
  hence any closed aspherical manifold with $\pi$ as fundamental group admits a
  homeomorphism to $M/\pi$ inducing the identity on the fundamental groups and $G$ acts on
  $M/\pi$ in the obvious way.  Hence the Nielsen Realization problem has a positive answer,
  if the conditions appearing in Theorem~\ref{the:special_case_of_ultiimate_theorem_intro}
  are satisfied.
  
\end{remark}


\subsection{Acknowledgments}\label{subsec:Acknowledgements}

The paper is funded by the Deutsche Forschungsgemeinschaft (DFG, German Research
Foundation) under Germany's Excellence Strategy \--- GZ 2047/1, Projekt-ID 390685813,
Hausdorff Center for Mathematics at Bonn.

The author is grateful especially to  James Davis and also to Christian Kremer  and Kevin Li
for fruitful discussions and useful comments. James Davis has informed the author
that he has independent solutions to Brown's Problem and to the Nielsen Realization Problem in the case
that all finite subgroups of $\Gamma$ are of order two,  which have not appeared as preprint so far.


The paper is organized as follows:

\tableofcontents


\typeout{------ Section 2: Some results about the classifying space for proper actions
  -----------------------}


\section{Some results about the classifying space for proper actions}%
\label{sec:Some_results_about_the_classifying_space_for_proper_actions}

This section is devoted to the proof of
Theorem~\ref{the:Models_for_the_classifying_space_for_proper_Gamma-actions_intro}

\begin{proposition}\label{pro:constructing_underline(E)Gamma)}
  Suppose that $\Gamma$ satisfies (M) and (NM).  Let $\calm$ be a complete system of
  representatives of the conjugacy classes of maximal finite subgroups
  $F\subseteq \Gamma$.  Consider the cellular $\Gamma$-pushout
  \[
    \xymatrix{\partial E\Gamma =\coprod_{F\in\calm} \Gamma \times_{F}EF
      \ar[d]_{\coprod_{F\in\calm}p_F} \ar[r]^-i & E\Gamma \ar[d] \\
      \partial \eub{\Gamma} =\coprod_{F\in\calm} \Gamma/F \ar[r] & Z}
  \]
  where the map $p_F$ comes from the projection $EF \to \pt$, and $i$ is an inclusion of
  $\Gamma$-$CW$-complexes.

  Then $Z$ is a model for $\underline{E}\Gamma$.
\end{proposition}

\begin{proof} This follows from~\cite[Corollary~2.11]{Lueck-Weiermann(2012)}.
\end{proof}

\begin{proposition}\label{pro:underline(E)Gamma_for_hyperbolic_Gamma}
  Let $\Gamma$ be a hyperbolic group.  Then there is a finite $\Gamma$-$CW$-model for the
  classifying space for proper actions $\underline{E}\Gamma$.
\end{proposition}
\begin{proof}
  See~\cite{Meintrup-Schick(2002)}.
\end{proof}

\begin{proposition}\label{pro:CAT(0)_and_eub_Ontaneda}
  Suppose that $\Gamma$ acts cocompactly, properly, and isometrically on a proper
  $\CATzero$-space $X$ Then there is a finite $\Gamma$-$CW$-model for $\eub{\Gamma}$.
\end{proposition}
\begin{proof}
  By~\cite[Corollary II.2.8 on page 179]{Bridson-Haefliger(1999)} the $H$-fixed point set
  $X^H$ of $X$ is a non-empty convex subset of $X$ and hence contractible for any finite
  subgroup $H \subset G$.  Since the action is proper, all isotropy groups are finite.  We
  conclude from~\cite[Proposition~A]{Ontaneda(2005)} that $X$ is $\Gamma$-homotopy
  equivalent to a finite $\Gamma$-$CW$-complex $Y$, since for a proper $\Gamma$-action
  each $\Gamma$-orbit is discrete. Hence $Y$ is a finite $\Gamma$-$CW$-model for
  $\eub{\Gamma}$.
\end{proof}

\begin{lemma}\label{lem:finite_d-dimensional_models}
  Let $\Gamma$ be a group and $d$ be a natural number satisfying $d \ge 3$.  Let $X$ be a
  $\Gamma$-$CW$-complex of finite type.  Suppose that $X$ is $\Gamma$-homotopy equivalent
  to some (not necessarily finite) $d$-dimensional $\Gamma$-$CW$-complex and to some
  finite $\Gamma$-$CW$-complex (of arbitrary dimension).  Then there exists a
  $\Gamma$-$CW$-complex $\widehat{X}$ satisfying
  \begin{enumerate}
    \item $\widehat{X}$ is finite and $d$-dimensional;
   \item  The $(d-2)$-skeleton of $X$ and the
    $(d-2)$-skeleton of $\widehat{X}$ agree;
  \item If $\widehat{X}$ has an equivariant  cell of the type $\Gamma/H \times D^{d-1}$, then $X$ contains
    an equivariant  cell of the type $\Gamma/H \times D^k$ for some $k \in \{d-1,d\}$.
    If $\widehat{X}$ has an equivariant  cell of the type $\Gamma/H \times D^d$, then $X$ contains
    an equivariant  cell of the type $\Gamma/H \times D^d$. In particular $Y^{>1} = X^{>1}$,
    if $X^{>1}$ is contained in the $(d-2)$-skeleton of $X$;
  \item The $\Gamma$-$CW$-complexes $X$ and $\widehat{X}$ are $\Gamma$-homotopy equivalent.
\end{enumerate}
\end{lemma}
\begin{proof}
  We use the machinery and notation developed in~\cite{Lueck(1989)}. During this proof we
  abbreviate $\calc : = \Pi(\Gamma,X)$, where $\Pi(\Gamma,X)$ is the fundamental category,
  see~\cite[Definition~8.15 on page~144]{Lueck(1989)}.  Let $C_*^c(X)$ be the cellular
  $\IZ\calc$-chain complex. It is a finitely generated free $\IZ\calc$-chain complex, as
  $X$ is of finite type.  Since $X$ is homotopy equivalent to a $d$-dimensional
  $\Gamma$-$CW$-complex, $C_*^c(X)$ is $\IZ\calc$-chain homotopy equivalent to a free
  $d$-dimensional $\IZ\calc$-chain complex. We conclude from~\cite[Proposition~11.10 on
  page~221]{Lueck(1989)} that we can find a $d$-dimensional finitely generated projective
  $\IZ\calc$-subchain complex $D_*$ of $C_*^c(X)$ such that $D_*$ and $C_*^c(X)$ agree in
  dimensions $\le (d-1)$, $D_d$ is a direct summand of $C_*^c(X)$, and the inclusion
  $j_* \colon D_* \to C_*^c(X)$ is a $\IZ\calc$-chain homotopy equivalence.  Hence the equivariant
  finiteness obstruction $\widetilde{o}^G(X) \in \widetilde{K}_0(\IZ\calc)$ is the
  class of the finitely generated projective $\IZ\calc$-module $D_d$.  Since $X$
  is by assumption $\Gamma$-homotopy equivalent to a finite $\Gamma$-$CW$-complex, the
  equivariant finiteness obstruction
  $\widetilde{o}^G(X) \in \widetilde{K}_0(\IZ\calc)$ vanishes,
  see~\cite[Definition~14.4 and Theorem~14.6 on page~278]{Lueck(1989)}.  This implies that
  $D_d$ is stably free.  We will need a stronger statement about $D_d$ stated and proved
  below.

  There is a $\IZ\calc$-isomorphism
  \begin{eqnarray*}
    C_d^c(X) & = &  \bigoplus_{i= 1}^m \IZ\mor_{\calc}(?,x_i)^{a_i}
  \end{eqnarray*}
  for integers $m \ge 0$ and $a_i \ge 1$ and for a finite set
  $\{x_i \colon \Gamma/H_i \to X\mid i = 1,2, \ldots, m\}$ of objects in $\calc$, whose
  elements are mutually not isomorphic in $\calc$.  Note that $X$ must contain at least
  one equivariant cell of type $\Gamma/H_i \times D^d$ for each
  $i \in \{1,2, \ldots ,m\}$.  Let $S_x$ be the splitting functor and $E_x$ be the
  extension functor associated to an object $x \colon \Gamma/H \to X$, see~\cite[(9.26)
  and (9.28) on page~170]{Lueck(1989)}.  If $S_{x}(C_d^c(X)) \not= 0$ holds, then $x$ is
  isomorphic to precisely one  of the $x_i$-s and in particular $H$ is conjugated to 
  $H_i$.  Since $D_d$ is a direct summand of $C_d(X)$ and $S_x$ is compatible with
  direct sums, the analogous statement holds for $D_d$. Hence $D_d$ is
  $\IZ\calc$-isomorphic to $\bigoplus_{i = 1}^m E_{x_i} \circ S_{x_i}(D_d)$
  by~\cite[Corollary~9.40 page~176]{Lueck(1989)}. We conclude
  from~\cite[Theorem~10.34  page~196]{Lueck(1989)} and the
  vanishing of the equivariant finiteness obstruction
  $\widetilde{o}^G(X) \in \widetilde{K}_0(\IZ\calc)$ that for every
  $i \in \{1,2, \ldots ,m\}$ the class of the finitely generated projective
  $\IZ[\aut_{\calc}(x_)]$-module $S_{x_i}D_d$ in $\widetilde{K}_0(\IZ[\aut_{\calc}(x)])$
  vanishes.  Hence we can find for every $i \in \{1,2, \ldots, m\}$ integers $k_i,l_i$
  with $0 \le k_i \le l_i$ such that $S_{x_i}D_d \oplus \IZ[\aut_{\calc}(x_i)]^{k_i}$ is
  $\IZ[\aut_{\calc}(x_i)]$ isomorphic to $\IZ[\aut_{\calc}(x_i)]^{l_i}$.  This implies the
  existence of an isomorphism of $\IZ\calc$-modules
    \[
      D_d \oplus \bigoplus_{i = 1}^m \IZ\mor_{\calc}(?,x_i)^{k_i}
      \cong  \bigoplus_{i = 1}^n \IZ\mor_{\calc}(?,x_i)^{l_i}.
    \]
    Hence we can add to $D_*$ an elementary $\IZ\calc$-chain complex $E_*$ concentrated in
    dimensions $d$ and $(d-1)$ and associated to the finitely generated free
    $\IZ\calc$-module $\bigoplus_{i = 1}^n \IZ\mor_{\calc}(?,x_i)^{k_i} $.  Denote the
    result by $D'_* := D_* \oplus E_*$. We obtain a $\IZ\calc$-chain homotopy equivalence
    $j_* \oplus \id_{E_*} \colon D_* \oplus E_* \to C_*(X) \oplus E_*$.  By attaching
    equivariant cells of the shape $\Gamma/H_i \times D^{d-1}$ and $\Gamma/H_i \times D^d$
    for $i \in \{1,2, \ldots m\}$ to $X$, we get a finite $\Gamma$-$CW$-complex $X'$ such
    that the inclusion $X \to X'$ is a $\Gamma$-homotopy equivalence and there is an
    identification of $\IZ\calc$-chain complexes $C_*(X') = C_*(X) \oplus E_* $. Then
    $D_*'$ is a $\IZ\calc$-subchain complex of $X'$ such that $C_*(X')$ and $D_*'$ agree
    in dimensions $\le (d-1)$, the inclusion $j_* \colon D_*' \to C_*(X')$ is a
    $\IZ\calc$-chain homotopy equivalence, and
    $D'_d = \bigoplus_{i =1}^n \IZ\mor_{\calc}(?,x_i)^{l_i}$.  Note that $X$ and $X'$ have
    the same $(d-2)$-skeleton. Moreover, if $X'$ has an equivariant cell of the type
    $\Gamma/H \times D^{d-1}$, then $X$ contains an equivariant cell of the type
    $\Gamma/H \times D^k$ for some $k \in \{d-1,d\}$ and if $X'$ has an equivariant cell
    of the type $\Gamma/H \times D^d$, then $X$ contains an equivariant cell of the type
    $\Gamma/H \times D^d$.  Now apply the Equivariant Realization Theorem,
    see~\cite[Theorem~13.19 on page~268]{Lueck(1989)}, in the case $r = (d-1)$,
    $A = B= \emptyset$, $Z = X'_{d-1}$, $Y = X'$, and $h$ the inclusion $X_{d-1} \to X$,
    and $C = D_*'$.  The resulting finite $\Gamma$-$CW$-complex $\widehat{X}$ has the
    desired properties, since by construction its $(d-1)$-skeleton is $X'_{d-1}$ and we have
    $C_*^c(\widehat{X}) = D_*'$.  
\end{proof}

\begin{lemma}\label{lem:Brown_small_singular_set}
  Let $d$ be a natural number with $d \ge 3$. Let $\Gamma$ be a group with a torsionfree
  subgroup $\pi$ of finite index.  Suppose that there is a $\Gamma$-$CW$-model
  $\underline{E}\Gamma$ such that $\dim(\underline{E}\Gamma^{>1})\le (d-1)$ holds for the
  $\Gamma$-$CW$-subcomplex $\underline{E}\Gamma^{>1}$ of $\underline{E}\Gamma$ consisting
  of points with non-trivial isotropy group.  Then the following assertions are
  equivalent:
  \begin{enumerate}
  \item\label{lem:Brown_small_singular_set:Gamma} There is a model $\underline{E}\Gamma$
    with $\dim(\underline{E}\Gamma) \le d$;

  \item\label{lem:Brown_small_singular_set:pi} There is a model $E\pi$ with
    $\dim(E\pi) \le d$.
    
  \end{enumerate}

\end{lemma}
\begin{proof} Let $i \colon \pi \to \Gamma$ be the inclusion. Since the restriction
  $i^*\underline{E}\Gamma$ of $\underline{E}\Gamma$ with $i$ is a model for $E\pi$, the
  implication~\eqref{lem:Brown_small_singular_set:Gamma}
  $\implies$~\eqref{lem:Brown_small_singular_set:pi} is true. The hard part is to show the
  implication~\eqref{lem:Brown_small_singular_set:pi}
  $\implies$~\eqref{lem:Brown_small_singular_set:Gamma} what we do next.

  We show by induction for
  $m = d \cdot [\Gamma : \pi], d \cdot [\Gamma : \pi] -1, d \cdot [\Gamma : \pi] -2,
  \ldots, d$ that there  is a model for $\underline{E}\Gamma$ with
  $\dim(\underline{E}\Gamma) \le m$.

  The induction beginning is done as follows. Let $ i \colon \pi \to \Gamma$ be the
  inclusion.  Then the coinduction $i_!E\pi$ is a $\Gamma$-$CW$-model for
  $\underline{E}\Gamma$ of dimension $\dim(E\pi) \cdot [\Gamma : \pi]$,
  see~\cite[Theorem~2.4]{Lueck(2000a)}.

  The induction step from $m+1$ to $m$ for $d \le m \le [\Gamma : \pi] -1$ is described
  next.  Let $X$ be a $\Gamma$-$CW$-complex such that all isotopy groups are finite, $X^H$
  is contractible for every finite subgroup $H \subseteq \Gamma$, and for every
  $\Gamma$-cell $\Gamma/H \times D^k$ with $H \not= \{1\}$ we have $k \le (d-1)$. Note
  that $X$ is a model for $\underline{E}\Gamma$. We have to show that $X$ is
  $\Gamma$-homotopy equivalent to a $\Gamma$-$CW$-complex of dimension $\le m$, provided
  that $X$ is $\Gamma$-homotopy equivalent to a $\Gamma$-$CW$-complex of dimension
  $\le (m+1)$ and $m \ge d$.

  Because of  the Equivariant Realization Theorem, see~\cite[Theorem 13.19 on
  page~268]{Lueck(1989)}, and~\cite[Proposition~11.10 on page~221]{Lueck(1989)},
  it suffices to show that the cellular
  $\IZ\OrGF{\Gamma}{\calfin}$-chain complex $C_*^c(X)$ is
  $\IZ\OrGF{\Gamma}{\calfin}$-chain homotopy equivalent to an $m$-dimensional
  $\IZ\OrGF{\Gamma}{\calfin}$-chain complex.  Note that we can confine ourselves in this
  special case to the $\calfin$-restricted orbit category $\OrGF{\Gamma}{\calfin}$, since all
  isotropy groups are finite and $X^H$ is non-empty and simply connected for every finite
  subgroup $H \subseteq \Gamma$. Namely, the latter implies that the fundamental category
  $\Pi(\Gamma,X)$ appearing in~\cite[Definition~8.15 on page~144]{Lueck(1989)} reduces to
  $\OrGF{\Gamma}{\calfin}$ and the cellular $\IZ\Pi(\Gamma,X)$-chain complex of $X$ appearing
  in~\cite[Definition~8.37 on page~152]{Lueck(1989)} reduces to the corresponding cellular
  chain complex $C_*^c(X)$ over $\OrGF{\Gamma}{\calfin}$, which sends an object $\Gamma/H$ to
  the cellular chain complex of $X^H$. Recall that $\OrGF{\Gamma}{\calfin}$ has as objects
  homogeneous $\Gamma$-spaces $\Gamma/H$ with $|H| < \infty$ as objects and $\Gamma$-maps
  between them as
  morphisms.

  Define the $\IZ\OrGF{G}{\calfin}$-module $B_m$ to be the image of the
  $(m+1)$th-differential $c_{m+1} \colon C^c_{m+1}(X) \to C_m^c(X)$.  Because
  of~\cite[Proposition~11.10 on page~221]{Lueck(1989)} it suffices to show that
  $B_m \subseteq C^c_m(X)$ is a direct summand in $C_m^c(X)$. The following argument shows
  that for this purpose it suffices to show that
  $H^{m+1}_{\IZ\OrGF{\Gamma}{\calfin}}(C_*^c(X),B_m)$ is trivial. Namely, from
  $H^{m+1}_{\IZ\OrGF{\Gamma}{\calfin}}(C_*^c(X);B_m) = 0$ we get the exact sequence
  \begin{multline*}
    \hom_{\IZ\OrGF{\Gamma}{\calfin}}(C^c_m(X),B_m) \xrightarrow{c_{m+1}^*}
    \hom_{\IZ\OrGF{\Gamma}{\calfin}}(C^c_{m+1}(X),B_m)
    \\
    \xrightarrow{c_{m+2}^*} \hom_{\IZ\OrGF{\Gamma}{\calfin}}(C^c_{m+2}(X),B_m).
  \end{multline*}
  The element $c_{m+1} \colon C^c_{m+1}(X) \to B_m$ in
  $\hom_{\IZ\OrGF{\Gamma}{\calfin}}(C^c_{m+1}(X),B_m)$ is sent under the right arrow to
  $c_{m+1} \circ c_{m+2} = 0$ and hence has a preimage $r \colon C^c_m(X) \to B_m$ under
  the first arrow. This implies $r \circ c_{m+1} = c_{m+1}$ and hence
  $r|_{B_m} = \id_{B_m}$.

  We have the short exact sequence
  $0 \to Z_{m+1} \to C^c_{m+1}(X) \xrightarrow{c_{m+1}} B_m \to 0$ of
  $\IZ\OrGF{\Gamma}{\calfin}$-modules, where $Z_{m+1}$ is the kernel of $c_{m+1}$. In the
  sequel we abbreviate $F = C^c_{m+1}(X)$. It yields a long exact sequence
  \begin{multline*}
    \cdots \to H^{m+1}_{\IZ\OrGF{\Gamma}{\calfin}}(C_*^c(X);Z_{m+1}) \to
    H^{m+1}_{\IZ\OrGF{\Gamma}{\calfin}}(C_*^c(X);F)
    \\
    \to H^{m+1}_{\IZ\OrGF{\Gamma}{\calfin}}(C_*^c(X);B_m) \to
    H^{m+2}_{\IZ\OrGF{\Gamma}{\calfin}}(C_*^c(X);Z_{m+1}) \to \cdots.
  \end{multline*}
  Since by the induction hypothesis $X$ is $\Gamma$-homotopy equivalent to an
  $(m+1)$-dimensional $\Gamma$-$CW$-complex,
  $H^{m+2}_{\IZ\OrGF{\Gamma}{\calfin}}(C^c_*(X);Z_{m+1})$ vanishes and hence the map
  $H^{m+1}_{\IZ\OrGF{\Gamma}{\calfin}}(C_*^c(X);F) \to
  H^{m+1}_{\IZ\OrGF{\Gamma}{\calfin}}(C_*^c(X);B_m)$ is surjective.  Therefore it suffices
  to show that $H^{m+1}_{\IZ\OrGF{\Gamma}{\calfin}}(C_*^c(X);F)$ vanishes.

  We have the short exact sequence of $\IZ\OrGF{\Gamma}{\calfin}$-chain complexes
  \begin{equation}
    0 \to C_*^c(\underline{E}\Gamma^{> 1}) \to C_*^c(\underline{E}\Gamma)
    \to C_*^c(\underline{E}\Gamma,\underline{E}\Gamma^{> 1})\to 0.
    \label{def_D_ast}
  \end{equation}
  It yields a long cohomology sequence
  \begin{multline*}
    \cdots \to H^m_{\IZ\OrGF{\Gamma}{\calfin}}(C_*^c(\underline{E}\Gamma^{> 1});F) \to\
     H^{m+1}_{\IZ\OrGF{\Gamma}{\calfin}}(C_*^c(\underline{E}\Gamma,\underline{E}\Gamma^{>1});F)
    \\
    \to  H^{m+1}_{\IZ\OrGF{\Gamma}{\calfin}}(C_*^c(\underline{E}\Gamma);F)
    \to H^{m+1}_{\IZ\OrGF{\Gamma}{\calfin}}(C_*^c(\underline{E}\Gamma^{> 1});F)
    \to \cdots.
  \end{multline*}
  Since $\dim(\underline{E}\Gamma^{> 1}) \le (d-1) \le (m-1)$ holds, we get
  $H^l_{\IZ\OrGF{\Gamma}{\calfin}}(C_*^c(\underline{E}\Gamma^{> 1});F) = 0$ for $l \ge m$.
  Hence we get an isomorphism
  \[
    H^{m+1}_{\IZ\OrGF{\Gamma}{\calfin}}(C_*^c(\underline{E}\Gamma,\underline{E}\Gamma^{> 1});F)    \xrightarrow{\cong}
    H^{m+1}_{\IZ\OrGF{\Gamma}{\calfin}}(C_*^c(\underline{E}\Gamma);F).
  \]
  Note that $C_*^c(\underline{E}\Gamma,\underline{E}\Gamma^{> 1})$ for every $l \ge 0$ and
  $F$ are $\IZ\OrGF{\Gamma}{\calfin}$-modules, which are direct sums of
  $\IZ\OrGF{\Gamma}{\calfin}$-modules of the shape
  $\IZ\mor_{\OrGF{\Gamma}{\calfin}}(?,\Gamma/1)$. Let $\caltr$ be the family of subgroups of
  $\Gamma$ consisting of one elements, namely the trivial subgroup.  We have the inclusion
  $ j \colon \OrGF{\Gamma}{\caltr} \to \OrGF{\Gamma}{\calfin}$ where
  $ \OrGF{\Gamma}{\caltr}$ is the full subcategory of $\OrGF{\Gamma}{\calfin}$ consisting
  of one object $\Gamma/1$.  Note that the category of $\IZ\OrGF{\Gamma}{\caltr}$-modules
  can be identified with the category of $\IZ\Gamma$-modules. Hence
  $j_*j^*C_*^c(\underline{E}\Gamma,\underline{E}\Gamma^{> 1})$ and
  $C_*^c(\underline{E}\Gamma,\underline{E}\Gamma^{> 1})$ are isomorphic
  $\IZ\OrGF{\Gamma}{\calfin}$-chain complexes and $j_*j^*F$ and $F$ are isomorphic
  $\IZ\OrGF{\Gamma}{\calfin}$-modules.  From the adjunction $(j_*,j^*)$, we obtain an
  isomorphism
  \[
    H^{m+1}_{\IZ\OrGF{\Gamma}{\calfin}}(C_*^c(\underline{E}\Gamma,\underline{E}\Gamma^{> 1});F)
    \cong
    H^{m+1}_{\IZ\Gamma}(j^*C_*^c(\underline{E}\Gamma,\underline{E}\Gamma^{> 1});j^*F)
  \]
  Obviously $j^*F$ is a free $\IZ\Gamma$-module. Hence there exists a free $\IZ\pi$-module
  $F'$ with $i_*F'\cong_{\IZ\Gamma} j^*F$.  We conclude from~\cite[Proposition~5.9
  in~III.5 on page~70]{Brown(1982)} that the coinduction $i_{!}F'$ with $i$ of $F'$ is
  $\IZ\Gamma$-isomorphic to $j^*F$.  From the adjunction $(i^*,i_!)$, we obtain an
  isomorphism
  \[
    H^{m+1}_{\IZ\Gamma}(j^*C_*^c(\underline{E}\Gamma,\underline{E}\Gamma^{> 1});j^*F)
    \cong H^{m+1}_{\IZ\pi}(i^*j^*C_*^c(\underline{E}\Gamma,\underline{E}\Gamma^{> 1});F')
  \]
  Hence it suffices to show
  $H^{m+1}_{\IZ\pi}(i^*j^*C_*^c(\underline{E}\Gamma,\underline{E}\Gamma^{> 1});F') = 0$.

  If we apply $i^* \circ j^*$ to the short exact sequence~\eqref{def_D_ast} of
  $\IZ\OrGF{\Gamma}{\calfin}$-chain complexes, we obtain an exact sequence of
  $\IZ\pi$-chain complexes and hence a long exact sequence
  \begin{multline*}
    \cdots \to H^m_{\IZ \pi}(i^*j^*C_*^c(\underline{E}\Gamma^{> 1});F')
    \to H^{m+1}_{\IZ \pi}(i^*j^*C_*^c(j^*\underline{E}\Gamma,\underline{E}\Gamma^{>1});F')
    \\
    \to H^{m+1}_{\IZ\pi}(i^*j^*C_*^c(\underline{E}\Gamma);F')
    \to H^{m+1}_{\IZ \pi}(i^*j^*C_*^c(\underline{E}\Gamma^{> 1});F')\to \cdots.
  \end{multline*}
  Since $\dim(\underline{E}\Gamma^{> 1}) \le (d-1) \le (m-1)$ holds, we get
  $H^k_{\IZ\pi}(C_*^c(i^*\underline{E}\Gamma^{> 1});F') = 0$ for $k \ge m$.  Therefore we
  obtain an isomorphism
  \[ H^{m+1}_{\IZ \pi}(i^*j^*C_*^c(\underline{E}\Gamma,\underline{E}\Gamma^{> 1});F')
   \xrightarrow{\cong}  H^{m+1}_{\IZ \pi}(i^*j^*C_*^c(\underline{E}\Gamma);F').
     \]
  Hence it remains to show $H^{m+1}_{\IZ \pi}(i^*j^*C_*^c(\underline{E}\Gamma);F') =
  0$. This follows from the fact that $i^*\underline{E}\Gamma$ is a model for $E\pi$ and
  up to $\pi$-homotopy there is a $\pi$-$CW$-model for $\underline{E}\pi$ with
  $\dim(\underline{E}\pi) \le d \le m$.
\end{proof}

Now we are ready to give the proof of
Theorem~\ref{the:Models_for_the_classifying_space_for_proper_Gamma-actions_intro}.

\begin{proof}[Proof of
  Theorem~\ref{the:Models_for_the_classifying_space_for_proper_Gamma-actions_intro}]
  We conclude from Proposition~\ref{pro:underline(E)Gamma_for_hyperbolic_Gamma} and
  Proposition~\ref{pro:CAT(0)_and_eub_Ontaneda} that we can assume without loss of
  generality that the condition~\eqref{the:Nielsen_Realization_Problem:eub_direct}
  appearing under~\eqref{the:Nielsen_Realization_Problem:eub} in
  Theorem~\ref{the:Models_for_the_classifying_space_for_proper_Gamma-actions_intro} is
  satisfied, namely, that there is a finite $\Gamma$-$CW$-model for $\underline{E}\Gamma$.

  This implies that there is a $\Gamma$-$CW$-model for $E\Gamma$ of finite type,
  see~\cite[Lemma~4.1]{Lueck(2000a)}.  Moreover, we can arrange in the $\Gamma$-pushout
  appearing in Proposition~\ref{pro:constructing_underline(E)Gamma)} that the
  $\Gamma$-$CW$-complexes $\coprod_{F\in\calm} \Gamma \times_{N_{\Gamma} F}EF$ and $E\Gamma$ are
  of finite type.  Hence we get from Proposition~\ref{pro:constructing_underline(E)Gamma)}
  a $\Gamma$-$CW$-model $X$ of finite type for $\underline{E}\Gamma$ such that $X^{>1}$ is
  $\coprod_{F\in\calm} \Gamma/F$.

  Let $\pi$ be the subgroup of finite index appearing in
  condition~\eqref{the:Nielsen_Realization_Problem_vcd} of
  Theorem~\ref{the:Models_for_the_classifying_space_for_proper_Gamma-actions_intro}. Since
  $\Gamma$ is finitely generated, there exists a normal subgroup $\pi' \subseteq \Gamma$
  of finite index with $\pi' \subseteq \pi$.  Hence we can assume without loss of
  generality that $\pi$ is normal, otherwise replace $\pi$ by $\pi'$.

  There is a $d$-dimensional $\Gamma$-$CW$-model for $\underline{E}\Gamma$ by
  Lemma~\ref{lem:Brown_small_singular_set}.

  Now Theorem~\ref{the:Models_for_the_classifying_space_for_proper_Gamma-actions_intro}
  follows from Lemma~\ref{lem:finite_d-dimensional_models}.
\end{proof}


\typeout{--------------------------------- Section 3: Free slice systems
  -------------------------------}

\section{Free slice systems}%
\label{sec:free_slice-systems}

Let $\calm$ be a complete system of representatives of the conjugacy classes of maxi\-mal
finite subgroups of $\Gamma$.

\begin{definition}\label{def:free_d-dimensional_slide_system}
  A \emph{$d$-dimensional free slice system} $\cals = \{S_F \mid F\in \calm\}$, or shortly
  \emph{slice system}, consists of a free $(d-1)$-dimensional $CW$-complex $S_F$ for every
  $F \in \calm$ such that $S_F$ is after forgetting the $F$-action homotopy equivalent to
  the $(d-1)$-dimensional standard sphere $S^{d-1}$.

  We call $\cals$ \emph{oriented}, if we have chosen a generator $[S_F]$, called
  fundamental class, for the infinite cyclic group $H_{d-1}(S_F)$ for every $F \in \calm$.
\end{definition}

We denote by $D_F$ the cone over $S_F$.  Obviously $(D_F,S_F)$ is $F$-$CW$-pair and
$(D_F,S_F)$ is after forgetting the $F$-action homotopy equivalent to $(D^d,S^{d-1})$.

Let $\cals = \{S_F \mid F\in \calm\}$ and $\cals' = \{S'_F \mid F\in \calm\}$ be
$d$-dimensional free slice systems. Note that this implies that each $F$ in $\calm$ has periodic
cohomology, or, equivalently, that any Sylow subgroup of $F$ is cyclic or a generalized
quaternion group, see~\cite[Section~VI.9 on pages~153--160]{Brown(1982)}.

Assume that $F$ acts orientation preserving on $S_F$.  Then $H_{d-1}(S_F/F)$ is infinite
cyclic and comes with a preferred generator $[S_F/F]$, which is uniquely determined by the
property that the map $H_{d-1}(S_F) \to H_{d-1}(S_F/F)$ induced by the projection
$S_F \to S_F/F$ sends $[S_F]$ to $|F| \cdot [S_F/F]$.  Let $c(S_F) \colon S_F \to EF$ be a
classifying map. Denote by $c_F \colon S_F/F \to BF$ the map $c(S_F)/F$.  If $\cals$ is
oriented, we can define
\begin{equation}
  d(S_F) \in H_{d-1}(BF)
  \label{d(S_F)}
\end{equation}
to be the class given by the image of $[S_F/F] \in H_{d-1}(S_F/F)$ under the homomorphism
$H_{d-1}(c_F) \colon H_{d-1}(S_F/F) \to H_{d-1}(BF)$.  Note that $H_{d-1}(S_F/F)$ is
infinite cyclic.  Since $c(S_F) \colon S_F \to EG$ and hence $c_F \colon S_F/F \to BF$ are
$(d-1)$-connected, the homomorphism $H_{d-1}(c_F) \colon H_{d-1}(S_F/F) \to H_{d-1}(BF)$
is surjective. We conclude that $H_{d-1}(c_F)([S_F/F])$ is a generator of the cyclic group
$H_{d-1}(BF)$. The order of $H_{d-1}(BF)$ is $|F|$, see~\cite[Theorem~9.1 in Chapter VI on
page~154]{Brown(1982)}.

\begin{lemma}\label{lem:basics_about_free_slice_systems}
  
  Let $\cals = \{S_F \mid F\in \calm\}$ and $\cals' = \{S'_F \mid F\in \calm\}$ be slice
  systems.

  \begin{enumerate}
  \item\label{lem:basics_about_free_slice_systems:|F|_ge_3}
    If $d$ is odd, then $F \cong \IZ/2$ for all $F \in \calm$;
  \item\label{lem:basics_about_free_slice_systems:orientation_preserving} 
    If $d$ is even, then the $F$-action on $S_F$ is orientation preserving. If $d$ is odd,
    then the $F$-action on $S_F$ is orientation reversing;

  \item\label{lem:basics_about_free_slice_systems:homotopy_equivalent} Suppose that $F$
    acts orientation preserving on $S_F$ and $S_F'$ and $\cals$ and $\cals'$ are oriented.
    Then there exists an orientation preserving $F$-homotopy equivalence
    $S_F \xrightarrow{\simeq} S'_F$, if and only if $d(S_F) = d(S'_F)$ holds in
    $H_{d-1}(BF)$;

  \item\label{lem:basics_about_free_slice_systems:uniquness_of_homotopy_equivalences} If
    $|F| \ge 3$, then any $F$-selfhomotopy equivalence $S_F \to S_F$ is $F$-homotopic to
    the identity.

  \item\label{lem:basics_about_free_slice_systems:trivial_Whitehead_torsion} Any
    $F$-selfhomotopy equivalence $S_F \to S_F$ is simple
  \end{enumerate}
\end{lemma}
\begin{proof}\eqref{lem:basics_about_free_slice_systems:|F|_ge_3}
  Since $F$ acts freely, we have $1 +  (-1)^{d-1} = \chi(S_F) \equiv 0 \mod |F|$. This implies that $|F| = 2$,
  if $d$ is odd.
  \\[1mm]~\eqref{lem:basics_about_free_slice_systems:orientation_preserving}
  Consider a nontrivial element $g \in F$. Multiplication with $g$ induces a map
  $l_g \colon S_F \to S_F$, which has no fixed points. Hence its Lefschetz number
  $\Lambda(l_g)$ vanishes.  We have $\Lambda(l_g) = 1 +(-1)^{d-1} \cdot \deg(l_g)$, where
  $\deg(l_g)$ is the degree. If $d$ is even, $\deg(l_g) = 1$, and, if $d$ is odd, $\deg(l_g) = -1$;
  \\[1mm]~\eqref{lem:basics_about_free_slice_systems:homotopy_equivalent} Let
  $f \colon S_F \to S'_F$ be an orientation preserving $F$-homotopy equivalence.  Since
  the $F$-maps $c(S'_F) \circ f$ and $c(S_F)$ are $F$-homotopic,
  $H_{d-1}(c(S_F))([S_F/F])$ agrees with $H_{d-1}(c(S'_F))([S'_F/F])$. This implies
  $d(S_F) = d(S'_F)$.

  Now suppose $d(S_F) = d(S'_F)$. By elementary obstruction theory one can find an $F$-map
  $f \colon S_F \to S'_F$.  We get $H_{d-1}(f/F)([S_F/F]) = \deg(f) \cdot [S'_F/F]$, if
  $\deg(f)$ is the degree of $f$.  Since $c(S_F') \circ f$ is $F$-homotopic to $c(S_F)$,
  we conclude
  \[
    H_{d-1}(c_F)([S_F]) = \deg(f) \cdot H_{d-1}(c(S'_F)/F)([S'_F]).
  \]
  Since $H_{d-1}(BF)$ is a finite cyclic of order $|F|$ and both elements
  $H_{d-1}(c_F)([S_F])$ and $H_{d-1}(c(S'_F)/F)([S'_F])$ are generators, we conclude
  $\deg(f) \equiv 1 \mod |F|$ from $d(S_F) = d(S'_F)$.  Since for any integer $m \in \IZ$
  with $m \equiv \deg(f) \mod |F|$ one can find an $F$-map $f' \colon S_F \to S'_F$ with
  $\deg(f') = m$, see~\cite[Theorem~4.11 on page~126]{Dieck(1987)} or~\cite[Theorem~3.5 on
  page 139]{Lueck(1988)}, there exists an $F$-map $f' \colon S_F \to S'_F$ of degree
  $1$. This implies that $f'$ is an orientation preserving $F$-homotopy equivalence.
  \\[1mm]~\eqref{lem:basics_about_free_slice_systems:uniquness_of_homotopy_equivalences}
  If $f \colon S_F \to S_F$ is an $F$-map, then its degree satisfies
  $\deg(f) \equiv 1 \mod |F|$, see~\cite[Theorem~4.11 on page~126]{Dieck(1987)}
  or~\cite[Theorem~3.5 on page 139]{Lueck(1988)}.  Consider an $F$-self homotopy
  equivalence $f \colon S_F \to S_F$. Then $\deg(f) \in \{\pm 1\}$. Since $|F| \ge 3$
  holds by assumption, $\deg(f) = 1$. Since two $F$-maps $S_F \to S_F$ are $F$-homotopic,
  if and only if their degrees agree, see~\cite[Theorem~4.11 on page~126]{Dieck(1987)}
  or~\cite[Theorem~3.5 on page 139]{Lueck(1988)}, $f$ is $F$-homotopic to the identity.
  \\[1mm]~\eqref{lem:basics_about_free_slice_systems:trivial_Whitehead_torsion} If
  $F \ge 3$, this follows from
  assertion~\eqref{lem:basics_about_free_slice_systems:uniquness_of_homotopy_equivalences}.
  If $|F| \le 2$, this follows from $\Wh(F) = \{0\}$.
\end{proof}

\begin{lemma}\label{lem:partial_f_yields_partial_f_F}
  Let $\cals = \{S_F \mid F\in \calm\}$ and $\cals' = \{S'_F \mid F\in \calm\}$ be slice
  systems.  Let
  \[
     v  \colon \coprod_{F \in \calm}\Gamma \times_F S_F \to \coprod_{F \in \calm}
    \Gamma \times_F S_F'
  \]
  be a $\Gamma$-map. Suppose that conditions (M) and (NM) are satisfied.
  Then:

  \begin{enumerate}
     \item\label{lem:partial_f_yields_partial_f_F:S_F_to_S_F_prime}
  The map $v$ induces for each $F \in \calm$ an $F$-map
  $v_F \colon S_F \to S_F'$ such that
  $v = \coprod_{F \in \calm} \id_{\Gamma} \times_F \, v_F$
  holds. Moreover, the collection of the $F$-homotopy classes of the $F$-maps $v_F$
  for $F \in \calm$ is determined by the $\Gamma$-homotopy class of $v$ and vice
  versa;

  \item\label{lem:partial_f_yields_partial_f_F:extension} There exists a $\Gamma$-map
  $V \colon \coprod_{F \in \calm}\Gamma \times_F D_F \to \coprod_{F \in \calm} \Gamma
  \times_F D'_F$ extending $v$;  

\end{enumerate}
\end{lemma}
\begin{proof}
  We obtain a commutative diagram of $\Gamma$-sets  induced by $v$
  \[
    \xymatrix{\pi_0\bigl(\coprod_{F \in \calm}\Gamma \times_F S_F\bigr) 
      \ar[r]^{\pi_0(v)} \ar[d]^{\cong}
      &
      \pi_0\bigl(\coprod_{F \in \calm}\Gamma \times_F S'_F\bigr) \ar[d]^{\cong}
      \\
      \coprod_{F \in \calm} \Gamma/F \ar[r]^{\overline{v}} 
      &
      \coprod_{F \in \calm} \Gamma/F 
      }
  \]
  whose vertical arrows are the obvious bijections. Consider $F \in \calm$.  Since
  $\overline{v}$ is a $\Gamma$-map, there exists $F' \in F$ such that $\overline{v}$
  induces a $\Gamma$-map $\overline{v}_F \colon \Gamma/F \to \Gamma/F'$.  There exists
  $\gamma \in \Gamma$ such that $\overline{v}_F$ sends $eF$ to $\gamma F'$ for
  $e \in \Gamma$ the unit. Then $\gamma^{-1}F\gamma \subseteq F'$.  Since $F$ and $F'$ are
  maximal finite, we obtain $\gamma^{-1}F\gamma = F'$. Since two elements in $\calm$,
  which are conjugated, are automatically equal, we get $F = \gamma^{-1}F\gamma = F'$.
  Hence we get $\gamma \in N_{\Gamma} F = F = F'$ because of (NM). This implies
  $F = F'$ and $\overline{v}_F = \id_{\Gamma/F}$. Hence $\overline{v}$ is the identity.
  This implies $v(S_F) \subseteq S'_F$, where we identify $S_F$ with the subspace
  $\{(\gamma, x) \mid \gamma \in F, x \in S_F\}$ of $\Gamma \times_F S_F$ and analogously
  for $S_F'$.  This shows
  assertion~\eqref{lem:partial_f_yields_partial_f_F:S_F_to_S_F_prime}.
  Assertion~\eqref{lem:partial_f_yields_partial_f_F:extension} follows from
  assertion~\eqref{lem:partial_f_yields_partial_f_F:S_F_to_S_F_prime}.  
\end{proof}

\begin{remark}\label{rem:slice-systems_and_Gamma}
  Note that the existence of a $d$-dimensional free slice system depends only on $\Gamma$,
  actually only on $\calm$.  If for instance $\Gamma$ contains a normal torsionfree
  subgroup $\pi$ such that $G := \Gamma/\pi$ is finite cyclic and $d$ is even, then
  obviously one can find a $d$-dimensional free slice system, since $G$ acts freely on
  $S^{d-1}$.
\end{remark}


\typeout{-------------------- Section 4: Slice complement models --------}

\section{Slice complement models}~%
\label{subsec:Slice-models}

    \begin{notation}\label{not_C(Z)}
      Given a space $Z$, let $C(Z)$ be its \emph{path componentwise cone}, i.e,
      $C(Z) := \coprod_{C \in \pi_0(Z) } \cone(C)$.
    \end{notation}

    One may describe $C(Z)$ also by the pushout
    \[
      \xymatrix{Z \ar[r]^p \ar[d]_{i_0} & \pi_0(Z) \ar[d]
        \\
        Z \times [0,1] \ar[r] & C(Z) }
    \]
    where $i_0 \colon Z \to Z \times[0,1]$ sends $z$ to $(z,0)$, $p \colon Z \to \pi_0(Z)$
    is the projection, and $\pi_0(Z)$ is equipped with the discrete topology.  If $Z$ is a
    $\Gamma$-$CW$-complex, then $C(Z)$ inherits a $\Gamma$-$CW$-structure. If
    $\cals = \{S_F \mid F \in \calm\}$ is a free $d$-dimension slice system, we get an
    identification of $\Gamma$-$CW$-complexes
    \[
      C(\coprod_{F \in \calm} \Gamma \times_F S_F) = \coprod_{F \in \calm} \Gamma \times_F
      D_F.
    \]
  
\begin{definition}[Slice complement model]\label{def:slice-model_for_eub(Gamma)}
  We call a free $\Gamma$-$CW$-pair $(X,\partial X)$ a \emph{slice complement model for
    $\underline{E}\Gamma$}, or shortly \emph{slice complement model}, with respect to the slice
  system $\cals = \{S_F \mid F \in \calm\}$, if
  $\partial X = \coprod_{F \in \calm} \Gamma \times_F S_F$ and
  $X \cup_{\partial X} C(\partial X)$ is a model for $\eub{\Gamma}$.
\end{definition}

We will frequently use that a slice complement model comes with a canonical homotopy $\Gamma$-pushout
\begin{equation}
  \xymatrix{\partial X \ar[r] \ar[d]
    &
    X
    \ar[d]
    \\
    \partial \eub{\Gamma} \ar[r]
    &
    \eub{\Gamma}.
  }
  \label{canonical_homotopy_Gamma_pushout_for_a_slice_model}
\end{equation}

\begin{lemma}\label{lem:characterization_of_slice-models}
  Consider a free $d$-dimensional $\Gamma$-$CW$-pair $(X,\partial X)$ and a slice system
  $\cals = \{S_F \mid F \in \calm\}$ such that
  $\partial X = \coprod_{F \in \calm} \Gamma \times_F S_F$. Suppose $d \ge 3$ and that
  conditions (M) and (NM) hold, see Notation~\ref{not:(M)_and_(NM)_intro}.  Fix
  $k \in \{2, 3, \ldots, (d-1)\}$.

  Then the following assertions are equivalent:
  \begin{enumerate}
  \item\label{lem:characterization_of_slice-models:slice-model} $(X,\partial X)$ is a
    slice complement model for $\eub{\Gamma}$;
  \item\label{lem:characterization_of_slice-models:contractible}
    $X \cup_{\partial X} C(\partial X)$ is contractible (after forgetting the
    $\Gamma$-action);
  \item\label{lem:characterization_of_slice-models:homology} The space $X$ is
    $(k-1)$-connected and $H_n(X,\partial X)$ vanishes for $k \le n \le d$.
  \end{enumerate}
\end{lemma}
\begin{proof}~\eqref{lem:characterization_of_slice-models:slice-model}~$\Longleftrightarrow$~%
\eqref{lem:characterization_of_slice-models:contractible} For every finite non-trivial
  subgroup $H \subseteq G$, there exists precisely one element $F \in \calm$ such that for
  some $\gamma \in \Gamma$ we have $\gamma H \gamma^{-1} \subseteq F$ and hence
  \[
    X \cup_{\partial X} C(\partial X)^{H} \cong X \cup_{\partial X} C(\partial X)^{\gamma
      H \gamma} = X \cup_{\partial X} C(\partial X)^{F} = \pt
  \]
  holds.  Hence $X \cup_{\partial X} C(\partial M)$ is a model for $\eub{\Gamma}$, if and
  only if $X \cup_{\partial X} C(\partial X)$ is contractible.
  \\[1mm]~\eqref{lem:characterization_of_slice-models:contractible}~$\Longleftrightarrow$~%
\eqref{lem:characterization_of_slice-models:homology} Since the map
  $\partial X\to C(\partial X)$ is $(d-1)$-connected, the inclusion
  $X \to X \cup_{\partial X} C(\partial X)$ is $(d-1)$-connected. In particular $X$ is
  $(k-1)$-connected if and only if $X \cup_{\partial X} C(\partial X)$ is $(k-1)$
  connected.  We have $H_n(C(\partial X)) = 0$ for $n \ge 1$.  Hence we get from the long
  exact sequence of the pair $(X \cup_{\partial X} C(\partial X), C(\partial X))$ and by
  excision isomorphisms for $n \ge 2$
  \[
    H_n(X,\partial X) \xrightarrow{\cong} H_n(X \cup_{\partial X} C(\partial X),
    C(\partial X)) \xleftarrow{\cong} H_n(X \cup_{\partial X} C(\partial X)).
  \]
  By the Hurewicz Theorem $X \cup_{\partial X} C(\partial X)$ is contractible, if and only
  if $X \cup_{\partial X} C(\partial X)$ is $(k-1)$-connected and
  $H_n(X \cup_{\partial X} C(\partial X))$ vanishes for $k \le n$.  As
  $X \cup_{\partial X} C(\partial X)$ is $d$-dimensional
  $H_n(X \cup_{\partial X} C(\partial X))$ vanishes for $n \ge (d+1)$.  This finishes the
  proof of Lemma~\ref{lem:characterization_of_slice-models}.
\end{proof}

\begin{lemma}\label{lem:slice-models_and_pushouts}
  Consider two free $d$-dimensional slice systems $\cals = \{S_F \mid F \in \calm\}$ and
  $\cals' = \{S'_F \mid F \in \calm\}$. Put
  $\partial X = \coprod_{F \in \calm} \Gamma \times_F S_F$ and
  $\partial X' = \coprod_{F \in \calm} \Gamma \times_F S_F'$.  Let
  $\partial u \colon \partial X \to \partial X'$ be any cellular $\Gamma$-map. Consider
  the $\Gamma$-pushout
  \[
    \xymatrix{\partial X \ar[r] \ar[d]_{\partial u} & X \ar[d]^{u}
      \\
      \partial X' \ar[r] & X' }
  \]
  Suppose that $(X,\partial X)$ is a finite free $d$-dimensional $\Gamma$-$CW$-pair.
  Equip $(X',\partial X')$ with the induced structure of a finite free $d$-dimensional
  $\Gamma$-$CW$-pair.

  Then $(X',\partial X')$ is a slice complement model for $\eub{\Gamma}$ if $(X,\partial X)$ is a
  slice complement model for $\eub{\Gamma}$ and $d \ge 3$, and $(X,\partial X)$ is a slice complement model for
  $\eub{\Gamma}$ if $(X',\partial X')$ is a slice complement model for $\eub{\Gamma}$ and $d \ge 4$.
\end{lemma}
\begin{proof}
  Since $ \partial u$ is $(d-2)$-connected, the map $u$ is $(d-2)$-connected.  We conclude
  from excision that the map
  $H_n(u,\partial u) \colon H_n(X,\partial X) \xrightarrow{\cong} H_n(X',\partial X')$ is bijective
  for all $n \in \IZ$.  Now the claim follows from
  Lemma~\ref{lem:characterization_of_slice-models}.
\end{proof}


\typeout{--------------------------------- Section 5: Poincar\'e pairs
  -------------------------------}

\section{Poincar\'e pairs}%
\label{sec:Poincare_pairs}

Recall from the introduction that the main goal of this paper is to find a slice complement model
$(X,\partial X)$ such that $(X/\partial X,\partial X/\Gamma)$ is a Poincar\'e pair.  This
is a necessary condition for finding a slice complement model $(X,\partial X)$ such that
$(X/\partial X,\partial X/\Gamma)$ is a compact manifold, which we will finally prove in the
sequel~\cite{Davis-Lueck(2022_manifold_models)} to this paper. Next we give some information about
Poincar\'e pairs in general and prove some results needed later.


\subsection{Review on Poincar\'e pairs}\label{subsec:Poincare_pairs}

We recall some basics about Poincar\'e pairs following~\cite{Wall(1967)} and prove
Lemma~\ref{lem:subtracting_a_Poincare_pair}, which is a mild generalization
of~\cite[Theorem~2.1]{Wall(1967)}.

Given a $CW$-complex $X$, we denote by $p_X \colon \widetilde{X} \to X$ its universal
covering.  If $\pi$ is the fundamental group of $X$, then $\widetilde{X}$ is a free
$\pi$-$CW$-complex and $p_X$ is the quotient map of this $\pi$-action on $\widetilde{X}$.
If $A \subseteq X$ is a $CW$-subcomplex, we denote by $\overline{A} = p_X^{-1}(A)$.  We
get a free $\pi$-$CW$-pair $(\widetilde{X},\overline{A})$. Note that
$p|_{\overline{A}} \colon \overline{A} \to A$ is a $\pi$-covering.  It is the universal covering
of $A$, if and only if $A$ is connected and the inclusion $A \to X$ induces a bijection on
the fundamental groups. Given an element $w \in H^1(\pi,\IZ/2)$, which is the same as a
group homomorphism $w \colon \pi \to \{\pm 1\}$, we denote by $\IZ^w$ the $\IZ\pi$-module,
whose underlying abelian group is $\IZ$ and on which $\omega \in \pi$ acts by
multiplication with $w(\omega)$. We denote by
$H_*^{\pi}(\widetilde{X},\overline{A};\IZ^w)$ the homology of the $\IZ$-chain complex
$\IZ^w \otimes_{\IZ \pi} C_*(\widetilde{X},\overline{A})$.  Consider an element
$u \in H_d^{\pi}(\widetilde{X},\overline{A};\IZ^w)$.  Let $\partial \colon H_d^{\pi}(\widetilde{X},\overline{A};\IZ^w)
\to H_{d-1}^{\pi}(\overline{A};\IZ^w)$ be the boundary homomorphism. We obtain by the cap product
$\IZ\pi$-chain maps, unique up to $\IZ \pi$-chain homotopy
\begin{eqnarray}
  - \cap u \colon C^{d-*}(\widetilde{X},\overline{A})
  & \to &
          C_*(\widetilde{X});
          \label{Poincare_chain_map;rel_to_absolute}
  \\
  - \cap u\colon C^{d-*}(\widetilde{X})
  & \to &
          C_*(\widetilde{X},\overline{A});
          \label{Poincare_chain_map_absolute_rel}
  \\
  - \cap \partial(u) \colon C^{d-1-*}(\overline{A})
  & \to &
          C_*(\overline{A}).
          \label{Poincare_chain_map_boundary}        
\end{eqnarray}
Here the dual chain complexes are to be understood with respect to the $w$-twisted
involution $\IZ\pi \to \IZ\pi$ sending $\sum_{\omega \in \pi} n_{\omega} \cdot \omega$ to
$\sum_{\omega \in \pi} w(\omega) \cdot n_{\omega} \cdot \omega^{-1}$.

\begin{definition}[(Simple) finite $d$-dimensional Poincar\'e pair]%
\label{def:Simple_finite_d-dimensional_Poincare_pair}
  We call a finite $d$-dimensional $CW$-pair $(X,\partial X)$ with connected $X$ a
  \emph{(simple) finite $d$-dimensional Poincar\'e pair} with respect to the
  \emph{orientation homomorphism} $w \in H^1(X,\IZ/2) = \hom(\pi_1(X),\{\pm 1\})$ and
  \emph{fundamental class}
  $[X,\partial X] \in H_d^{\pi}(\widetilde{X},\overline{\partial X};\IZ^w)$, if the
  $\IZ\pi$-chain
  maps~\eqref{Poincare_chain_map;rel_to_absolute},~\eqref{Poincare_chain_map_absolute_rel}
  and~\eqref{Poincare_chain_map_boundary} for $u = [X,\partial X]$ are (simple)
  $\IZ\pi$-chain homotopy equivalences.
\end{definition}

Note that all three $\IZ\pi$-chain maps~%
\eqref{Poincare_chain_map;rel_to_absolute},~\eqref{Poincare_chain_map_absolute_rel},
and~\eqref{Poincare_chain_map_boundary} are (simple) $\IZ\pi$-chain homotopy equivalences,
if two of them are.

Since the $\IZ\pi$-chain map~\eqref{Poincare_chain_map;rel_to_absolute} is a
$\IZ\pi$-chain homotopy equivalence and induces an isomorphism
$H_d^{\pi}(\widetilde{X},\overline{X};\IZ^w) \xrightarrow{\cong} H^0_{\pi}(\widetilde{X})$
and $H^0_{\pi}(\widetilde{X}) \cong H^0(X) \cong \IZ$, the group
$H_d^{\pi}(\widetilde{X},\overline{X};\IZ^w)$ is infinite cyclic and
$[X,\partial X] \in H_d^{\pi_1(X)}(\widetilde{X},\overline{\partial X};\IZ^w)$ is
a generator.

\begin{remark}[Orientation homomorphisms]\label{rem:extension_of_the_orientation_homomorphism_for_pi}
  Part of the definition of a Poincar\'e pair $(X,\partial X)$ is the existence of an
  appropriate orientation homomorphism $w \in H^1(X;\IZ/2)$. Note that $w$ is uniquely
  determined by the mere fact that ($X,\partial X)$ together with the choice of $w$ is a
  finite Poincar\'e pair. There is even a recipe, how to construct $w$ from the underlying
  $CW$-pair $(X,\partial X)$, see~\cite[Lemma~5.46 on page~109]{Crowley-Lueck-Macko(2021)},
  which we will recall next.

  Let the untwisted dual $\IZ\pi$-chain complex
  $C^{d-*}_{\untw}(\widetilde{X},\overline{\partial X})$ be defined with respect to the
  untwisted involution on $\IZ\pi$ sending $\sum_{\omega \in \pi} n_{\omega} \cdot \omega$
  to $\sum_{\omega \in \pi} n_{\omega} \cdot \omega^{-1}$. Note that
  $C^{d-*}_{\untw}(\widetilde{X},\overline{\partial X})$ depends only on the
  $\Gamma$-$CW$-complex $X$ but not on $w$, and that
  $\IZ^w \otimes_{\IZ} C^{d-*}_{\untw}(\widetilde{X},\overline{\partial X})$ equipped with
  the diagonal $\pi$-action is $C^{d-*}(\widetilde{X},\overline{\partial X})$. Now $w$ is
  given by the $\pi$-action on the infinite cyclic group
  $H_0(C^{d-*}_{\untw}(\widetilde{X},\overline{\partial X}))$, or, equivalently by the
  condition that the $\IZ \Gamma$-module
  $H_0(C^{d-*}_{\untw}(\widetilde{X},\overline{\partial X}))$ is $\IZ\Gamma$-isomorphic to
  $\IZ^w$.  This follows from the fact that the $\IZ\pi$-chain
  map~\eqref{Poincare_chain_map;rel_to_absolute} induces a $\IZ\pi$-isomorphism from
  $H_0(C^{d-*}(\widetilde{X},\overline{\partial X}))$ to $H_0^{\pi}(C_*(\widetilde{X}))$
  and the $\pi$-action on $H_0^{\pi}(C_*(\widetilde{X})) \cong \IZ$ is trivial.

  There are only two possible choices for the fundamental class $[X/\pi,\partial X/\pi]$,
  since it has to be a generator of the infinite cyclic group
  $H_d^{\pi}(\widetilde{X},\overline{X};\IZ^w)$.
\end{remark}

This definition extends to not necessarily connected $X$ as follows. We call a finite
$d$-dimensional $CW$-pair $(X,\partial X)$ a finite $d$-dimensional Poincar\'e pair with
respect to $w \in H^1(X;\IZ/2)$ with fundamental class
\[ [X,\partial X] \in \bigoplus_{C \in \pi_0(C)} H_d^{\pi_1(C)}(\widetilde{C},
  \overline{C \cap \partial X};\IZ^{w|_C})
\]
if for each path component $C$ of $X$ the pair $(C,C \cap \partial X)$ is a finite
$d$-dimensional Poincar\'e complex with respect to $w|_C \in H^1(C;\IZ/2)$ coming from $w$
by restriction to $C$ and the fundamental class $[C,(C \cap \partial X)]$ in
$H_d^{\pi_1(C)}(\widetilde{C}, \overline{C \cap \partial X};\IZ^{w|_C})$, which is the
component associated to $C \in \pi_0(X)$ of $[X,\partial X]$.

Note that $\partial X$ inherits the structure of a finite $(d-1)$-dimensional
Poincar\'e-complex, provided that $\pi_1(\partial X,x) \to \pi_1(X,x)$ is injective for
every $x \in \partial X$. Moreover, $\partial X$ inherits the structure of a simple finite
$(d-1)$-dimensional Poincar\'e-complex, provided that $\pi_1(\partial X,x) \to \pi_1(X,x)$
is injective for every $x \in \partial X$ and $\Wh(\partial X) \to \Wh(X)$ is injective.
The last condition is satisfied for instance, if the functor $\Pi(\partial X) \to \Pi(X)$ on the
fundamental groupoids induced by the inclusion $\partial X \to X$ is an equivalence of
categories.  All these claims follow from the following subsection.

\subsection{Arbitrary coverings}\label{subsec:arbitrary_coverings}

We have to deal with arbitrary coverings as well.  Let $(X,A)$ be finite $CW$-pair with
connected $X$. Put $\pi = \pi_1(X)$.  Let $\Gamma$ be a group and
$p \colon \widehat{X} \to X$ be a $\Gamma$-covering.  Put $\widehat{A} = p^{-1}(A)$. Let
$p_X \colon \widetilde{X} \to X$ be the universal covering.  Put
$\overline{A} = p_X^{-1}(\widetilde{X})$. Consider group homomorphisms
$w \colon \pi \to \{\pm 1\}$ and $v \colon \Gamma \to \{\pm 1\}$.  Let
$\phi \colon \pi \to \Gamma$ be the group homomorphism, for which there is a
$\Gamma$-homeomorphism
$F \colon \Gamma \times_{\phi} \widetilde{X} \xrightarrow{\cong} \widehat{X}$.  Note
that $F$ induces a $\Gamma$-homeomorphism
$f \colon \Gamma \times_{\phi} \overline{A} \xrightarrow{\cong}\widehat{A}$.  Suppose that
$v = w \circ \phi$.

Then the $\Gamma$-map $(F,f)$ induces an isomorphism of $\IZ\Gamma$-chain complexes
\[\IZ\Gamma \otimes_{\IZ\phi} C_*(\widetilde{X},\overline{A})
  \xrightarrow{\cong} C_*(\widehat{X},\widehat{A}).
\]
In particular we get an isomorphism
\begin{equation}
  \mu \colon H_d^{\pi}(\widetilde{X},\overline{A},\IZ^w) \xrightarrow{\cong}
  H_d^{\Gamma}(\widehat{X},\widehat{A},\IZ^v).
  \label{iso_mu}
\end{equation}
Given an element $u \in H_d^{\pi}(\widetilde{X},\overline{A},\IZ^w)$, the $\IZ\Gamma$-chain
homotopy equivalence
\[- \cap \mu(u) \colon C^{d-*}(\widehat{X},\widehat{A}) \to C_*(\widehat{X})
\]
is obtained from the $\IZ\pi$-chain map~\eqref{Poincare_chain_map;rel_to_absolute} by
induction with $\phi \colon \pi \to \Gamma$. Note that
$- \cap \mu(u) \colon C^{d-*}(\widehat{X},\widehat{A}) \to C_*(\widehat{X})$ is a
$\IZ\Gamma$-chain homotopy equivalence, if the $\IZ\pi$-chain
map~\eqref{Poincare_chain_map;rel_to_absolute} is a $\IZ\pi$-chain homotopy equivalence.
The converse is true, provided that $\phi$ is injective.  Moreover,
$- \cap\mu(u) \colon C^{d-*}(\widehat{X},\widehat{A}) \to C_*(\widehat{X})$ is a simple
$\IZ\Gamma$-chain homotopy equivalence, if the $\IZ\pi$-chain
map~\eqref{Poincare_chain_map;rel_to_absolute} is a simple $\IZ\pi$-chain homotopy
equivalence.  The converse is true, provided that $\phi$ is injective and induces an
injection $\Wh(\pi) \to \Wh(\Gamma)$.

The analogous statements are true for~\eqref{Poincare_chain_map_absolute_rel}.

    
\subsection{Subtracting a Poincar\'e pair}\label{subsec:Subtracting_a_Poincare_pair}

Let $M$ be a closed manifold. Suppose that we have embedded a codimension zero manifold
$(N,\partial N)$ into $M$.  If we subtract $(N,\partial N)$ from $M$ in the sense that we
delete the interior of $N$ from $M$, then we obtain a manifold with boundary $\partial N$.
We want to prove the analogue for Poincar\'e pairs.

Consider a connected finite $d$-dimensional $CW$-complex $Y$ with fundamental group $\pi$.
Consider (not necessarily connected) $CW$-subcom\-plexes $Y_1$, $Y_2$, and $Y_0$ of $Y$
satisfying $Y = Y_1 \cup Y_2$ and $Y_0 = Y_1 \cap Y_2$ such that
$\dim(Y_1) = \dim(Y_2) = d$ and $\dim(Y_0) = d-1$ hold.  Let
$p_Y \colon \widetilde{Y} \to Y$ be the universal covering of $Y$.  Put
$\overline{Y_i} = p_Y^{-1}(Y_i)$ for $i = 0,1,2$.

Consider elements $w \in H^1(Y;\IZ/2)$ and $w_i \in H^1(Y_i;\IZ/2)$ for $i = 1,2$, such
that $H^1(Y;\IZ/2) \to H^1(Y_i;\IZ/2)$ sends $w$ to $w_i$ for $i = 1,2$.  We conclude
from~\eqref{iso_mu} that there are isomorphisms
\begin{equation}
  \mu_i  \colon \bigoplus_{C \in \pi_0(Y_i)} H_d^{\pi_1(C)}(\widetilde{C},\overline{C \cap Y_0};\IZ^{w_i|_C})
  \to  H_d^{\pi}(\overline{Y_i},\overline{Y_0};\IZ^{w})
  \label{iso_mu_i}
\end{equation}
for $i = 1,2$.

By a Mayer-Vietoris argument we see that the map
\[H_d^{\pi}(\overline{Y_2},\overline{Y_0};\IZ^w) \oplus
  H_d^{\pi}(\overline{Y_1},\overline{Y_0};\IZ^w) \xrightarrow{\cong}
  H_d^{\pi}(\widetilde{Y},\overline{Y_0};\IZ^w)
\]
is bijective. Let $u \in H^{\pi}(\widetilde{Y},\IZ^w)$,
$u_1 \in H_d^{\pi}(\overline{Y_1},\overline{Y_0};\IZ^w)$, and
$u_2 \in H_d^{\pi}(\overline{Y_2},\overline{Y_0};\IZ^w)$ be elements such that the
isomorphism above sends $(u_1,u_2)$ to the image of $u$ under the map
$H_d^{\pi}(\widetilde{Y};\IZ^w) \to H_d^{\pi}(\widetilde{Y},\overline{Y_0};\IZ^w)$.  The
next result is a variation of~\cite[Theorem~2.1]{Wall(1967)}.

\begin{lemma}\label{lem:subtracting_a_Poincare_pair}

  \begin{enumerate}
  \item\label{lem:subtracting_a_Poincare_pair:conclusion_Y} Suppose that for $i = 1,2$
    there are elements $[Y_i,Y_0]$ in \linebreak
    $\bigoplus_{C \in \pi_0(Y_i)} H_d^{\pi_1(C)}(\widetilde{C},\overline{C \cap
      Y_0};\IZ^{w_i|_C})$ such that $\mu_i([Y_i,Y_0]) = u_i$ for the isomorphism $\mu_i$
    of~\eqref{iso_mu_i} holds and $(Y_i,Y_0)$ is a finite $d$-dimensional Poincar\'e
    complex with respect to $w_i$ and $[Y_i,Y_0]$ as fundamental class;
    
    Then $Y$ is a finite $d$-dimensional Poincar\'e complex with respect to $w$ and $u$ as
    fundamental class;

  \item\label{lem:subtracting_a_Poincare_pair:conclusion_Y_simple} If we assume in
    assertion~\eqref{lem:subtracting_a_Poincare_pair:conclusion_Y} additionally that
    $(Y_i,Y_0)$ is a simple finite $d$-dimensional Poincar\'e complex for $i = 1,2$, then
    $Y$ is a simple finite $d$-dimensional Poincar\'e complex;

  \item\label{lem:subtracting_a_Poincare_pair:conclusion_Y_2} Assume that the following
    conditions hold:
    \begin{itemize}

    \item $Y$ is a finite $d$-dimensional Poincar\'e complex with respect to $w$ and $u$
      as fundamental class;

    \item There is an element $[Y_1,Y_0] \in \bigoplus_{C \in \pi_0(Y_1)}$ in
      $H_d^{\pi_1(C)}(\widetilde{C},\overline{C \cap Y_0};\IZ^{w_1|_C})$ such that
      $\mu_1([Y_1, Y_0]) = u_1$ holds and $(Y_1,Y_0)$ is a finite $d$-dimensional
      Poincar\'e complex with respect to $w_1$ and $[Y_1,Y_0]$ as fundamental class;
    
    \item For every $y_2 \in Y_2$ the map $\pi_1(Y_2,y_2) \to \pi_1(Y,y_2)$ is injective.

    \end{itemize}
    
    Then $(Y_2,Y_0)$ is a finite $d$-dimensional Poincar\'e pair with respect to $w_2$ and
    $\mu_2^{-1}(u_2)$ as fundamental class;

  \item\label{lem:subtracting_a_Poincare_pair:conclusion_Y_2_simple} If we assume in
    assertion~\eqref{lem:subtracting_a_Poincare_pair:conclusion_Y_2} additionally that
    $(Y_1,Y_0)$ and $Y$ are simple and the map $\Wh(Y_2) \to \Wh(Y)$ is injective, then
    $(Y_2,Y_0)$ is a simple finite $d$-dimensional Poincar\'e pair.
  
  \end{enumerate}
\end{lemma}
\begin{proof}
  Consider the following diagram of $\IZ\pi$-chain complexes
  \[\xymatrix{ 0 \ar[r] & C^{d-*}(\widetilde{Y},\overline{Y_1}) \ar[r] \ar[d]_{\cong} &
      C^{d-*}(\widetilde{Y}) \ar[r] \ar[ddd]^{- \cap u} & C^{d-*}(\overline{Y_1}) \ar[r]
      \ar[dd]^{- \cap u_1} & 0
      \\
      & C^{d-*}(\overline{Y_2},\overline{Y_0}) \ar[dd]^{- \cap u_2} & & &
      \\
      & & & C_*(\overline{Y_1},\overline{Y_0}) \ar[d]^{\cong}
      &\\
      0 \ar[r] & C_*(\overline{Y_2}) \ar[r] & C_*(\widetilde{Y}) \ar[r] &
      C_*(\widetilde{Y},\overline{Y_2}) \ar[r] & 0 }
  \]
  One can arrange the representatives of the $\IZ\pi$-chain maps $- \cap u_1$, $- \cap u$,
  and $- \cap u_2$ such that the diagram commutes. The two rows are based exact sequences
  of pairs.  The two arrows marked with $\cong$ are the base preserving isomorphisms given
  by excision.  We conclude that all three $\IZ\pi$-chain maps $- \cap u_1$, $- \cap u$,
  and $- \cap u_2$ are (simple) $\IZ\pi$-chain homotopy equivalences, if and only if two of
  them are.  Provided that for every $y_2 \in Y_2$ the map
  $\pi_1(Y_2,y_2) \to \pi_1(Y,y_2)$ is injective, then $(Y_2,Y_0)$ is a Poincar\'e pair, if
  and only if
  $- \cap u_2 \colon C^{d-*}(\overline{Y_2},\overline{Y_0}) \to C_*(\overline{Y_2})$ is a
  $\IZ\pi$-chain homotopy equivalence

  Now the claims follows from the conderations appearing in
  Subsection~\ref{subsec:arbitrary_coverings}.
\end{proof}

    
\subsection{Special Poincar\'e complexes}\label{subsec:Special_Poincare_complexes}

\begin{definition}[Special Poincar\'e pair]\label{def:special_Poincare_pair}
  Let $(X,\partial X)$ be a finite (simple) Poincar\'e pair of dimension $n \ge5$.  It is
  called \emph{special}%
\index{special Poincar\'e pair}\index{Poincare pair@Poincar\'e pair!special} 
if there  exists
  \begin{itemize}

  \item An $n$-dimensional compact smooth manifold $H$ with boundary $\partial H$ such that the
    inclusion $i \colon \partial H \to H$ induces an epimorphism
    $\pi_0(\partial H) \to \pi_0(H)$, and for every component $D$ of $H$ an epimorphism
    $\ast_{\substack{C \in \pi_0(\partial H)\\ i(C) \subseteq D}}\; \pi_1(C) \to
    \pi_1(D)$;

  \item A finite $CW$-sub complex $ \widehat{X}$ of $X$ containing $\partial X$
    such that $\widehat{X} \setminus \partial X$ contains only cells of dimension $\le (n-2)$;

  \item A cellular map $z \colon \partial H \to \widehat{X}$, which induces a bijection on $\pi_0$
    and for every choice of base point in $\partial H$ an epimorphism on $\pi_1$,
    where we consider $(H,\partial H)$ as a simple $CW$-pair by a smooth triangulation,

\end{itemize}

such that $X$ is the pushout
\[
  \xymatrix{\partial H \ar[r]^{z} \ar[d] & \widehat{X} \ar[d]
    \\
    H \ar[r] & X = H \cup_{z} \widehat{X}}
\]
\end{definition}

\begin{lemma}\label{lem:H_inside_a_Poincare_pair}
  Let $(X,\partial X)$ be an $n$-dimensional finite (simple) Poincar\'e pair of dimension
  $n \ge5$.  Then there exists a special Poincar\'e pair $(X',\partial X')$ with
  $\partial X = \partial X'$ together with a (simple) homotopy equivalence
  $g \colon X \to X'$ inducing the identity on $\partial X$.
\end{lemma}
\begin{proof}
  See~\cite[Lemma~2.8 on page~30 and the following 
  paragraph]{Wall(1999)}.
\end{proof}


\typeout{-------------------- Section 6: Fixing the orientation homomorphisms and the
  fundamental classes --------}

\section{Fixing the orientation homomorphisms and the fundamental classes}~%
\label{sec:Fixing_the_orientation_homomorphisms_and_the_fundamental_classes}


\subsection{Transfer}\label{subsec:Transfer}

In this subsection we recall the notion of a transfer. Consider a $\IZ\Gamma$-chain
complex $C_*$ and a right $\IZ\Gamma$-module $M$.  We make no assumptions about $C_*$ of
the kind that its chain modules are projective or finitely generated.  Let $i^*C_*$ and
$i^*M$ be the $\IZ\pi$-chain complex and $\IZ \pi$-module obtained by restriction with
the inclusion $i \colon \pi \to \Gamma$.

Define $\IZ$-chain maps
\begin{equation}
  i_* \colon i^*M \otimes_{\IZ \pi} i^*C_*\to M \otimes_{\IZ \Gamma} C_*
  \quad m \otimes x \mapsto m \otimes x
  \label{i_ast_for_C_ast_and_M}
\end{equation}
and
\begin{equation}
  \trf_*  \colon M \otimes_{\IZ \Gamma} C_* \to i^*M \otimes_{\IZ \pi} i^*C_*,
  \quad m \otimes x \mapsto \sum_{g \in G} m\widehat{g} \otimes \widehat{g}^{-1}  x,
  \label{trf_ast_for_C_ast_and_M}
\end{equation}
where $\widehat{g}$ is any element in $\Gamma$, which is mapped under the projection
$\Gamma \to G$ to $g$.  The definition~\eqref{trf_ast_for_C_ast_and_M} is independent of
the choice of $\widehat{g}$, since for $\omega \in \pi$, $m \in M$ and $x \in C_*$ we get
in $i^*M \otimes_{\IZ \pi} i^*C_*$
\[
  m\widehat{g}\omega \otimes (\widehat{g}\omega)^{-1} x = m\widehat{g}\omega \otimes
  \omega^{-1}\widehat{g}^{-1} = m\widehat{g} \otimes \widehat{g}^{-1}x.
\]
One easily checks that both definitions are compatible with the tensor relations and
define indeed $\IZ$-chain maps.  Moreover, $i_*$ and $\trf_*$ are natural in both $C_*$
and $M$ and satisfy
\begin{equation}
  i_* \circ \trf_* = |G| \cdot \id_{M \otimes_{\IZ \Gamma} C_*}
  \label{i_ast_circ_trf_ast_is_|g|_cdot_id}
\end{equation}

Applying this to the $\IZ\Gamma$-chain complexes $C_*(X,\partial X)$ for a slice complement model
$(X,\partial X)$, to $C_*(E\Gamma,\partial E\Gamma)$, to
$C_*(\eub{\Gamma},\partial \eub{\Gamma})$ and to $C_*(\eub{\Gamma})$ and take the $d$-th
homology, we obtain the following commutative diagram of $\IZ$-modules, whose vertical
arrows are the obvious maps
\begin{equation}
  \xymatrix@!C=10em{H_d^{\Gamma}(X,\partial X;\IZ^w)
    \ar[r]^-{H_d(\trf_*)}  \ar[d]^{\cong}
    &
    H_d^{\pi}(i^*X,i^*\partial X;\IZ^v)
    \ar[r]^-{H_d(i_*)}  \ar[d]^{\cong}
    &
    H_d^{\Gamma}(X,\partial X;\IZ^w)  \ar[d]^{\cong}
    \\
    H_d^{\Gamma}(E\Gamma,\partial E\Gamma;\IZ^w)
    \ar[r]^-{H_d(\trf_*)}  \ar[d]^{\cong}
    &
    H_d^{\pi}(i^*E\Gamma,i^*\partial E\Gamma;\IZ^v)
    \ar[r]^-{H_d(i_*)}  \ar[d]^{\cong}
    &
    H_d^{\Gamma}(E\Gamma,\partial X;\IZ^w)  \ar[d]^{\cong}
    \\
    H_d^{\Gamma}(\eub{\Gamma},\partial \eub{\Gamma};\IZ^w)
    \ar[r]^-{H_d(\trf_*)}  
    &
    H_d^{\pi}(i^*\eub{\Gamma},i^*\partial \eub{\Gamma};\IZ^v)
    \ar[r]^-{H_d(i_*)}  
    &
    H_d^{\Gamma}(\eub{\Gamma},\partial \eub{\Gamma};\IZ^w)  \ar[d]^{\cong}
    \\
    H_d^{\Gamma}(\eub{\Gamma};\IZ^w)
    \ar[r]^-{H_d(\trf_*)}  \ar[u]_{\cong}
    &
    H_d^{\pi}(i^*\eub{\Gamma};\IZ^v)
    \ar[r]^-{H_d(i_*)}  \ar[u]_{\cong}
    &
    H_d^{\Gamma}(\eub{\Gamma};\IZ^w)  \ar[u]_{\cong}
  }
  \label{diagram_of_the_H_d-s}
\end{equation}
such that in each row the composite of the two horizontal maps is $|G| \cdot \id$.  The
vertical arrows from the second row to the third row and the composite of the arrows from
the first row to the second row with the arrows from the second row to the third row are all
bijective by excision applied to the homotopy $\Gamma$-pushouts appearing in
Proposition~\ref{pro:constructing_underline(E)Gamma)} and
in~\eqref{canonical_homotopy_Gamma_pushout_for_a_slice_model}. The horizontal arrows
from the fourth row to the third row are bijective, since $\partial \eub{\Gamma}$ is
zero-dimensional.  Hence all vertical arrows are bijective.


\subsection{Some necessary conditions}\label{subsec:Some_necessary_conditions}

In this subsection we assume that we have a slice complement model $(X,\partial X)$ for
$\eub{\Gamma}$ with respect to the free $d$-dimensional slice system $\cals$ such that the
quotient space $(X/\Gamma,\partial X\Gamma)$ carries the structure of a finite Poincar\'e
pair. Recall from Remark~\ref{rem:extension_of_the_orientation_homomorphism_for_pi} that
there is only one choice possible for the orientation homomorphism
$w \colon \Gamma \to \{\pm 1\}$. Namely, the abelian group
$H_0(C^{d-*}_{\untw}(X,\partial X))$ must be infinite cyclic and as a $\IZ\Gamma$-module
it must be isomorphic to $\IZ^w$.  Since $(X,\partial X)$ is a slice complement model for
$\eub{\Gamma}$, the projection
$\pr \colon (X,\partial X) \to (\eub{\Gamma},\partial \eub{\Gamma})$ is a
$\IZ \Gamma$-chain homotopy equivalence because of the homotopy
$\Gamma$-pushout~\eqref{canonical_homotopy_Gamma_pushout_for_a_slice_model} and induces an
$\IZ\Gamma$-isomorphism
$H_0(C^{d-*}_{\untw}(X,\partial X)) \xrightarrow{\cong_{\IZ\Gamma}}
H_0(C^{d-*}_{\untw}(\eub{\Gamma},\partial \eub{\Gamma}))$.  Since $\partial \eub{\Gamma}$
is zero-dimensional, the obvious inclusion induces an isomorphism of $\IZ\Gamma$-modules
$H_0(C^{d-*}_{\untw}(\eub{\Gamma},\partial \eub{\Gamma})) \xrightarrow{\cong}
H_0(C^{d-*}_{\untw}(\eub{\Gamma}))$.  Hence we obtain an isomorphism of
$\IZ\Gamma$-modules
\[
  H_0(C^{d-*}_{\untw}(X,\partial X)) \xrightarrow{\cong_{\IZ\Gamma}}
  H_0(C^{d-*}_{\untw}(\eub{\Gamma})).
\]
If we define $w \colon \Gamma \to \{\pm \}$ by requiring that the $\IZ\Gamma$-module
$H_0(C^{d-*}_{\untw}(\eub{\Gamma}))$ is isomorphic to $\IZ^w$, then $w$ is defined in
terms of $\Gamma$ only and has to be the orientation homomorphism for any structure of a
finite Poincar’e pair on $(X/\Gamma,\partial X/\Gamma)$ for any slice complement model
$(X,\partial X)$.

Note that for a finite Poincar\'e pair $(X/\Gamma,\partial X/\Gamma)$ there is the
induced structure of a finite Poincare complex on $\partial X/\Gamma$ and the orientation
homomorphism for $\partial X/\Gamma$ is obtained by the one for
$(X/\Gamma,\partial X/\Gamma)$ by restriction. Recall that $\partial X/\Gamma$ is
$\coprod_{F \in \calm} S_F/F$ and that we have figured out the first Stiefel-Whitney
class for $S_F/F$ in Lemma~\ref{lem:basics_about_free_slice_systems}.

Since $(X/\Gamma,\partial X/\Gamma)$ is a finite $d$-dimensional Poincar\'e pair, the same
is true for $(X/\pi,\partial X /\pi)$ by~\cite[Theorem~H]{Klein-Qin_Su(2019)} and the
homomorphisms of infinite cyclic groups
$H_d(\trf_*) \colon H_d^{\Gamma}(X,\partial X;\IZ^w) \to H_d^{\pi}(i^*X,i^*\partial
X;\IZ^w)$ sends the fundamental classes to one another and hence is bijective.  We
conclude from~\eqref{diagram_of_the_H_d-s} that the map
$H_d(\trf_*) \colon H_d^{\Gamma}(\eub{\Gamma};\IZ^w) \xrightarrow{\cong}
H_d^{\pi}(i^*\eub{\Gamma};\IZ^w)$ is bijective. Since $i^*\eub{\Gamma}$ is a model for
$B\pi$, we conclude from
Lemma~\ref{lem:subtracting_a_Poincare_pair}~\eqref{lem:subtracting_a_Poincare_pair:conclusion_Y}
applied to the restriction to $\pi$ of the $\Gamma$-pushout
\[
  \xymatrix{\coprod_{F \in \calm} \Gamma\times_F S_F \ar[r] \ar[d] & X \ar[d]
    \\
    \coprod_{F \in \calm} \Gamma\times_F D_F\ar[r] & \eub{\Gamma}.}
\]
and Remark~\ref{rem:extension_of_the_orientation_homomorphism_for_pi} that the restriction
of $w$ to $\pi$ must be the orientation homomorphism $v \colon \pi \to \{\pm 1\}$ of
$B\pi$. To summarize, we have the following necessary conditions for the existence of a
slice complement model $(X,\partial X)$ such that the quotient space $(X/\Gamma,\partial X/\Gamma)$
carries the structure of a finite Poincar\'e pair:

\renewcommand{\labelenumi}{(\roman{enumi})}

\begin{enumerate}

\item The abelian group $H_0(C^{d-*}_{\untw}(\eub{\Gamma}))$ is infinite cyclic. Let
  $w \colon \Gamma \to \{\pm 1\}$ be the homomorphisms uniquely determined by the property
  that the $\IZ\Gamma$-module $H_0(C^{d-*}_{\untw}(\eub{\Gamma}))$ is
  $\IZ\Gamma$-isomorphic to $\IZ^w$.

\item There is a finite $d$-dimensional Poincar\'e $CW$-complex model for $B\pi$ with
  respect to the orientation homomorphisms $v \colon \pi \to \{\pm 1 \}$;

\item We have $v = w|_{\pi}$;

\item Consider any $F \in \calm$. The restriction of $w$ to $F$ is trivial, if $d$ is even.
If $d$ is odd, $F \cong \IZ/2$ and restriction of $w$ to $F$ is non-trivial;

\item The transfer map
  \begin{eqnarray*}
    H_d(\trf_*) \colon H_d^{\Gamma}(\eub{\Gamma};\IZ^w)
    & \xrightarrow{\cong}  &
                             H_d^{\pi}(i^*\eub{\Gamma};\IZ^v) = H_d^{\pi}(E\pi;\IZ^v)
  \end{eqnarray*}
  is bijective.

\end{enumerate}

\renewcommand{\labelenumi}{(\arabic{enumi})}

These necessary conditions motivate the material of the next subsections.


\subsection{The orientation homomorphism}%
\label{subsec:The_orientation_homomorphism}

For the remainder of this section we will make the following assumptions

\begin{assumption}\label{ass:orientation_homomorphism}\

\begin{itemize}
\item There exists a finite $\Gamma$-$CW$-model for $\eub{\Gamma}$ of dimension $d$ such
  that its singular $\Gamma$-subspace $\eub{\Gamma}^{> 1}$ is
  $\coprod_{F\in\calm} \Gamma/F$;

\item There is a finite $d$-dimensional Poincar\'e $CW$-complex model for $B\pi$ with
  respect to the orientation homomorphisms $v \colon \pi \to \{\pm 1 \}$ and fundamental
  class $[B\pi] \in H_d^{\pi}(E\pi;\IZ^v)$.

\end{itemize}

\end{assumption}

Let $\trunc \colon \IZ \Gamma \to \IZ \pi$ be the homomorphisms of
$\IZ\pi$-$\IZ \pi$-bimodules, which sends
$\sum_{\gamma \in \Gamma} \lambda_{\gamma} \cdot \gamma$ to
$\sum_{\gamma \in \pi} \lambda_{\gamma} \cdot \gamma$.  Consider $\gamma_0 \in
\Gamma$. Let $r_{\gamma_0^{-1}} \colon \IZ\Gamma \to \IZ \Gamma$ be the $\IZ\Gamma$
automorphism of left $\IZ \Gamma$-modules sending
$\sum_{\gamma} \lambda_{\gamma} \cdot \gamma$ to
$\sum_{\gamma} \lambda_{\gamma} \cdot \gamma \gamma_0^{-1}$.  Denote by
$l_{\gamma_0^{-1}} \colon M \to M $ the automorphism of abelian groups sending $u$ to
$\gamma_0^{-1} u$.  Let $c_{\gamma_0} \colon \pi \xrightarrow{\cong} \pi$ be the group
automorphism sending $u$ to $\gamma_0u\gamma_0^{-1}$.  Denote by
$\IZ[c_{\gamma_0}] \colon \IZ \pi \xrightarrow{\cong} \IZ \pi$ the induced ring
automorphism. Let $M$ be a $\IZ\Gamma$-module.  We get by composition and precomposition
maps of abelian groups
\begin{eqnarray*}
  \trunc_* \colon \hom_{\IZ \Gamma}(M,\IZ\Gamma)
  & \to &
          \hom_{\IZ \pi}(i^*M,\IZ\pi);
  \\
  (r_{\gamma_0^{-1}})_* \colon \hom_{\IZ \Gamma}(M,\IZ\Gamma)
  & \to &
          \hom_{\IZ \Gamma}(M,\IZ\Gamma);
  \\
  \IZ[c_{\gamma_0}]_* \circ (l_{\gamma_0^{-1}})^* \colon   \hom_{\IZ \pi}(i^*M,\IZ\pi)
  & \to &
          \hom_{\IZ \pi}(i^*M,\IZ\pi).   
\end{eqnarray*}

The map $\IZ[c_{\gamma_0}]_* \circ (l_{\gamma_0^{-1}})^*$ is indeed well-defined as the
following calculation shows for $f \in \hom_{\IZ \pi}(i^*M,\IZ\pi)$, $\omega \in \pi$, and
$x \in M$
\begin{eqnarray*}
  \bigl(\IZ[c_{\gamma_0}]_* \circ (l_{\gamma_0^{-1}})^*(f)\bigr)(\omega x)
  & = &
        \IZ[c_{\gamma_0}](f(\gamma_0^{-1}\omega x))
  \\
  & = &
        \IZ[c_{\gamma_0}](f(\gamma_0^{-1}\omega \gamma_0\gamma_0^{-1}x))
  \\
  & = &
        \IZ[c_{\gamma_0}](\gamma_0^{-1}\omega \gamma_0 \cdot f(\gamma_0^{-1}x))
  \\
  & = &
        \IZ[c_{\gamma_0}](\gamma_0^{-1}\omega \gamma_0) \cdot \IZ[c_{\gamma_0}](f(\gamma_0^{-1}x))       
  \\
  & = &
        \omega\cdot \bigl(\IZ[c_{\gamma_0}]_* \circ (l_{\gamma_0^{-1}})^*(f)\bigr)(x).
\end{eqnarray*}
  
\begin{lemma}\label{lem:trunc}
  \begin{enumerate}
  \item\label{lem:trunc:bijective} The map
    $\trunc_* \colon \hom_{\IZ \Gamma}(M,\IZ\Gamma) \xrightarrow{\cong} \hom_{\IZ
      \pi}(i^*M,\IZ\pi)$ is bijective;
  \item\label{lem:trunc:diagram}
  
    The following diagram commutes
    \[
      \xymatrix@!C=10em{\hom_{\IZ \Gamma}(M,\IZ\Gamma) \ar[r]^-{\trunc_*}_-{\cong}
        \ar[d]_{(r_{\gamma_0^{-1}})_*} & \hom_{\IZ \pi}(i^*M,\IZ\pi)
        \ar[d]^{\IZ[c_{\gamma_0}]_* \circ (l_{\gamma_0^{-1}})^*}
        \\
        \hom_{\IZ \Gamma}(M,\IZ\Gamma) \ar[r]_-{\trunc_*}^-{\cong} & \hom_{\IZ
          \pi}(i^*M,\IZ\pi).  }
    \]
  
  \end{enumerate}
\end{lemma}
\begin{proof}~\eqref{lem:trunc:bijective} Since $\pi$ has finite index in $\Gamma$, the
  coinduction $i_!\IZ \pi$ of $\IZ\pi$ is $\IZ\Gamma$-isomorphic to $\IZ\Gamma$,
  see~\cite[Proposition~5.8 in III.5 on page~70]{Brown(1982)}.  One has the adjunction
  isomorphism $(i^*,i_!)$, see~\cite[(3.6) in III.3 on page~64]{Brown(1982)}
  \[
    \hom_{\IZ \Gamma}(M,\IZ\Gamma) \xrightarrow{\cong} \hom_{\IZ \pi}(i^*M,\IZ\pi).
  \]
  One easily checks by going through the definitions, that this isomorphism is the map
  $\trunc_*$.  \\[1mm]~\eqref{lem:trunc:diagram} This follows from the following
  computation for $f \in \hom_{\IZ \Gamma}(M,\IZ\Gamma)$ and $x \in M$.  Write
  $f(x) = \sum_{\gamma \in \Gamma} \lambda_{\gamma} \cdot \gamma$.  Then we get
  $f(x) \gamma_0^{-1} = \sum_{\gamma \in \Gamma} \lambda_{\gamma} \cdot \gamma
  \gamma_0^{-1}$ and
  $f(\gamma_0^{-1} x) = \sum_{\gamma \in \Gamma} \lambda_{\gamma} \cdot \gamma_0^{-1}
  \gamma$. This implies
  \begin{eqnarray*}
    \bigl(\trunc_* \circ (r_{\gamma_0^{-1}})_*(f)\bigr)(x)
    & = & 
          \sum_{\substack{\gamma \in \Gamma\\\gamma \gamma_0^{-1} \in \pi}} \lambda_{\gamma} \cdot \gamma \gamma_0^{-1};
    \\
    \bigl((l_{\gamma_0^{-1}})^* \circ \trunc_*(f)\bigr)(x)
    & = & 
          \sum_{\substack{\gamma \in \Gamma\\ \gamma_0^{-1} \gamma \in \pi}} \lambda_{\gamma} \cdot \gamma_0^{-1} \gamma.
  \end{eqnarray*}
  Now the claim follows from the computation
  \begin{eqnarray*}
    \bigl(\IZ[c_{\gamma_0}]\circ l_{\gamma_0^{-1}}^* \circ \trunc_*(f)\bigr)(x)
    & = & 
          \IZ[c_{\gamma_0}]\biggl(\sum_{\substack{\gamma \in \Gamma\\\gamma_0^{-1} \gamma \in \pi}}
    \lambda_{\gamma} \cdot \gamma_0^{-1} \gamma\biggr)
    \\
    & = &
          \sum_{\substack{\gamma \in \Gamma\\\gamma_0^{-1} \gamma \in \pi}} \lambda_{\gamma} \cdot \gamma  \gamma_0^{-1}
    \\
    & = &
          \sum_{\substack{\gamma \in \Gamma\\\gamma \gamma_0^{-1} \in \pi}} \lambda_{\gamma} \cdot \gamma  \gamma_0^{-1}
    \\
    & = &
          \bigl(\trunc_* \circ (r_{\gamma_0^{-1}})_*(f)\bigr)(x).
  \end{eqnarray*}
\end{proof}

Note that Lemma~\ref{lem:trunc}~\eqref{lem:trunc:bijective} implies that the abelian group
underlying the $\IZ\pi$-module $H_0(i^*C^{d-*}_{\untw}(\eub{\Gamma}))$ is infinite
cyclic, since $H_0(C^{d-*}_{\untw}(E\pi))$ is infinite cyclic.  Hence the following
notation makes sense

\begin{notation}\label{not:w} The $\Gamma$-homomorphism $w \colon \Gamma \to \{\pm 1\}$ is
  uniquely determined by the property that the $\IZ\Gamma$-module
  $H_0(C^{d-*}_{\untw}(\eub{\Gamma}))$ is $\IZ\Gamma$-isomorphic to $\IZ^w$.
\end{notation}

\begin{lemma}\label{lem:W_pi_is_v}
  The restriction $i^*w$ of $w$ to $\pi$ is $v \colon \pi \to \{\pm 1\}$.
\end{lemma}
\begin{proof}
  If $\gamma_0$ belongs to $\pi$, then the following diagram commutes
  \[
    \xymatrix@!C=10em{\hom_{\IZ \Gamma}(M,\IZ\Gamma) \ar[r]^-{\trunc_*}_-{\cong}
      \ar[d]_{(r_{\gamma_0^{-1}})_*} & \hom_{\IZ \pi}(i^*M,\IZ\pi)
      \ar[d]^{(r_{\gamma_0^{-1}})_*}
      \\
      \hom_{\IZ \Gamma}(M,\IZ\Gamma) \ar[r]_-{\trunc_*}^-{\cong} & \hom_{\IZ
        \pi}(i^*M,\IZ\pi) }
  \]
  by Lemma~\ref{lem:trunc}~\eqref{lem:trunc:diagram}, since in this special case
  $\IZ[c_{\gamma_0}]_* \circ (l_{\gamma_0^{-1}})^*$ reduces to $(r_{\gamma_0^{-1}})_*$.
  Recall that the free $\pi$-$CW$-complex $i^*\eub{\Gamma}$ is a model for $E\pi$.  Hence
  we get an isomorphism of $\IZ\pi$-chain complexes
  \[
    i^*C^{d-*}_{\untw}(\eub{\Gamma}) \xrightarrow{\cong_{\IZ \pi}} C^{d-*}_{\untw}(E\pi),
  \]
  which induces an isomorphism of $\IZ\pi$-modules
  \[
    H_0(i^*C^{d-*}_{\untw}(\eub{\Gamma}))\xrightarrow{\cong_{\IZ \pi}}
    H_0(C^{d-*}_{\untw}(E\pi)).
  \]
  This implies $ w|_{\pi} = v$.
\end{proof}

Next we want to express $w$ in terms of $\pi$, where we take only the $\Gamma$-action on $\pi$ by
conjugation into account and do not refer to $\eub{\Gamma}$.

Fix $\gamma_0 \in \Gamma$. The group automorphism
$c_{\gamma_0^{-1}} \colon \pi \xrightarrow{\cong} \pi$ induces a $\pi$-homotopy
equivalence $Ec_{\gamma_0^{-1}} \colon E\pi \to c_{\gamma_0^{-1}}^*E\pi$, where
$\omega \in \pi$ acts on $c_{\gamma_0^{-1}}^*E\pi$ by the action of
$\gamma_0^{-1}\omega \gamma_0$ on $E\pi$.  If $\gamma_0$ belongs to $\pi$, then
$Ec_{\gamma_0^{-1}}$ is $\pi$-homotopic to the map
$l_{\gamma_0^{-1}}  \colon E\pi \to c_{\gamma_0^{-1}}^*E\pi$ given by left
multiplication with $\gamma_0^{-1}$.  By assumption there is a finite Poincar\'e structure
on $B\pi$ with respect to the orientation homomorphism $v \colon \pi \to \{\pm 1\}$
and fundamental class $[B\pi] \in H_d^{\pi}(E\pi;\IZ^v)$. The $\pi$-homotopy
equivalence $Ec_{\gamma_0^{-1}}$ induces an isomorphism of abelian groups
\begin{equation}
  H_d^{\pi}(Ec_{\gamma_0^{-1}};\IZ^v) \colon H_d^{\pi}(E\pi;\IZ^v) \xrightarrow{\cong_{\IZ}}  H_d^{\pi}(c_{\gamma_0^{-1}}^*E\pi;\IZ^v).
  \label{H_d_upper_pi(Ec_(gamma_0_upper_1);Z_upper_v)}
\end{equation}
There is an obvious isomorphism of abelian groups
\begin{equation}
  a' \colon H_d^{\pi}(c_{\gamma_0^{-1}}^*E\pi;\IZ^{c_{\gamma_0^{-1}}^*(v)}) \xrightarrow{\cong} H_d^{\pi}(E\pi;\IZ^v)
  \label{identification_of_homology_a'}
\end{equation}
coming from the identification
\[
  \IZ^v \otimes_{\IZ \pi} C_*(E\pi) \xrightarrow{\cong} \IZ^{c_{\gamma_0^{-1}}^*(v)}
  \otimes_{\IZ \pi} C_*(c_{\gamma_0^{-1}}^* E\pi), \quad n \otimes x \mapsto n \otimes x.
\]
We get
\begin{equation}
  (c_{\gamma_0}^{-1})^* (v) = v
  \label{c_(gamma_upper_-1)(v)_is_v}
\end{equation}
from the following calculation for $\omega \in \pi$
\[
  (c_{\gamma_0}^{-1})^* (v)(\omega) = v(\gamma_0^{-1} \omega \gamma_0) = w(\gamma_0^{-1}
  \omega \gamma_0)
  \\
  = w(\gamma_0^{-1}) \cdot w(\omega) \cdot w(\gamma_0) = w(\omega) = v(\omega).
\]
Putting~\eqref{identification_of_homology_a'} and~\eqref{c_(gamma_upper_-1)(v)_is_v}
together yields an isomorphism of abelian groups
\begin{equation}
  a \colon H_d^{\pi}(c_{\gamma_0^{-1}}^*E\pi;\IZ^v) \xrightarrow{\cong} H_d^{\pi}(E\pi;\IZ^v).
  \label{identification_of_homology_a}
\end{equation}
We obtain an automorphism
\begin{equation}
  a \circ H_d^{\pi}(Ec_{\gamma_0^{-1}};\IZ^v) \colon H_d^{\pi}(E\pi;\IZ^v) \xrightarrow{\cong} H_d^{\pi}(E\pi;\IZ^v)
  \label{auto_a_circ_H_d_upper_pi(Ec_(gamma_0_upper_1);Z_upper_v)}
\end{equation}
of an infinite cyclic group $H_d^{\pi}(Ec_{\gamma_0^{-1}};\IZ^v)$ with the generator
$[B\pi]$ by composing the
isomorphisms~\eqref{H_d_upper_pi(Ec_(gamma_0_upper_1);Z_upper_v)}
and~\eqref{identification_of_homology_a}.  Now define
\begin{equation}
  u(\gamma_0) \in \{\pm 1\}
  \label{u(gamma_0)}
\end{equation}
by the equation
\[
  a \circ H_d^{\pi}(c_{\gamma_0^{-1}}^*E\pi;\IZ^v) ([B\pi]) = u(\gamma_0) \cdot
  [B\pi].
\]

\begin{lemma}\label{lem:w(gamma_0)_is_u(gamma_0)}
  We have $ w(\gamma_0) = u(\gamma_0)$ for every $\gamma_0 \in \Gamma$.
\end{lemma}
\begin{proof}
  Let $[c_{\gamma_0}^*E\pi] \in H_d^{\pi}(c_{\gamma_0^{-1}}^*E\pi;\IZ^v)$ be the image of
  $[B\pi]$ under $H_d^{\pi}(Ec_{\gamma_0^{-1}};\IZ^v)$. Then we get a diagram of
  $\IZ\pi$-chain complexes which commutes up to $\pi$-homotopy
  \begin{equation}
    \xymatrix@!C=12em{C^{d-*}(E\pi) \ar[d]^{\simeq_{\pi}}_{- \cap[B\pi]}
      &
      C^{d-*}(c_{\gamma_0^{-1}}^*E\pi)
      \ar[l]_{C^{d-*}(Ec_{\gamma_0^{-1}})}
      \ar[d]_{\simeq_{\pi}}^{- \cap [B\pi]}
      \\
      C_*(E\pi)
      \ar[r]_{C_*(Ec_{\gamma_0^{-1}})} &
      C_*(c_{\gamma_0^{-1}}^*E\pi)
    }
    \label{diagram_for_epi_and_c_(gamma_0_upper_1)_Epi_1}
  \end{equation}
  where the dual $\IZ\pi$-chain complexes are taken with respect to the $v$-twisted
  involution and $B\pi$ on the left side corresponds to $E\pi/\pi$ and on the right side
  to $(c_{\gamma_0^{-1}}E\pi)/\pi$

  The following diagram of $\IZ\pi$-chain complexes commutes up to $\IZ\pi$-chain homotopy
  \begin{equation}
    \xymatrix@!C=12em{C^{d-*}(c_{\gamma_0^{-1}}^*E\pi)
      \ar[d]^{\simeq_{\IZ\pi}}_{- \cap[B\pi]}
      &
      c_{\gamma_0^{-1}}^* C^{d-*}(E\pi) \ar[d]^{c_{\gamma_0^{-1}}^* (-\cap a([B\pi]))}_{\simeq_{\IZ\pi}}
      \ar[l]_{\IZ[c_{\gamma_0}]_*}
      \\
      C_*(c_{\gamma_0^{-1}}^*  E\pi) = c_{\gamma_0^{-1}}^* C_*(E\pi) \ar[r]_{=}
      &
      c_{\gamma_0^{-1}}^* C_*(E\pi)
    }
    \label{diagram_for_epi_and_c_(gamma_0_upper_1)_Epi_2}
  \end{equation}
  where $\IZ[c_{\gamma_0^{-1}}]_*$ is given by composing with the ring automorphism
  $\IZ[c_{\gamma_0^{-1}}] \colon \IZ \pi \xrightarrow{\cong} \IZ \pi$ induced by the group
  automorphism $c_{\gamma_0^{-1}} \colon \pi \xrightarrow{\cong} \pi$.

  By putting~\eqref{u(gamma_0)},~\eqref{diagram_for_epi_and_c_(gamma_0_upper_1)_Epi_1},
  and~\eqref{diagram_for_epi_and_c_(gamma_0_upper_1)_Epi_2} together and using the equality
  $\IZ[c_{\gamma_0}]_* \circ C^{d-*}(Ec_{\gamma_0^{-1}}) = C^{d-*}(Ec_{\gamma_0^{-1}}) \circ \IZ[c_{\gamma_0}]_*$,
  yields a commutative diagram of $\IZ\pi$-chain complex
  which commutes up to $\IZ \pi$-chain homotopy
  \begin{equation}
    \xymatrix@!C=16em{C^{d-*}(E\pi) \ar[d]^{\simeq_{\IZ\pi}}_{- \cap[B\pi]}
      &
      c_{\gamma_0^{-1}}^* C^{d-*}(E\pi)
      \ar[l]_-{\IZ[c_{\gamma_0}]_* \circ C^{d-*}(Ec_{\gamma_0^{-1}})}
      \ar[d]_{\simeq_{\IZ\pi}}^{u(\gamma_0) \cdot  c_{\gamma_0^{-1}}^*(- \cap [B\pi])}
      \\
      C_*(E\pi)
      \ar[r]_-{C_*(Ec_{\gamma_0^{-1}})} &
      c_{\gamma_0^{-1}}^* C_*(E\pi) = C_*(c_{\gamma_0^{-1}}^* E\pi).
    }
    \label{diagram_for_epi_and_c_(gamma_0_upper_1)_Epi_3}
  \end{equation}
  If we apply $H_0$ to the diagram above, we obtain a commutative diagram of
  $\IZ\pi$-modules taking into account that the $\IZ\pi$-module $H_0(C_*(E\pi))$ is
  isomorphic to the $\IZ\pi$-module $\IZ$ given by $\IZ$ equipped with the trivial
  $\pi$-action and the map $H_0(C_*(Ec_{\gamma_0^{-1}})) \colon H_0(E\pi) \to H_0(E\pi)$
  is the identity
  \begin{equation}
    \xymatrix@!C=16em{H_0(C^{d-*}(E\pi)) \ar[d]^{\simeq_{\IZ\pi}}_{H_0(- \cap[B\pi])}
      &
      c_{\gamma_0^{-1}}^* H_0(C^{d-*}(E\pi))
      \ar[l]_-{H_0(\IZ[c_{\gamma_0}]_* \circ C^{d-*}(Ec_{\gamma_0^{-1}}))}
      \ar[d]_{\simeq_{\IZ\pi}}^{u(\gamma_0) \cdot  c_{\gamma_0^{-1}}^*H_0(- \cap [B\pi])}
      \\
      \IZ
      \ar[r]_-{\id_{\IZ}} &
      \IZ.
    }
    \label{diagram_for_epi_and_c_(gamma_0_upper_1)_Epi_4}
  \end{equation}
  
  This implies that the map of infinite cyclic groups
  \[H_0(\IZ[c_{\gamma_0}]_* \circ C^{d-*}(Ec_{\gamma_0^{-1}})) \colon
    H_0(C^{d-*}(E\pi)) \xrightarrow{\cong} H_0(C^{d-*}(E\pi))
  \]
  is multiplication with $u(\gamma_0)$. 
  
  The $\Gamma$-map
  $l_{\gamma_0^{-1}} \colon \eub{\Gamma} \to c_{\gamma_0^{-1}}^*\eub{\Gamma}$ is after restriction with $i$
  $\pi$-homotopic to the $\pi$-map $Ec_{\gamma_0^{-1}} \colon E\pi \to c_{\gamma_0^{-1}}^* E\pi$
  taking into account that $i^*\eub{\Gamma}$ is a model for $E\pi$. 
  Hence the following diagram of
  $\IZ$-chain complexes commutes up to $\IZ$-chain homotopy
  \[
    \xymatrix@!C=13em{C^{d-*}(\eub{\Gamma})
      \ar[r]^-{\IZ[c_{\gamma_0}]_* \circ C^{d-*}(l_{\gamma_0^{-1}})} \ar[d]_{\cong}
      &
      C^{d-*}(\eub{\Gamma}) \ar[d]^{\cong}
      \\
      C^{d-*}(E\pi) \ar[r]_-{\IZ[c_{\gamma_0}]_* \circ C^{d-*}(Ec_{\gamma_0^{-1}})}
      &
      C^{d-*}(E\pi)
    }
  \]
  where both  vertical arrows are the $\IZ$-chain isomorphism coming
  from Lemma~\ref{lem:trunc}~\eqref{lem:trunc:bijective}. The upper arrow
  is  multiplication with $\gamma_0$ on $H_0(C^{d-*}_{\untw}(\eub{\Gamma}))$
  by Lemma~\ref{lem:trunc}~\eqref{lem:trunc:diagram}. Hence it induces on $H_0$
  multiplication with $w(\gamma_0)$ by definition.  We have already shown that the lower arrow induces
  on $H_0$ multiplication with $u(\gamma_0)$. Hence  $w(\gamma_0) = u(\gamma_0)$.
\end{proof}

The Hochschild-Serre spectral sequence applied to the group extension
$1 \to \pi \xrightarrow{i} \Gamma \xrightarrow{p} G \to 1$ and the $\IZ\Gamma$-module
$\IZ^w$ has a $E^2$-term $E^2_{p,q} = H_p^G(EG, H_q^{\pi}(E\pi;\IZ^v))$ and converges to
$H^{\Gamma}_{p+q}(E\Gamma;\IZ^w)$. The $G$-action on $H_q^{\pi}(E\pi;\IZ^v)$ can be
computed as follows.

The composite $E\Gamma \to B\Gamma \to BG$ has a preimage of the base point in $BG$ the
space $\Gamma \times_{\pi} E\pi$. The $\Gamma$-equivariant fiber transport along loops in
$BG$ assigns to each $g \in G = \pi_1(BG)$ a unique $\Gamma$-homotopy class $[t_g]$ of
$\Gamma$-maps $\Gamma \times_{\pi} E\pi \to \Gamma \times _{\times} E\pi$.  A
representative $t_g$ is given for $g \in G$ by the $\Gamma$-map
\[
  t_g \colon \Gamma \times_{\pi} E\pi \to \Gamma \times_{\pi} E\pi, \quad (\gamma, x) \to
  (\gamma \widehat{g}^{-1},Ec_{\widehat{g}}(x))
\]
for any $\widehat{g} \in \Gamma$, which is mapped under the projection $\Gamma \to G$ to
$g$, see for instance~\cite[Sections~1 and Subsection~7A]{Lueck(1986)}.  The
$\IZ\Gamma$-map $t_g$ induces a $\IZ$-chain map
\[
  \id_{\IZ^w} \otimes_{\IZ\Gamma} C_*(t_g) \colon \IZ^w \otimes_{\IZ\Gamma} C_*(\Gamma
  \times_{\pi} E\pi) \to \IZ^w \otimes_{\IZ\Gamma} C_*(\Gamma \times_{\pi} E\pi).
\]
Under the obvious identification
\[
  \IZ^w \otimes_{\IZ\Gamma}C_*(\Gamma \times_{\pi} E\pi) = \IZ^w \otimes_{\IZ\Gamma}
  \IZ\Gamma \otimes_{\IZ \pi} C_*(E\pi) = \IZ^v \otimes_{\IZ \pi} C_*(E\pi)
\]
this becomes the $\IZ$-chain map
\[
  w(\gamma_0^{-1}) \cdot \id_{\IZ^v} \otimes_{\IZ\pi} C_*(Ec_{\widehat{g}}) \colon \IZ^v
  \otimes_{\IZ\pi} C_*(E\pi) \to \IZ^v \otimes_{\IZ\pi} C_*(E\pi).
\]
Hence multiplication with $g \in G$ on $H_d^{\pi}(E\pi;\IZ^v)$ is given by
\[w(\gamma_0^{-1}) \cdot \bigl(a \circ H_d^{\pi}(Ec_{\gamma_0^{-1}};\IZ^v)\bigr)^{-1}
\]
for the automorphism $a \circ H_d^{\pi}(Ec_{\gamma_0^{-1}};\IZ^v)$ defined
in~\eqref{auto_a_circ_H_d_upper_pi(Ec_(gamma_0_upper_1);Z_upper_v)}. From the
definition~\eqref{u(gamma_0)} of $u(\gamma_0)$ we get that
$a \circ H_d^{\pi}(c_{\gamma_0^{-1}}^*E\pi;\IZ^v)$ is multiplication with $u(\gamma_0)$.
Lemma~\ref{lem:w(gamma_0)_is_u(gamma_0)} implies that the $G$-action on
$H_d^{\pi}(E\pi;\IZ^v)$ is trivial.  Hence
$H_0(BG,H_d^{\pi}(E\pi;\IZ^v) )= H_d^{\pi}(E\pi;\IZ^v)$.  Since
$H_p(BG;H_q^{\pi}(E\pi;\IZ^v))[1/|G|]$ vanishes for $p \ge 1$ and $i^*E\Gamma$ is a model
for $E\pi$, we conclude from the Hochschild-Serre spectral sequence that the map
$H_d(i_*) \colon H_d^{\pi}(i^*E\Gamma ;\IZ^v) \to H_d^{\Gamma}(E\Gamma;\IZ^w)$
of~\eqref{i_ast_for_C_ast_and_M} induces an isomorphism
\begin{equation}
  H_d(i_*)[1/|G|] \colon H_d^{\pi}(E\pi;\IZ^v)[1/|G|] \xrightarrow{\cong} H_d^{\Gamma}(E\Gamma;\IZ^w)[1/|G|].
  \label{iso_pi_to_Gamma_after_inverting_|G|}
\end{equation}


\subsection{The fundamental classes}%
\label{subsec:The_fundamental_classes}

\begin{lemma}\label{lem;everything_is_infinite_cyclic}
  Let $\pr \colon E\Gamma \to \eub{\Gamma}$ be the projection. Then the  maps
  \begin{eqnarray*}
    H_d(i_*) \colon H_d^{\pi}(E\pi;\IZ^v) & \to & H_d^{\Gamma}(E\Gamma;\IZ^w);
    \\
    H_d(i_*) \colon H_d^{\pi}(E\pi;\IZ^v) & \to & H_d^{\Gamma}(\eub{\Gamma};\IZ^w);
    \\
    H_d(\pr) \colon H_d^{\Gamma}(E\Gamma;\IZ^w) &\to &
                                                       H_d^\Gamma(\eub{\Gamma};\IZ^w),
  \end{eqnarray*}
  are inclusion of infinite cyclic groups such that the index of the image in the target
  divides $|G|$.
\end{lemma}
\begin{proof}
  From the Mayer-Vietoris sequence associated to the $\Gamma$-pushout with an inclusion of
  free $\Gamma$-$CW$-complexes as upper horizontal arrow, see
  Proposition~\ref{pro:constructing_underline(E)Gamma)}
  \[
    \xymatrix@!C=8em{\coprod_{F \in \calm} \Gamma \times_FEF \ar[r] \ar[d] & E\Gamma\ar[d]
      \\
      \coprod_{F \in \calm} \Gamma/F\ar[r] & \eub{\Gamma} }
  \]

  we obtain the exact sequence sequence
  \[\prod_{F \in F} H^F_{d}(EF;\IZ^{w|_F}) \to H_d^{\Gamma}(E\Gamma;\IZ^w)
    \to H_d^{\Gamma}(\eub{\Gamma};\IZ^w) \to \prod_{F \in F} H^F_{d-1}(EF;\IZ^{w|_F}).
  \]
  If $d$ is even, then $H^F_{d}(EF;\IZ^{w|_F}) = H_d(BF)$ by Lemma~\ref{lem:basics_about_free_slice_systems} and
  $H_{d}(BF) = 0$ holds, as 
  $d$ is even and $F$ has periodic homology, see~\cite[Exercise~4 in Section~VI.9 on page~159]{Brown(1982)}.
  If $d$ is odd, $F \cong \IZ/2$ and $w|_F$ is non-trivial by Lemma~\ref{lem:basics_about_free_slice_systems}
  and a direct  computation shows $H^F_{d}(EF;\IZ^{w|_F}) = 0$. Hence we get the short exact sequence
  \[1 \to H_d^{\Gamma}(E\Gamma;\IZ^w) \to H_d^{\Gamma}(\eub{\Gamma};\IZ^w) \to \prod_{F \in F} H^F_{d-1}(EF;\IZ^{w|_F}).
  \]
  Since there
  is a cocompact $d$-dimensional model for $\eub{\Gamma}$ with zero-dimensional
  $\eub{\Gamma}^{>1}$, the abelian group
  $H_d(C_*(\eub{\Gamma}) \otimes_{\IZ\Gamma} \IZ^{w})$ is finitely generated free.  Hence
  also the group $H_d^{\Gamma}(E\Gamma;\IZ^w)$ is finitely generated free and has the same
  rank as $H_d(C_*(\eub{\Gamma}) \otimes_{\IZ\Gamma} \IZ^{w})$. The group
  $H_d^{\pi}(E\pi;\IZ^v)$ is infinite cyclic.  Since $i^*E\Gamma$ and $i^*\eub{\Gamma}$
  are models for $E\pi$, we conclude from~\eqref{iso_pi_to_Gamma_after_inverting_|G|}
  that all three groups  $H_d^{\Gamma}(E\Gamma;\IZ^w)$, $H_d(C_*(\eub{\Gamma}) \otimes_{\IZ\Gamma} \IZ^{w})$ and 
  $H_d^{\pi}(E\pi;\IZ^v)$ is infinite cyclic. The index of the inclusion
  $H_d(i_*) \colon H_d^{\pi}(E\pi;\IZ^v) \to  H_d^{\Gamma}(\eub{\Gamma};\IZ^w)$ divides $|G|$
  by~\eqref{diagram_of_the_H_d-s}. Since it factorizes as the composite
  of  $H_d(i_*) \colon H_d^{\pi}(E\pi;\IZ^v)  \to  H_d^{\Gamma}(E\Gamma;\IZ^w)$
  and $H_d(\pr) \colon H_d^{\Gamma}(E\Gamma;\IZ^w) \to H_d^\Gamma(\eub{\Gamma};\IZ^w)$,
  also these two maps are inclusions of infinite cyclic groups, whose index divides $|G|$.
  This finishes the proof of e Lemma~\ref{lem;everything_is_infinite_cyclic}.
\end{proof}

We will improve Lemma~\ref{lem;everything_is_infinite_cyclic} in
Theorem~\ref{the:Checking_Poincare_duality}~\eqref{the:Checking_Poincare_duality:indices}.

\begin{notation}\label{not:fundamental_classes}
  The choice of the fundamental class $[B\pi] \in H_d^{\pi}(E\pi;\IZ^v)$ determines
  preferred choices of generators of infinite cyclic groups
  \begin{eqnarray*}
    {[}\eub{\Gamma}/\Gamma {]}
    & \in &
            H_d^{\Gamma}(\eub{\Gamma};\IZ^w);
    \\
    {[}\eub{\Gamma}/\Gamma,\partial \eub{\Gamma}/\Gamma {]}
    & \in &
            H_d^{\Gamma}(\eub{\Gamma},\partial \eub{\Gamma};\IZ^w);
    \\
    {[}E\Gamma/\Gamma{]}
    & \in &
            H_d^{\Gamma}(E\Gamma;\IZ^w);
    \\
    {[}E\Gamma/\Gamma,\partial E\Gamma/\Gamma{]}
    & \in &
            H_d^{\Gamma}(E\Gamma,\partial E\Gamma;\IZ^w),
  \end{eqnarray*}
  by the injections or bijections of infinite cyclic groups appearing
  in~\eqref{diagram_of_the_H_d-s} and in Lemma~\ref{lem;everything_is_infinite_cyclic} by
  requiring that the injection
  $H_d(i_*) \colon H_d^{\pi}(E\pi;\IZ^v) \to H_d^{\Gamma}(\eub{\Gamma};\IZ^w)$ sends
  $[B\pi]$ to $n \cdot [\eub{\Gamma}/\Gamma]$ for some integer $n$ satisfying
  $n \ge 1$.

  If $(X,\partial X)$ is a slice complement model for $\eub{\Gamma}$, it also inherits a generator of
  an infinite cyclic group
  \begin{eqnarray*}
    [X/\Gamma,\partial X/\Gamma] & \in & H_d^{\Gamma}(X,\partial X;\IZ^w)
  \end{eqnarray*}                        
  from~\eqref{diagram_of_the_H_d-s} and $[\eub{\Gamma},\partial \eub{\Gamma}]$.  We also
  obtain a preferred generator
  \begin{eqnarray*}
    [X/\pi,\partial X/\pi] & \in & H_d^{\pi}(X,\partial X;\IZ^v),
  \end{eqnarray*}      
  namely, the one, which is mapped under the isomorphism
  \[
    H_d^{\pi}(X,\partial X;\IZ^v) \xrightarrow{\cong} H_d^{\pi}(\eub{\Gamma},\partial
    \eub{\Gamma};\IZ^v) \xleftarrow{\cong} H_d^{\pi}(\eub{\Gamma};\IZ^v) =
    H_d(E\pi;\IZ^v).
  \]
  to $[B\pi]$. Equivalently, one can define $[X/\pi,\partial X/\pi]$ by requiring that
  the injection of infinite cyclic groups
  $H_d(i_*) \colon H_d^{\pi}(X,\partial X;\IZ^v) \to H_d^{\Gamma}(X,\partial X;\IZ^w)$
  sends $[X/\pi,\partial X/\pi]$ to $n \cdot [X/\Gamma,\partial X/\Gamma]$ for some
  integer $n$ satisfying $n \ge 1$.

  We call these generators fundamental classes as well.
\end{notation}

\begin{example}[Special case of trivial $v$]\label{exa:special_case_of_trvial_v}
  As an illustration we explain, what happens in the special case that $v$ is trivial.
  Then $[B\pi]$ is a generator of the infinite cyclic group $H_d^{\pi}(E\pi;\IZ^v) = H_d(B\pi)$. The
  $\Gamma$-action on $\pi$ by conjugation induces a $G$-action on $H_d(B\pi)$. Since
  $H_d(B\pi)$ is infinite cyclic, there is precisely one group homomorphism
  $\overline{w} \colon G \to \{ \pm 1\}$, for which the $\IZ G$-module $H_d(B\pi)$ is
  $\IZ G$-isomorphic to $\IZ^{\overline{w}}$.  Then $w$ is the composite
  $\Gamma \xrightarrow{\pr} G \xrightarrow{\overline{w}} \{ \pm 1\}$.

  The homomorphisms appearing in Lemma~\ref{lem;everything_is_infinite_cyclic} boil down
  to homomorphisms
  \begin{eqnarray*}
    H_d(i_*) \colon H_d(B\pi) & \to & H_d^G(B\pi;\IZ^{\overline{w}});
    \\
    H_d(i_*) \colon H_d(B\pi)  & \to & H_d^G(\eub{\Gamma}/\pi;\IZ^{\overline{w}});
    \\
    H_d(\pr) \colon H_d^G(B\pi;\IZ^{\overline{w}}) &\to &
                                                          H_d^G(\eub{\Gamma}/\pi;\IZ^{\overline{w}}).
  \end{eqnarray*}

  The string of isomorphisms given by the left column in the
  diagram~\eqref{diagram_of_the_H_d-s} reduces to the string of isomorphisms
  \begin{multline*}
    H_d^G(X/\pi,\partial X/\pi;\IZ^{\overline{w}}) \xrightarrow{\cong}
    H_d^G(E\Gamma/\pi,\partial E\Gamma/\pi;\IZ^{\overline{w}})
    \\
    \xrightarrow{\cong} H_d^G(\eub{\Gamma}/\pi,\partial
    \eub{\Gamma}/\pi;\IZ^{\overline{w}}) \xleftarrow{\cong}
    H_d^G(\eub{\Gamma}/\pi;\IZ^{\overline{w}}).
  \end{multline*}
  If we furthermore assume that $\overline{w}$ is trivial, this reduces further to
  \begin{eqnarray*}
    H_d(Bi) \colon H_d(B\pi) & \to & H_d(B\Gamma);
    \\
    H_d(\pr \circ Bi) \colon H_d(B\pi)  & \to & H_d(\bub{\Gamma});
    \\
    H_d(\pr) \colon H_d(B\Gamma) &\to &
                                        H_d(\eub{\Gamma}),
  \end{eqnarray*}
  and
  \[H_d(X/\Gamma,\coprod_{F \in \calm} S_F/F) \xrightarrow{\cong} H_d^G(B\Gamma,
    \coprod_{F \in \calm} BF) \xrightarrow{\cong} H_d(\bub{\Gamma},\coprod_{F \in \calm}
    \pt) \xleftarrow{\cong} H_d(\bub{\Gamma}).
  \]
\end{example}


\typeout{-------------------- Section 7: Constructing and classifying slice complement models
  --------}

\section{Constructing slice complement models}~%
\label{subsec:Constructing_slice-models}

In this section we construct appropriate slice complement models, for which we will later show that
they carry the desired structure of a finite $d$-dimensional Poincar\'e pair.  Not every
slice complement model has such a structure. Moreover, we will classify slice complement models up to (simple)
$\Gamma$-homotopy equivalence in terms of the underlying free $d$-dimensional slice system
in Sections~\ref{sec:homotopy_classification_of_slice_models}
and~\ref{sec:Simple_homotopy_classification_of_slice_models}.

Throughout this section we will make the following assumptions:

\begin{assumption}\label{ass:general_assumptions}\

\begin{itemize}

\item The natural number $d$ is even and satisfies $d \ge 4$;

\item The group $\Gamma$ satisfies  conditions (M) and (NM), see Notation~\ref{not:(M)_and_(NM)_intro};

  \item The homomorphism $w \colon \Gamma \to \{\pm 1\}$ of Notation~\ref{not:w} has the
  property that $w|_F$ is trivial for every $F \in \calm$;

\item The composite
  \[
    H_d^{\Gamma}(E\Gamma,\partial E\Gamma;\IZ^w) \xrightarrow{\partial}
    H^{\Gamma}_{d-1}(\partial E\Gamma;\IZ^w ) \xrightarrow{\cong} \bigoplus_{F \in \calm}
    H_{d-1}(BF) \to H_{d-1}(BF)
  \]
  of the boundary map, the inverse of the obvious isomorphism and the projection to the
  summand of $F \in \calm$ is surjective for all $F \in \calm$;

\item There exists a finite $\Gamma$-$CW$-model for $\eub{\Gamma}$ of dimension $d$ such
  that its singular $\Gamma$-subspace $\eub{\Gamma}^{> 1}$ is
  $\coprod_{F\in\calm} \Gamma/F$.
  (This condition is discussed and simplified in Theorem~\ref{the:Models_for_the_classifying_space_for_proper_Gamma-actions_intro}
  and implies conditions (M) and (NM), see Remark~\ref{Theorem_models_and_necessary_conditions}.)

\item There is a finite $d$-dimensional Poincar\'e $CW$-complex model for $B\pi$ with
  respect to the orientation homomorphisms $v = w|_{\pi}\colon \pi \to \{\pm 1 \}$. We fix a
  choice of a fundamental class $[B\pi] \in H_d^{\pi}(E\pi;\IZ^v)$.

\end{itemize}

\

\end{assumption}

Note that the assumption that $d$ is even implies, that $F$ acts orientation preserving on $S_F$ by
Lemma~\ref{lem:basics_about_free_slice_systems}.

\begin{remark}[Reformulation of (H)]\label{rem:reformulation_(H)}
  One can easily check using~\eqref{diagram_of_the_H_d-s} that the map appearing in
  condition (H) can be identified with the map
  \[
    H_d^{\Gamma}(\eub{\Gamma};\IZ^w) \xrightarrow{\partial} H^{\Gamma}_{d-1}(\partial
    E\Gamma;\IZ^w ) \xrightarrow{\cong} \bigoplus_{F \in \calm} H_{d-1}(BF) \to
    H_{d-1}(BF),
  \]
  where the first map is the boundary map of the Mayer-Vietoris sequence associated to the
  $\Gamma$-pushout appearing in Proposition~\ref{pro:constructing_underline(E)Gamma)} and
  the second map is the projection onto the summand belonging to $F$. Recall that
  $H_d^{\Gamma}(\eub{\Gamma};\IZ^w)$ is infinite cyclic.  From this Mayer
  Vietoris sequence, we also conclude that condition (H) is satisfied, if and only if the kernel of the
  map $\bigoplus_{F \in \calm} H_{d-1}(BF) \to H_{d-1}^{\Gamma}(E\Gamma;\IZ^w)$ induced by the
  various inclusions $F \to \Gamma$ contains an element $\{\kappa_F \mid F \in \calm\}$
  such that each $\kappa_F \in H_{d-1}(BF)$ is a generator of the finite cyclic group $H_{d-1}(BF)$
  of order $|F|$.
\end{remark}


\subsection{Invariants associated to slice complement models}%
\label{subsec:Invariants_associated_to_slice_models}

\begin{notation}\label{not:kapp}
  Let $\kappa_F \in H_d(BF)$ be the image of $[E\Gamma/\Gamma,\partial E\Gamma/\Gamma]$
  defined in Notation~\ref{not:fundamental_classes} under the composite
  \[
    H_d^{\Gamma}(E\Gamma,\partial E\Gamma;\IZ^w) \xrightarrow{\partial}
    H^{\Gamma}_{d-1}(\partial E\Gamma;\IZ^w ) \xrightarrow{\cong} \bigoplus_{F \in \calm}
    H_{d-1}(BF) \to H_{d-1}(BF)
  \]
\end{notation}
Note that $\kappa_F$ is a generator of the finite cyclic group $H_{d-1}(BF)$ of order
$|F|$, since the composite above is surjective by assumption.

Let $(X,\partial X)$ be a slice complement model with respect to the slice system $\cals$. Fix an
orientation on $\cals$ in the sense of
Definition~\ref{def:free_d-dimensional_slide_system}, i.e., a choice of fundamental class
$[S_F] \in H_{d-1}(S_F)$ for every $F \in \calm$. Recall that this is the same as a choice
of fundamental class $[S_F/F] \in H_{d-1}(S_F/F)$ for every $F \in \calm$, as explained
after Definition~\ref{def:free_d-dimensional_slide_system}.

Define
\begin{equation}
  \mu(X,\partial X) = (\mu(X,\partial X)_F)_{F \in \calm}  \in H_{d-1}(\partial X/\Gamma)
  = \bigoplus_{F \in \calm} H_d(S_F/F)
  \label{mu(X,partial_X)_in_H_(d-1)(partial_X/Gamma)}
\end{equation}
to be the image of $[X/\Gamma,\partial X/\Gamma]$ defined in
Notation~\ref{not:fundamental_classes} under the boundary map
$H_d^{\Gamma}(X,\partial X;\IZ^w) \to H_{d-1}^{\Gamma}(\partial X;\IZ^w) =
H_{d-1}(\partial X/\Gamma)$.  Define the element
\begin{equation}
  s \in H_{d-1}^{\Gamma}(\partial X;\IZ^w) = \bigoplus_{F \in \calm} H_d(S_F/F)
  \label{element_s}
\end{equation}
by $([S_F/F])_{F \in \calm}$.  For $F \in \calm$ denote by
\begin{equation}
  m_F(X,\partial X) \in \IZ
  \label{integer_m_F}
\end{equation}
the integer, for which $\mu(X,\partial X)_F = m_F(X,\partial X) \cdot [S_F/F]$ holds.

    \begin{lemma}\label{lem:kappa_F_is_m_F(X,partial_X)_cdot_d(S_F)}
      Let $(X,\partial X)$ be a slice complement model with respect to the oriented free
      $d$-dimensional slice system $\cals = \{S_F \mid F \in \calm\}$.  Then we get for
      every $F \in \calm$
      \[
        \kappa_F = m_F(X,\partial X) \cdot d(S_F)
      \]
      in $H_d(BF)$, where the invariant $d(S_F)$ has been defined in~\eqref{d(S_F)}.
    \end{lemma}
    \begin{proof}
      We have the commutative diagram
      \[
        \xymatrix{H_d^{\Gamma}(X,\partial X;\IZ^w) \ar[r]^-{\partial}
          \ar[d]^{\cong}_{H_d^{\Gamma}(c(X),\partial c(\partial X);\IZ^w)} &
          H_{d-1}^{\Gamma}(\partial X;\IZ^w) = \bigoplus_{F \in \calm} H_d(S_F/F)
          \ar[d]^{\bigoplus_{F \in \calm} H_d(c_{S_F})}
          \\
          H_d^{\Gamma}(E\Gamma,\partial E\Gamma;\IZ^w) \ar[r]_-{\partial} &
          H_{d-1}^{\Gamma}(\partial E\Gamma;\IZ^w) = \bigoplus_{F \in \calm} H_d(BF),}
      \]
      where the maps $c(X)$, $\partial c(\partial X)$, and  $c(S_F)$ are given by
      classifying maps.  The left vertical arrow is an isomorphism,
      see~\eqref{diagram_of_the_H_d-s} and sends $[X/\Gamma,\partial X/\Gamma]$ to
      $[E\Gamma/\Gamma,\partial E\Gamma/\Gamma]$ by definition. Now
      $\kappa_F = m_F(X,\partial X) \cdot d(S_F)$ follows from the definitions of $d(S_F)$,
      $\kappa_F$, and $m_F(X,\partial X)$.
    \end{proof}

    \begin{definition}[Poincar\'e slice complement model]\label{def:Poincare_slice_model}
      We call a slice complement model $(X,\partial X)$ a \emph{Poincar\'e slice complement model} if
      $(X/\Gamma,\partial X/\Gamma)$ carries the structure of a finite Poincar\'e pair.
    \end{definition}

    Recall that the orientation homomorphism underlying the Poincar\'e structure on
    $(X/\Gamma,\partial X/\Gamma)$ must be the map $w \colon \Gamma \to \{\pm1 \}$ defined
    in~\eqref{not:w} by  Remark~\ref{rem:extension_of_the_orientation_homomorphism_for_pi}.
    Moreover, we have a preferred fundamental class
    $[X/\Gamma,\partial X/\Gamma] \in H_d^{\Gamma}(X,\partial X;\IZ^w)$, see
    Notation~\ref{not:fundamental_classes}.

    \begin{definition}[Condition (S)]\label{def:condition_(S)}
      An oriented free $d$-dimensional slice system satisfies condition (S) if
      $d(S_F) = \kappa_F$ holds for all $F \in \calm$.
    \end{definition}

    Note that condition (S) determines each $S_F$ up to oriented $F$-homotopy equivalence.
    It determines also the orientation $[S_F]$ for those $F \in \calm$, for which
    $|F| \ge 3$ holds.  If $|F|$ has order $2$, then replacing $[S_F]$ by $-[S_F]$ does
    not affect condition (S).

     \begin{lemma}\label{lem:necessary_condition_for_slice_model_to_be_Poincare}
       Let $(X,\partial X)$ be a Poincar\'e slice complement model with respect to the free
       $d$-dimensional slice system $\cals = \{S_F \mid F \in \calm\}$. Then there is an
       orientation on $\cals$ such that $m_F(X,\partial X) = 1$ holds for  every
       $F \in \calm$ and $\cals$ satisfies condition (S).
     \end{lemma}
     \begin{proof}
       Since $(X/\Gamma,\partial X/\Gamma)$ admits the structure of a finite Poincar\'e
       pair, it is part of the definition that the boundary map
       \[H_d^{\Gamma}(X,\partial X;\IZ^w) \to H_{d-1}^{\Gamma}(\partial X;\IZ^w) =
         \bigoplus_{F \in \calm} H_d(S_F/F)
       \]
       sends the fundamental class $[X,\partial X]$ to an element whose component for
       $F \in \calm$ is a generator of $H_{d-1}(S_F/F)$ for every $F \in \calm$.  Choose
       some orientation on $\cals$. With respect to it we get
       $m_F(X,\partial X) \in \{ \pm 1\}$.  Then we get the desired orientation by
       replacing $[S_F]$ by $m_F(X,\partial X) \cdot [S_F]$. Namely, with this new
       orientation we have $m_F(X,\partial X) = 1$ for every $F \in \calm$ and $(S)$ holds
       by Lemma~\ref{lem:kappa_F_is_m_F(X,partial_X)_cdot_d(S_F)}.
     \end{proof}

     \begin{remark}[On the condition (S)]\label{rem:On_the_condition_(S)}
       Our goal is to construct a Poincar\'e slice complement model $(X,\partial X)$ with respect to
       some free $d$-dimensional slice system $\cals = \{S_F \mid F \in \calm\}$ By
       Lemma~\ref{lem:necessary_condition_for_slice_model_to_be_Poincare} there exists an
       orientation on $\cals$ such that condition (S) holds. 
       This motivates that in the sequel we will consider only oriented systems $\cals$
       satisfying condition (S). Note that condition (S) implies by
       Lemma~\ref{lem:kappa_F_is_m_F(X,partial_X)_cdot_d(S_F)}
       \[
         m_F(X,\partial X) \equiv 1 \mod |F|
       \]
       for every $F \in \calm$, since $\kappa_F$ is a generator of the finite cyclic group
       $H_d(BF)$ of order $|F|$ for every $F \in \calm$.
     \end{remark}

     The main result of this section is

     \begin{theorem}[Existence of Poincar\'e slice complement models]\label{the:existence_Poincare_slice_models}
       Suppose that Assumption~\ref{ass:general_assumptions} is satisfied and let $\cals$ be 
       an oriented free $d$-dimensional slice system $\cals$ satisfying condition (S).

       Then there exists a  Poincar\'e slice complement model
       $(X,\partial X)$ with respect to $\cals$ such that $m_F(X,\partial X) = 1$ holds  for every
       $F \in \calm$.
     \end{theorem}

     \begin{remark}[Basic strategy]\label{rem:basic_strategy}
       Consider any oriented slice system $\cals$ satisfying condition (S).
       Let $(X,\partial X)$ be a slice complement model with respect to $\cals$. Then 
       $m_F(X,\partial X) \equiv 1 \mod |F|$
       holds for every $F \in \calm$ as explained in Remark~\ref{rem:On_the_condition_(S)}.
      So our basic strategy will be to construct some slice complement model $(X,\partial X)$ with respect to $\cals$
       and  then to modify  it using $m_F(X,\partial X) \equiv 1 \mod |F|$
       such that we have get $m_F(X,\partial X) = 1$ for every $F \in \calm$.
       This will be done in
       Theorem~\ref{the:constructing_slice_models}~\eqref{the:Checking_Poincare_duality:duality}.
        Then we will show in Theorem~\ref{the:Checking_Poincare_duality} that
       $(X,\partial X)$ carries the structure of a finite $d$-dimensional Poincar\'e pair.

       Hence Theorem~\ref{the:existence_Poincare_slice_models} will be a direct
       consequence of Theorem~\ref{the:constructing_slice_models} and
       Theorem~\ref{the:Checking_Poincare_duality}~\eqref{the:Checking_Poincare_duality:duality}.
     \end{remark}


     \subsection{Constructing slice complement models}\label{subsec:constructing_slice_models}

     \begin{theorem}[Constructing slice complement models]\label{the:constructing_slice_models}
       Suppose that Assumption~\ref{ass:general_assumptions} is satisfied. Let $\cals$ be an
       oriented free $d$-dimensional slice system $\cals$ satisfying condition (S).
       
       Then there exists a slice complement model $(X,\partial X)$ for $\eub{\Gamma}$ with respect to
       $\cals$ such that $m_F(X,\partial X) = 1$ holds for every $F \in \calm$.
     \end{theorem}
     \begin{proof}
       We begin with constructing a $\Gamma$-$CW$-pair $(X,\partial X)$ together with a
       cellular $\Gamma$-map of $\Gamma$-$CW$-pairs
       \[
         (u,\partial u) \colon (X,\partial X) \to (\eub{\Gamma},\partial \eub{\Gamma}).
       \]
       Note for the sequel that $\eub{\Gamma}_n \cup \partial \eub{\Gamma}$ is
       $\partial \eub{\Gamma}$ for $n = -1$ and is $\eub{\Gamma}_n$ for
       $n = 0,1,2 \ldots, d$.

       We will construct by induction over $n = -1,0,1, \ldots, d$ $\Gamma$-$CW$-pairs
       $(X_n,\partial X)$ and $\Gamma$-maps
       $u_n \colon X_n \to \eub{\Gamma}_n \cup \partial \eub{\Gamma}$ satisfying
       \begin{itemize}

       \item $X_{n-1} \subseteq X_n$ holds for $n = 0,1,2 \ldots, d$;
       \item $u_n|_{X_{n-1}} = u_{n-1}$ holds for $n = 0,1,2 \ldots, d$;
       \item $u_n$ is $(d-1)$-connected for $n = -1, 0,1,2 \ldots, d$;
       \item
         $H_n(u_n,u_{n-1}) \colon H_n(X_n,X_{n-1}) \to H_n(\eub{\Gamma}_n,
         \eub{\Gamma}_{n-1} \cup \partial \eub{\Gamma})$ is bijective for
         $n = 0,1,2 \ldots, d$.  (Actually, the source and target will come with explicit
         $\IZ \pi$-bases, which are respected by this map.)
       \end{itemize}
       Since $\eub{\Gamma} = \eub{\Gamma}_d$, we then can and will  define $X$ to be $X_d$, and $u$
       to be $u_d$.

       The induction beginning $n = -1$ is given by $X_{-1} = \partial X$ and
       $u_{-1} = \partial u\colon \partial X \to \partial \eub{\Gamma}$ is the coproduct
       over $F \in \calm$ of the projections $\Gamma\times_F S_F \to \Gamma/F$. The
       induction step from $(n-1)$ to $n$ for $n = 0,1,2, \ldots, d$ is done as follows.

       Choose a finite index set $I_n$ and for $i \in I_n$ a map of pairs
       \[
         (Q_{i}^n,q_{i}^n) \colon (D^n,S^{n-1}) \to (\eub{\Gamma}_n,\eub{\Gamma}_{n-1}
         \cup \partial \eub{\Gamma})
       \]
       such that for the induced $\Gamma$-maps
       \[
         (\widehat{Q_{i}^n},\widehat{q_{i}^n}) \colon (\Gamma \times D^n, \Gamma \times
         S^{n-1}) \to (\eub{\Gamma}_n,\eub{\Gamma_{n-1}} \cup \partial \eub{\Gamma})
       \]
       sending $(\gamma, x)$ to $\gamma \cdot Q_{i}^n(x)$ for $\gamma \in \Gamma $ and
       $x \in D^n$ we get a $\Gamma$-pushout
       \[
         \xymatrix@!C=11em{\coprod_{i \in I_n} \Gamma \times S^{n-1} \ar[r]^-{\coprod_{i
               \in I_n} \widehat{q_{i}^n}} \ar[d] & \eub{\Gamma_{n-1}} \cup \partial
           \eub{\Gamma} \ar[d]
           \\
           \coprod_{i \in I_n} \Gamma \times D^n \ar[r]^-{\coprod_{i \in I_n}
             \widehat{Q_{i}^n}} & \eub{\Gamma}_n.  }
       \]
       Since $u_{n-1}$ is $(d-1)$-connected by induction hypothesis, we can find for
       $i \in I_n$ a map $p_{i}^n \colon S^{n-1} \to X_{n-1}$ and a homotopy
       \[
         h^n_{i} \colon S^n \times [0,1] \to \eub{\Gamma}_{n-1} \cup \partial \eub{\Gamma}
       \]
       from $q_{i}^n$ to $u_{n-1} \circ p_i^n$. By the Cellular Approximation Theorem, we
       can assume without loss of generality that the image of $p_{i}^n$ is contained in
       $X_{n-1} \cap (\partial X)_{n-1}$.  Now define $X_n$ by the $\Gamma$-pushout
       \[
         \xymatrix@!C=11em{\coprod_{i \in I_n} \Gamma \times S^{n-1} \ar[r]^-{\coprod_{i
               \in I_n} \widehat{p_{i}^n}} \ar[d] & X_{n-1}\ar[d]
           \\
           \coprod_{i \in I_n} \Gamma \times D^n \ar[r]^-{\coprod_{i \in I_n}
             \widehat{P_{i}^n}} &X_n }
       \]
       for an appropriate extension $P_{i}^n \colon D^n \to X_n$ of $p_{i}^n$.  In order
       to define the extension $u_n \colon X_n \to \eub{\Gamma}_n$ of
       $u_{n-1} \colon X_{n-1} \to \eub{\Gamma}_{n-1} \cup \partial \eub{\Gamma}$, we have
       to specify for each $i \in I_n$ a map $v_{i}^n \colon D^n \to \eub{\Gamma}_n$ whose
       restriction to $S^{n-1}$ is $u_{n-1} \circ p_{i}^n$. It is given by sending
       $x \in D^n$ to $h^n_{i}(x,2 \cdot |x| -1 )$ if $1/2 \le |x| \le 1$ and to
       $Q_{i}^n(2 \cdot x)$ if $0 \le |x| \le 1/2$.

       We get for each $i \in I_n$ an element
       $e_{i}^n \in H_n(\eub{\Gamma}_n, \eub{\Gamma}_{n-1} \cup \partial \eub{\Gamma})$,
       namely the image of the class of $(Q_{i}^n,q_{i}^n)$ under the Hurewicz
       homomorphism
       $\pi_n(\eub{\Gamma}_n, \eub{\Gamma}_{n-1} \cup \partial \eub{\Gamma}) \to
       H_n(\eub{\Gamma}_n, \eub{\Gamma}_{n-1} \cup \partial \eub{\Gamma})$.  Analogously
       get for each $i \in I_n$ an element $x_{i}^n \in H_n(X_n, X_{n-1})$, namely the
       image of the class of $(P_{i}^n,p_{i}^n)$ under the Hurewicz homomorphism
       $\pi_n(X_n, X_{n-1}) \to H_n(X_n,X_{n-1})$. One easily checks that
       $\{e^n_{i} \mid i \in I_n\}$ and $\{x^n_{i} \mid i \in I_n\}$ are
       $\IZ \Gamma$-basis for the finitely generated free $\IZ\Gamma$-modules
       $H_n(\eub{\Gamma}_n, \eub{\Gamma}_{n-1} \cup \partial \eub{\Gamma})$ and
       $H_n(X_n,X_{n-1})$ and that $H_n(u_n,u_{n-1})$ sends $x_{i}^n$ to $e_{i}^n$. In
       particular $H_n(u_n,u_{n-1})$ is bijective.

       It is not hard to show that the commutative diagram
       \begin{equation}
         \xymatrix{X_{n-1} \ar[r]^-{u_{n-1}} \ar[d]
           &
           \eub{\Gamma}_{n-1} \cup \partial \eub{\Gamma} \ar[d]
           \\
           X_n \ar[r]^-{u_n}&
           \eub{\Gamma}_n
         }
         \label{homotopy_pushout_for_X_(n-1)_and_X_n}
       \end{equation}
       is a homotopy pushout, actually a $\Gamma$-homotopy pushout.  The reason is
       essentially that changing the attaching maps of the $\Gamma$-cells for a
       $\Gamma$-$CW$-complex by a $\Gamma$-homotopy does not change the $\Gamma$-homotopy
       type. Now one easily checks that $u_n$ is $(d-1)$-connected using the induction
       hypothesis that $u_{n-1}$ is $(d-1)$-connected. This finishes the construction of the map $(u,\partial u)$.

       Next we analyze this construction in the top dimension closer.
       Since $\eub{\Gamma}$ is a finite $d$-dimensional $\Gamma$-$CW$-complex, the
       inclusion
       $(\eub{\Gamma},\partial \eub{\Gamma}) \to (\eub{\Gamma}, \eub{\Gamma}_{d-1})$
       induces an exact sequence of finitely generated abelian groups
       \begin{multline*}
         0 \to H_d^{\Gamma}(\eub{\Gamma},\partial \eub{\Gamma};\IZ^w) \to
         H_d^{\Gamma}(\eub{\Gamma},\eub{\Gamma}_{d-1};\IZ^w) = \IZ^w \otimes_{\IZ\Gamma}
         C_d(\eub{\Gamma},\partial \eub{\Gamma})
         \\
         \xrightarrow{c_d \otimes_{\IZ \pi} \id_{\IZ^w}} \IZ^w \otimes_{\IZ\Gamma}
         C_{d-1}(\eub{\Gamma}, \partial \eub{\Gamma}),
       \end{multline*}
       where
       $c_d \colon C_d(\eub{\Gamma},\partial \eub{\Gamma}) \to C_{d-1}(\eub{\Gamma},
       \partial \eub{\Gamma})$ is the $d$-th differential of the cellular
       $\IZ\Gamma$-chain complex of the $\Gamma$-$CW$-pair
       $(\eub{\Gamma},\partial \eub{\Gamma})$.  Since the image of
       $c_d \otimes_{\IZ \pi} \id_{\IZ^w}$ is a free $\IZ$-module, the $\IZ$-map
       $H_d^{\Gamma}(\eub{\Gamma},\partial \eub{\Gamma};\IZ^w) \to
       H_d^{\Gamma}(\eub{\Gamma},\eub{\Gamma}_{d-1};\IZ^w)$ is split injective.  Recall
       that $H_d^{\Gamma}(\eub{\Gamma},\partial \eub{\Gamma};\IZ^w)$ is an infinite cyclic
       group and comes with a preferred generator
       $[\eub{\Gamma}/\Gamma,\partial \eub{\Gamma}/\Gamma]$ and that
       $\{1 \otimes e_{i}^d \mid i \in I_d\}$ is a $\IZ$-basis for
       $\IZ^w \otimes_{\IZ\Gamma} C_d(\eub{\Gamma},\partial \eub{\Gamma})$.  Hence we can
       find integers $\lambda_{i}$ and $\mu_{i}$ for $i \in I_d$ such that the image of
       $[\eub{\Gamma}/\Gamma,\partial \eub{\Gamma}/\Gamma]$ under
       \[
         H_d^{\Gamma}(\eub{\Gamma},\partial \eub{\Gamma};\IZ^w) \to
         H_d^{\Gamma}(\eub{\Gamma},\eub{\Gamma}_{d-1};\IZ^w) = \IZ^w \otimes_{\IZ\Gamma}
         C_d(\eub{\Gamma},\partial \eub{\Gamma})
       \]
       is $\sum_{i \in I_d} \lambda_{i} \cdot (1 \otimes e^d_{i})$ and
       $\sum_{i \in I_{d}} \lambda_{i} \cdot \mu_i = 1 $ holds.

       The map $u$ induces a commutative diagram
       \[
         \xymatrix{ H_d^{\Gamma}(X,\partial X;\IZ^w) \ar[r]
           \ar[d]^{H_d^{\Gamma}(u,\partial u;\IZ^w)}_{\cong} & H_d^{\Gamma}(X,X_{d-1}
           ;\IZ^w) = \IZ^w \otimes_{\IZ\Gamma} C_d(X,\partial X)
           \ar[d]^{H_d^{\Gamma}(u,u_{d-1};\IZ^w) = \id_{\IZ^w} \otimes_{\IZ\Gamma}
             C_d(u,\partial u) }_{\cong}
           \\
           H_d^{\Gamma}(\eub{\Gamma},\partial \eub{\Gamma};\IZ^w) \ar[r] &
           H_d^{\Gamma}(\eub{\Gamma},\eub{\Gamma}_{d-1};\IZ^w) = \IZ^w \otimes_{\IZ\Gamma}
           C_d(\eub{\Gamma},\partial \eub{\Gamma}_{d-1}) }
       \]
       where the left vertical arrow sends $[X/\Gamma,\partial X/\Gamma]$ to
       $[\eub{\Gamma}/\Gamma,\partial \eub{\Gamma}/\Gamma]$ and the right vertical arrow
       $1 \otimes x_{i}^d$ to $1 \otimes e_{i}^d$. Hence the map
       \[
         H_d^{\Gamma}(X,\partial X;\IZ^w) \to H_d^{\Gamma}(X,X_{d-1};\IZ^w) = \IZ^w
         \otimes_{\IZ\Gamma} C_d(X,\partial X)
       \]
       sends $[X/\Gamma,\partial X/\Gamma]$ to
       $\sum_{i \in I_d} \lambda_{i} \cdot (1 \otimes x^d_{i})$.  The following diagram
       commutes
       \[
         \xymatrix{H_d^{\Gamma}(X,\partial X;\IZ^w) \ar[r] \ar[d] &
           H_{d-1}^{\Gamma}(\partial X;\IZ^w) = H_{d-1}(\partial X/\Gamma)
           \ar[d]^{H_{d-1}^{\Gamma}(j;\IZ^w)}
           \\
           H_d^{\Gamma}(X,X_{d-1};\IZ^w) \ar[r] & H_{d-1}^{\Gamma}(X_{d-1};\IZ^w),}
       \]
       where $j \colon \partial X \to X_{d-1}$ is the inclusion and the horizontal arrows
       are boundary homomorphism of pairs.  Let
       $\pr_{\partial X} \colon \partial X \to \partial X_{d-1}/\Gamma$ be the projection.
       One easily checks that the lower horizontal arrow sends $1 \otimes x_{i}^d$ to the
       image of the class $[p_{i}^d]$ of $p_{i}^d$ under the composite
       \[
         \pi_{d-1}(X_{d-1}) \xrightarrow{h_{d-1}[X_{d-1}]} H_{d-1}(X_{d-1})
         \xrightarrow{H_d(f_*)} H_d^{\Gamma}(X_{d-1};\IZ^w),
       \]
       where $h_{d-1}[X_{d-1}]$ is the Hurewicz homomorphism and $f_*$ is the obvious
       chain map $f_* \colon C_*(X_{d-1}) \to C_*(X_{d-1}) \otimes_{\IZ \Gamma}
       \IZ^w$. Recall that $\mu(X,\partial X)$ is the image of
       $[X/\Gamma,\partial X/\Gamma]$ under the upper horizontal arrow in the diagram
       above. Hence we get for the image of $\mu(X,\partial X)$ under the map
       \[
         H_{d-1}^{\Gamma}(j;\IZ^w) \colon H_{d-1}^{\Gamma}(\partial X;\IZ^w) =
         H_{d-1}(\partial X/\Gamma) \to H_{d-1}^{\Gamma}(X_{d-1};\IZ^w)
       \]
       the equality
       \begin{equation}
         H_{d-1}^{\Gamma}(j;\IZ^w)(\mu(X,\partial X)) = \sum_{i \in I_d} \lambda_i \cdot H_{d-1}(f_*) \circ h_{d-1}[X_{d-1}]([p_{i}^d]).
         \label{relating_mu(X,partial_X)_to_the_attaching_maps}
       \end{equation}
       Note that $H_{d-1}(j;\IZ^w)$ is injective, as $j \colon \partial X \to X_{d-1}$ is
       an inclusion of $(d-1)$-dimensional $\Gamma$-$CW$-complexes.  Hence we have
       expressed $\mu(X,\partial X)$ in terms of the attaching maps
       $p_{i}^d \colon S^{d-1} \to X_{d-1}$.

       Next we investigate how we can change the maps $p_{i}^d$ and how this change
       affects $\mu(X,\partial X)$.  Since we have for $n = 0,1,2, \ldots, (d-1)$ the homotopy
       pushout~\eqref{homotopy_pushout_for_X_(n-1)_and_X_n}, the following diagram is a
       homotopy pushout
       \[
         \xymatrix{\partial X \ar[r]^-{u_{-1}} \ar[d]_{j} & \partial \eub{\Gamma} \ar[d]
           \\
           X_{d-1} \ar[r]^-{u_{d-1}}& \eub{\Gamma}_{d-1}.  }
       \]
       Since $H_d(\eub{\Gamma}_{d-1})$, $H_{d-2}(\partial X)$, and
       $H_{d-1}(\partial \eub{\Gamma})$ vanish, we get from the associated Mayer-Vietoris
       sequence a short exact sequence of $\IZ \Gamma$-modules
       \[
         0 \to H_{d-1}(\partial X) \xrightarrow{H_{d-1}(j)} H_{d-1}(X_{d-1})
         \xrightarrow{H_{d-1}(u_{d-1})} H_{d-1}(\eub{\Gamma}_{d-1}) \to 0.
       \]
       Since $\eub{\Gamma}$ is contractible and
       $u_{d-1} \colon X_{d-1} \to \eub{\Gamma}_{d-1}$ is $(d-1)$-connected, we conclude
       that $X_{d-1}$ and $\eub{\Gamma}_{d-1}$ are $(d-2)$-connected.  Hence we can extend
       the short exact sequence above to a commutative diagram whose vertical maps are
       Hurewicz isomorphisms
       \[
         \xymatrix@!C=10em{ & \pi_{d-1}(X_{d-1}) \ar[r]^{\pi_{d-1}(u_{d-1})}
           \ar[d]_{h_{d-1}[X_{d-1}]}^{\cong} & \pi_{d-1}(\eub{\Gamma}_{d-1})
           \ar[d]_{h_{d-1}[\eub{\Gamma}_{d-1}]}^{\cong}
           \\
           H_{d-1}(\partial X) \ar[r]^{H_{d-1}(j)} & H_{d-1}(X_{d-1})
           \ar[r]^{H_{d-1}(u_{d-1})} & H_{d-1}(\eub{\Gamma}_{d-1}) }
       \]
       Note that the only requirement about $p_{i}^d$ is that the image of the class
       $[p_{i}^d]$ under
       $\pi_{d-1}(u_{d-1}) \colon \pi_{d-1}(X_{d-1}) \to \pi_{d-1}(\eub{\Gamma}_{d-1})$ is
       $[q_i^d]$. Hence we can choose for every $i \in I_d$ an element
       $a_{i} \in H_{d-1}(\partial X)$ and replace $p_{i}^d$ by any map
       $\widehat{p_{i}^d} \colon S^{d-1} \to X_{d-1}$ satisfying
       \[
         h_{d-1}[X_{d-1}]([\widehat{p_{i}^d}] - [p_{i}^d]) = H_{d-1}(j)(a_{i}).
       \]
       If we use the maps $\widehat{p_{i}^d}$, then we get a $\Gamma$-$CW$-pair
       $(\widehat{X},\partial X)$ and
       equation~\eqref{relating_mu(X,partial_X)_to_the_attaching_maps} becomes
       \begin{equation}
         H_{d-1}^{\Gamma}(j;\IZ^w)(\mu(\widehat{X},\partial X))
         = \sum_{i \in I_d} \lambda_i \cdot H_{d-1}(f_*) \circ h_{d-1}[X_{d-1}]([\widehat{p_{i}^d}]).
         \label{relating_mu(widehat(X),partial_widehat(X))_to_the_attaching_maps}
       \end{equation}
       One easily checks
       \begin{multline*}
         H_{d-1}(f_*) \circ h_{d-1}[X_{d-1}]([\widehat{p_{i}^d}]) - H_{d-1}(f_*) \circ
         h_{d-1}[X_{d-1}]([p_{i}^d])
         \\
         = H_{d-1}(f_*) \circ  H_{d-1}(j)(a_i) = H_{d-1}^{\Gamma}(j;\IZ^w) \circ
         H_{d-1}(\pr_{\partial X})(a_i).
       \end{multline*}
       Since $H_{d-1}^{\Gamma}(j;\IZ^w)$ is injective, we conclude
       from\eqref{relating_mu(X,partial_X)_to_the_attaching_maps}
       and~\eqref{relating_mu(widehat(X),partial_widehat(X))_to_the_attaching_maps}
       \[
         \mu(\widehat{X},\partial X) - \mu(X,\partial X) = \sum_{i \in I_d} \lambda_i
         \cdot H_{d-1}(\pr_{\partial X})(a_{i}).
       \]
       Now consider any element $a \in H_{d-1}(\partial X)$. If we choose
       $a_{i} = \mu_{i} \cdot a$, we get
       \begin{eqnarray}
         \mu(\widehat{X},\partial X) - \mu(X,\partial X)
         & = &
               \sum_{i \in i_d} \lambda_{i} \cdot H_{d-1}( \pr_{\partial X})(\mu_{i} \cdot a)
               \label{mu_difference}
         \\
         & = &
               \sum_{i \in I_d} \lambda_{i} \cdot \mu_{i} \cdot H_{d-1}( \pr_{\partial X})(a)
               \nonumber
         \\
         & = &
               \left(\sum_{i \in I_d} \lambda_{i} \cdot \mu_{i} \right) \cdot H_{d-1}( \pr_{\partial X})(a)             
               \nonumber
         \\
         & = &
               H_{d-1}( \pr_{\partial X})(a).
               \nonumber
       \end{eqnarray}
       The map
       $H_{d-1}(\pr_{\partial X}) \colon H_{d-1}(\partial X) \to H_{d-1}(\partial
       X/\Gamma)$ can be identified with the map
       \[\bigoplus_{F \in \calm} \pi_X[F] \colon \bigoplus_{F \in \calm} \IZ \Gamma
         \otimes_{\IZ F} H_d(S_F) \to \bigoplus_{F \in \calm} H_d(S_F/F)
       \]
       where $\pi_X[F] \colon \IZ \Gamma \otimes_{\IZ F} H_{d-1}(S_F) \to H_{d-1} (S_F/F)$
       sends $(\gamma, x)$ to the image of $x$ under
       $H_{d-1}(\pr_{S_F}) \colon H_{d-1}(S_F) \to H_{d-1}(S_F/F)$ for the projection
       $\pr_{S_F} \colon S_F \to S_F/F$.  Since
       $H_{d-1}(\pr_{S_F}) \colon H_{d-1}(S_F) \to H_{d-1}(S_F/F)$ is the inclusion of
       infinite cyclic groups of index $[F]$, we conclude that an element
       $b = (b_F)_{F \in \calm} \in \bigoplus_{F \in \calm} H_{d-1}(S_F/F) =
       H_{d-1}(\partial X/\Gamma)$ lies in the image of
       $H_{d-1}(\pr_{\partial X}) \colon H_{d-1}(\partial X) \to H_{d-1}(\partial
       X/\Gamma)$, if and only if $b_F = m_F \cdot [S_F/F]$ for some integer $m_F$
       satisfying $m_F \equiv 0 \mod |F|$.

       Recall from Remark~\ref{rem:On_the_condition_(S)} that the integer
       $m_F(X,\partial X)$ defined by
       $\mu(X,\partial X)_F = m_F(X,\partial X) \cdot [S_F/F]$ satisfies
       $m_F(X,\partial X) \equiv 1 \mod |F|$. Hence the difference $s - \mu(X,\partial X)$
       lies in the image of the map $H_{d-1}(\pr_{\partial X})$, where $s$ has been
       defined in~\eqref{element_s}. If $a$ is such a preimage, we have associated to it
       the new pair $(\widehat{X},\partial X)$.  We conclude
       $m_F(\widehat{X},\partial X) = 1$ for every $F \in \calm$
       from~\eqref{mu_difference}.  This finishes the proof of
       Theorem~\ref{the:constructing_slice_models}.
     \end{proof}


     \subsection{Checking Poincar'e duality}\label{subsec:Checking_Poincare_duality}

\begin{theorem}[Checking Poincar\'e duality]\label{the:Checking_Poincare_duality}
  Suppose that  Assumption~\ref{ass:general_assumptions} is satisfied. Let $\cals$ be an
  oriented free $d$-dimensional slice system  satisfying condition (S).
  Let $(X,\partial X)$ be a slice complement model for $\eub{\Gamma}$ with respect
  to $\cals$ such that $m_F(X,\partial X) = 1$ holds for every $F \in \calm$. Then:

  \begin{enumerate}
  \item\label{the:Checking_Poincare_duality:duality} The $\Gamma$-$CW$-pair
    $(X/\Gamma,\partial X/\Gamma)$ carries the structure of a finite Poincar\'e pair with
    respect to the orientation homomorphism $w \colon \Gamma \to \{\pm 1\}$ of
    Notation~\ref{not:w} and the fundamental class
    $[X/\Gamma, \partial X/\Gamma] \in H_d^{\Gamma}(X,\partial X;\IZ^w)$ defined in
    Notation~\ref{not:fundamental_classes};
         
  \item\label{the:Checking_Poincare_duality:indices} The map
    \[
      H_d(i_*(X,\partial X))\colon H_{d}^{\pi}(X,\partial X;\IZ^v) \to
      H_{d}^{\Gamma}(X,\partial X;\IZ^w)
    \]
    induced by the chain map $i_*(X,\partial X)$ of~\eqref{i_ast_for_C_ast_and_M} is an
    inclusion of infinite cyclic groups of index $|G|$ and the map induced by the transfer
    chain map of~\eqref{trf_ast_for_C_ast_and_M}
    \[
      H_d(\trf_*(X,\partial X)) \colon H_{d}^{\Gamma}(X,\partial X;\IZ^w)
      \xrightarrow{\cong} H_{d}^{\pi}(X,\partial X;\IZ^v)
    \]
    is an isomorphism. Moreover, the map
    \[
    H_d(i_*(\eub{\Gamma})) \colon H_{d}^{\pi}(E\pi;\IZ^v) \to
      H_{d}^{\Gamma}(\eub{\Gamma};\IZ^w)
    \]
    is an inclusion of infinite cyclic groups with index $|G|$.
  \end{enumerate}
\end{theorem}
\begin{proof}
  Let
  $i_*(X,\partial X) \colon \IZ^v \otimes_{\IZ\pi} i^*C_*(X,\partial X) \to \IZ^w
  \otimes_{\IZ\Gamma} C_*(X,\partial X)$ be the $\IZ$-chain map defined
  in~\eqref{i_ast_for_C_ast_and_M}. Define
  $i_*(\partial X) \colon \IZ^v \otimes_{\IZ\pi} i^*C_*(\partial X) \to \IZ^w
  \otimes_{\IZ\Gamma} C_*(\partial X)$ and
  $i_*(\eub{\Gamma},\partial \eub{\Gamma}) \colon \IZ^v \otimes_{\IZ\pi}
  i^*C_*(\eub{\Gamma},\partial \eub{\Gamma}) \to \IZ^w \otimes_{\IZ\Gamma}
  C_*(\eub{\Gamma},\partial \eub{\Gamma})$ analogously.  Consider the following
  commutative diagram
  \[
    \xymatrix{H_{d}^{\pi}(X,\partial X;\IZ^v) \ar[r]^-{\partial}
      \ar[d]_{H_d(i_*(X,\partial X))} & H_{d-1}^{\pi}(\partial X;\IZ^v) = H_{d-1}(\partial
      X/\pi) \ar[d]^{H_d(i_*(\partial X)) = H_{d-1}(q)}
      \\
      H_{d}^{\Gamma}(X,\partial X;\IZ^w) \ar[r]_-{\partial} & H_{d-1}^{\Gamma}(\partial
      X;\IZ^w) = H_{d-1}(\partial X/\Gamma),}
  \]
  where the horizontal arrows are boundary maps and
  $q \colon \partial X/\pi \to \partial X/\Gamma$ is the projection. We can determine
  $H_{d-1}(q)$ by
  \[
    \xymatrix{\bigoplus_{F \in \calm} \bigoplus_{G/\pr(F)} H_{d-1}(S_F) \ar[r]^-{\cong}
      \ar[d]_{\bigoplus_{F \in \calm} \bigoplus_{G/\pr(F)} H_{d-1}(\pr_{S_F})} &
      H_{d-1}(\partial X/\pi) \ar[d]^{H_{d-1}(q)}
      \\
      \bigoplus_{F \in \calm} H_{d-1}(S_F/F) \ar[r]^-{\cong} & H_{d-1}(\partial X/\Gamma).
    }
  \]
  
  We know already that $H_{d}^{\Gamma}(\eub{\Gamma},\partial \eub{\Gamma};\IZ^w)$ is an
  infinite cyclic group. We conclude from~\eqref{i_ast_circ_trf_ast_is_|g|_cdot_id} that
  the map
  \[
    H_d(i_*(X,\partial X)_*) \colon H_{d}^{\Gamma}(X,\partial X;\IZ^w) \to
    H_{d}^{\Gamma}(\eub{\Gamma},\partial \eub{\Gamma};\IZ^w)
  \]
  is injective and its image has finite index which divides $|G|$. Let the element
  $k \in \IZ$ with $k \ge 0$ be uniquely determined by the equation
  \begin{equation*}
    H_{d}(i_*(\eub{\Gamma},\partial \eub{\Gamma}))([\eub{\Gamma}/\pi,\partial \eub{\Gamma/\pi}])
    = k \cdot [\eub{\Gamma}/\Gamma,\partial \eub{\Gamma/\Gamma}].
  \end{equation*}
  Note that $k$ divides $|G|$.  We can identify $H_{d}(i_*(X,\partial X))$ and
  $H_{d}(i_*(\eub{\Gamma},\partial \eub{\Gamma}))$ by
  \[
    \xymatrix@!C=19em{ H_{d}^{\pi}(X,\partial X;\IZ^v) \ar[d]_{H_{d}(i_*(X,\partial X))}
      \ar[r]^-{H_d^{\pi}(p,\partial p;\IZ^v)}_-{\cong} & H_{d}^{\pi}(\eub{\Gamma},\partial
      \eub{\Gamma};\IZ^v)\ar[d]^{H_{d}(i_*(\eub{\Gamma},\partial \eub{\Gamma}))}
      \\
      H_{d}^{\Gamma}(X,\partial X;\IZ^w) \ar[r]_-{H_d^{\Gamma}(p,\partial
        p;\IZ^w)}^-{\cong} & H_{d}^{\Gamma}(\eub{\Gamma},\partial \eub{\Gamma};\IZ^w)}
  \]
  where $(p,\partial p) \colon (X,\partial X) \to (\eub{\Gamma},\partial \eub{\Gamma})$ is
  the projection.  We conclude
  \begin{equation}
    H_{d}(i_*(X,\partial X))([X/\pi,\partial X/\pi]) = k \cdot [X/\Gamma,\partial X/\Gamma].
    \label{k_for-(X,partial_X)}
  \end{equation}

  Consider the composite
  \[H_{d}^{\pi}(X,\partial X;\IZ^v) \xrightarrow{\partial } H_{d}^{\pi}(\partial X;\IZ^v)
    \xrightarrow{\cong} \bigoplus_{F \in \calm} \bigoplus_{G/\pr(F)} H_{d-1}(S_F).
  \]
  There is the $G$-action on its source given for $g\in G$ and any element
  $\widehat{g} \in \Gamma$, which is sent by the projection $\Gamma \to G$ to $g$, by the
  $\IZ$-chain map
  \[
    \IZ^v \otimes_{\IZ \pi} C_*(X,\partial X) \to \IZ^v \otimes_{\IZ \pi} C_*(X,\partial
    X), \quad (m \otimes x) \mapsto (m \cdot w(\widehat{g}) \otimes \widehat{g}^{-1}x).
  \]
  There is the $G$-action on the target given by permuting the summand according to the
  canonical $G$-action on $G/\pr(F)$.  One easily checks that the composite above is
  compatible with these $G$-actions.

  We have defined a specific $G$-action on $H_d^{\pi}(E\pi;\IZ^v)$ at the end of
  Subsection~\ref{subsec:The_orientation_homomorphism} and shown that it is trivial. The
  isomorphism $H_d^{\pi}(E\pi;\IZ^v) \xrightarrow{\cong} H_{d}^{\pi}(X,\partial X;\IZ^v)$
  obtained by the composite of the isomorphisms (or their inverses) appearing in the
  middle column of~\eqref{diagram_of_the_H_d-s} is compatible with these $G$-actions. Therefore   the
  $G$-action on $H_{d}^{\pi}(X,\partial X;\IZ^v)$ is trivial as well. Hence we can find a
  collection of integers $\{n_F \mid F \in \calm\}$ such that the image of
  $[X/\pi,\partial X/\pi]$ under the composite above has as entry in the summand
  $H_n(S_F)$ for $F \in \calm$ and $g\pr(F) \in G/\pr(F)$ the element $n_F \cdot [S_F]$.
  This element is sent under
  \[\bigoplus_{F \in \calm} \bigoplus_{G/\pr(F)} H_{d-1}(\pr_{S_F}) \colon \bigoplus_{F
      \in \calm} \bigoplus_{G/\pr(F)} H_{d-1}(S_F) \to \bigoplus_{F \in \calm}
    H_{d-1}(S_F/F)
  \]
  to $\{n_F \cdot |G| \cdot [S_F/F] \mid F \in \calm\}$, since
  $|G| = |F| \cdot |G/\pr(F)|$. Since the composite
  \[H_{d}^{\Gamma}(X,\partial X;\IZ^w) \xrightarrow{\partial} H_{d}^{\Gamma}(\partial X;\IZ^w)
    \xrightarrow{\cong} \bigoplus_{F \in \calm} H_{d-1}(S_F/F)
  \]
  sends $[X/\Gamma,\partial X/\Gamma]$ to $\{[S_F]\mid F \in \calm\}$, we get
  $k = n_F \cdot [G]$ for every $F \in \calm$. Since $k$ divides $[G]$ and $k$ and $n_F$
  are positive, we conclude $n_F = 1$ for every $F \in \calm$ and $k = |G|$.

  This implies that the maps
  \begin{eqnarray*}
    H_d(i_*(X,\partial X)) \colon H_{d}^{\pi}(X,\partial X;\IZ^v)
    & \to &
    H_{d}^{\Gamma}(X,\partial X;\IZ^w);
    \\
    H_d(i_*(\eub{\Gamma})) \colon H_{d}^{\pi}(E\pi;\IZ^v)
    & \to &
    H_{d}^{\Gamma}(\eub{\Gamma};\IZ^w),
  \end{eqnarray*}
  are inclusions of infinite cyclic groups with index
  $|G|$. We conclude from~\eqref{i_ast_circ_trf_ast_is_|g|_cdot_id} that the map
  $H_d(\trf_*) \colon H_{d}^{\Gamma}(X,\partial X;\IZ^w) \xrightarrow{\cong}
  H_{d}^{\pi}(X,\partial X;\IZ^v)$ is an isomorphism.

  We also conclude that the composite
  \[
    H_{d}^{\pi}(X,\partial X;\IZ^v) \xrightarrow{\partial } H_{d}^{\pi}(\partial X;\IZ^v)
    \xrightarrow{\cong} \bigoplus_{F \in \calm} \bigoplus_{G/\pr(F)} H_{d-1}(S_F)
  \]
  sends $[X/\pi,\partial X/\pi]$ to the element, which is given in any of the summands by
  the fundamental class $[S_F]$.

  It remains to show that $(X/\Gamma,\partial X/\Gamma)$ carries the structure of a finite
  Poincar\'e pair with respect to the orientation homomorphism
  $w \colon \Gamma \to \{\pm 1\}$ and the fundamental class
  $[X/\Gamma, \partial X/\Gamma] \in H_d^{\Gamma}(X,\partial X;\IZ^w)$.  Because
  of~\cite[Theorem~H]{Klein-Qin_Su(2019)} it suffices to show that $(X/\pi,\partial X\pi)$
  carries the structure of a finite Poincar\'e pair with respect to the orientation
  homomorphism $v \colon \pi \to \{\pm 1\}$ and the fundamental class
  $[X/\pi, \partial X/\pi] \in H_d^{\pi}(X,\partial X;\IZ^v)$. This follows from
  Lemma~\ref{lem:subtracting_a_Poincare_pair}~\eqref{lem:subtracting_a_Poincare_pair:conclusion_Y_2}
  applied in the case
  \begin{eqnarray*}
    Y
    & = &
          B\pi;
    \\
    Y_1
    & = &
          \coprod_{F \in \calm} G/\pr(F) \times D^d;
    \\
    Y_2
    & = &
          X/\pi;
    \\
    Y_0
    & = &
          \partial X/\pi = \coprod_{F \in \calm}  G/\pr(F) \times S^{d-1},
  \end{eqnarray*}
  using the assumption that there is a finite $d$-dimensional Poincar\'e $CW$-complex
  model for $B\pi$ with respect to the orientation homomorphism
  $v \colon \pi \to \{\pm 1 \}$ and fundamental classes
  $[B\pi] \in H_d^{\pi}(E\pi;\IZ^v)$, the fundamental class $[X/\pi,\partial X/\pi] \in H_d^{\pi}(X,\partial X;\IZ^v)$,
  and the preimage under the composite of isomorphisms
  \begin{multline*}
   \partial \colon H_d(\coprod_{F \in \calm} G/\pr(F) \times
   (D^d,S^{d-1}))  \xrightarrow{\cong} H_{d-1}(\coprod_{F \in \calm} G/\pr(F) \times  S^{d-1})
   \\
   \xrightarrow{\cong} \bigoplus_{F \in \calm} \bigoplus_{G/\pr(F)} H_{d-1}(S_F)
 \end{multline*}
 of the obvious element in $H_{d-1}(\coprod_{F \in \calm} G/\pr(F) \times  S^{d-1})$,
 which is given in any of the summands by  the fundamental class $[S_F]$. Now  the conditions
 about the fundamental classes
  appearing in Lemma~\ref{lem:subtracting_a_Poincare_pair}
  follow from the following commutative diagram with exact right row
  \[
    \xymatrix{& 0 \ar[d]
      \\
      & H_d^{\pi}(\widetilde{B\pi};\IZ^v) \ar[d]
      \\
      H_d^{\pi}(X,\partial X;\IZ^v) \oplus H_d(\coprod_{F \in \calm} G/\pr(F) \times
      (D^d,S^{d-1})) \ar[r]^-{\cong} \ar[d]^{\partial \oplus \partial} &
      H_d^{\pi}(\widetilde{B\pi},\partial X;\IZ^v) \ar[d]^{\partial}
      \\
      H_{d-1}^{\pi}(\partial X;\IZ^v) \oplus H_{d-1}(\coprod_{F \in \calm} G/\pr(F) \times
      S^{d-1}) \ar[r]^-{\id \oplus -\id} & H_{d-1}^{\pi}(\partial X;\IZ^v),}
  \]
  where we identify $H_{d-1}^{\pi}(\partial X;\IZ^v)$ and $H_{d-1}(\coprod_{F \in \calm} G/\pr(F) \times S^{d-1})$
  by the obvious isomorphism.
\end{proof}


\typeout{-------------------------- Section 8: Homotopy classification of slice complement models
  --------------------------}

\section{Homotopy classification of slice complement models}%
\label{sec:homotopy_classification_of_slice_models}

Throughout this section we make Assumption~\ref{ass:general_assumptions}.

For the remainder of this section we fix two oriented free $d$-dimensional slice systems $\cals$
and $\cals'$  satisfying condition (S), see Notation~\ref{def:condition_(S)}.
Recall that we have defined an orientation homomorphism
$w\colon \Gamma \to \{\pm 1\}$ in Notation~\ref{not:w}.
Let $(X,\partial X)$ and $(X',\partial X')$ respectively be slice complement models for
  $\eub{\Gamma}$ with respect to $\cals$ and $\cals'$ respectively. Then we have defined
fundamental classes
$[X/\Gamma,\partial X/\Gamma]$ and $[X/\Gamma,\partial X/\Gamma]$ in
Notation~\ref{not:fundamental_classes}, and integers $m(X,\partial X)_F$ and
$m(X',\partial X')_F$ satisfying
$m(X,\partial X)_F \equiv m(X',\partial X')_F \equiv 1 \mod |G|$ for $F \in \calm$
in~\eqref{integer_m_F}. The main result of   this
section will be

\begin{theorem}[Homotopy classification of slice
  models]\label{the:Homotopy_classification_of_slice_models}
   Suppose that  Assumption~\ref{ass:general_assumptions} is satisfied. 
  Let $(X,\partial X)$ and $(X',\partial X')$ respectively be slice complement models for
  $\eub{\Gamma}$ with respect to $\cals$ and $\cals'$ respectively.

  Then the following two assertions are equivalent:

  \begin{enumerate}
  \item\label{the:homotopy_classification_of_slice_models:existence:maps} There exists a
    $\Gamma$-homotopy equivalence
    $(f,\partial f) \colon (X,\partial X) \to (X',\partial X')$ of free
    $\Gamma$-$CW$-pairs with the properties that the $\Gamma$-map $\partial f$ extends to
    a $\Gamma$-homotopy equivalence of $\Gamma$-$CW$-pairs
    $ C(\partial X) \to C(\partial X')$ and the isomorphism
    $H^{\Gamma}_d(X,\partial X;\IZ^w)\xrightarrow{\cong} H^{\Gamma}_d(X',\partial
    X';\IZ^w)$ induced by $(f,\partial f)$ sends $[X,\partial X]$ to $[X',\partial X']$;

  \item\label{the:homotopy_classification_of_slice_models:existence:m_F} We have
    $m(X,\partial X)_F = m(X',\partial X')_F$ for every $F \in \calm$ with $|F| \ge 3$
    and $m(X,\partial X)_F = \epsilon_F  \cdot m(X',\partial X')_F$ for every $F \in \calm$ with $|F| = 2$
    for some $\epsilon_F \in \{\pm1 \}$.
  \end{enumerate}
\end{theorem}

Its proof needs some preparations.
Recall that up $F$-homotopy there is precisely one orientation preserving
$F$-homotopy equivalence, see Lemma~\ref{lem:basics_about_free_slice_systems},
\begin{equation}
  s_F \colon S_F\to S'_F.
  \label{s_F}
\end{equation}

\begin{lemma}\label{lem:maps_of_slice_models_homotopy_equivalence}
  Let $\partial u \colon \partial X \to \partial X'$ be the $\Gamma$-homotopy equivalence
  given by the disjoint union of the $\Gamma$-homotopy equivalences
  $\id_{\Gamma} \times _F s_F \colon \Gamma \times_F S_F \to \Gamma \times_F S_F'$ for
  $F \in \calm$.  Suppose that there is a $\Gamma$-map $u \colon X \to X'$ extending
  $\partial u \colon \partial X \to \partial X'$.

  Then $(u,\partial u) \colon (X,\partial X) \to (X',\partial X')$ is a $\Gamma$-homotopy
  equivalence of $\Gamma$-$CW$-pairs and the isomorphism
  $H_d(X,\partial X;\IZ^w) \xrightarrow{\cong} H_d(X',\partial X';\IZ^w)$ induced by
  $(u,\partial u)$ sends $[X/\Gamma,\partial X/\Gamma]$ to
  $[X'/\Gamma,\partial X'/\Gamma]$.
\end{lemma}
\begin{proof}
  Each map $s_F \colon S_F \to S_F'$ extends to $F$-map $D_F \colon D_F'$ by taking the
  cone.  Hence there exists a $\Gamma$-homotopy equivalence
  $C(\partial u) \colon C(\partial X) \to C(\partial X')$ extending $\partial u$.  We
  obtain a $\Gamma$-map
  \[u \cup_{\partial u} C(\partial u) \colon X \cup_{\partial X} C(\partial X) \to X'
    \cup_{\partial X'} C(\partial X').
  \]
  Since the source and target of this map are $\Gamma$-homotopy equivalent to
  $\eub{\Gamma}$, it is a $\Gamma$-homotopy equivalence.  Since
  $u \cup_{\partial u} C(\partial u)$, $\partial u$ and $C(\partial u)$ induce homology
  equivalences, the map $H_n(u) \colon H_n(X) \to H_n(X')$ is  bijective for all
  $n \ge 0$.  Since $X$ and $X'$ are simply connected by
  Lemma~\ref{lem:characterization_of_slice-models}, the map $u \colon X \to X'$ is a
  non-equivariant homotopy equivalence.  Since $X$ and $X'$ are free
  $\Gamma$-$CW$-complexes, $u \colon X \to X'$ is a $\Gamma$-homotopy equivalence.  Since
  $\partial u \colon \partial X \to \partial X'$ is a $\Gamma$-homotopy equivalence,
  $(u,\partial u) \colon (X,\partial X) \to (X',\partial X')$ is a $\Gamma$-homotopy
  equivalence of $\Gamma$-$CW$-pairs.

  One easily checks by inspecting the definitions and the commutative diagram
  \[
    \xymatrix{H^{\Gamma}_d(X,\partial X;\IZ^w) \ar[d]^{\cong}_{H^{\Gamma}_d(u,\partial
        u;\IZ^w)} \ar[r]^-{\partial} & H^{\Gamma}_{d-1}(\partial X;\IZ^w)
      \ar[d]_{\cong}^{H^{\Gamma}_d(\partial u;\IZ^w)}
      \\
      H^{\Gamma}_d(X',\partial X';\IZ^w) \ar[r]_-{\partial} & H^{\Gamma}_{d-1}(\partial
      X';\IZ^w) }
  \]
  that the isomorphism
  $H_d(X,\partial X;\IZ^w) \xrightarrow{\cong} H_d(X',\partial X';\IZ^w)$ induced by
  $(u,\partial u)$ sends $[X/\Gamma,\partial X/\Gamma]$ to
  $[X'/\Gamma,\partial X'/\Gamma]$, since the horizontal arrows and the right vertical arrow
  respects the fundamental classes.
\end{proof}

\begin{lemma}\label{lem:extensions_over_Gamma_and_pi}
  Suppose additionally that $(X,\partial X)$ is a Poincar\'e slice complement model and
  $m_F(X,\partial X) = 1$ hold for all $F \in \calm$. (Such $(X,\partial X)$ exists by
  Lemma~\ref{the:existence_Poincare_slice_models}.) 
  Let $v_F \colon S_F \to S_F'$ be the
  $F$-map uniquely determined up to $F$-homotopy by the property that it sends $[S_F/F]$
  to $m_F(X',\partial X') \cdot [S_F'/F]$.  (It exists, since we have
  $\mu_F(X',\partial X') \equiv 1 \mod |G|$.)  Let
  $\partial u \colon \partial X \to \partial X'$ be the $\Gamma$-map given by the disjoint
  union of the $\Gamma$-maps $\Gamma \times_F v_F$.

  Then the following assertions are equivalent

  \begin{enumerate}
  \item\label{lem:extensions_over_Gamma_and_pi:Gamma} There exists a $\Gamma$-map
    $u \colon X \to X'$ extending the $\Gamma$-map
    $\partial u \colon \partial X \to \partial X'$;

  \item\label{lem:extensions_over_Gamma_and_pi:pi} There exists a $\pi$-map
    $u'  \colon i^*X \to i^*X'$ extending the $\pi$-map
    $i^*u \colon i^*\partial X \to i^*\partial X'$.
  \end{enumerate}
\end{lemma}
\begin{proof}
  Obviously~\eqref{lem:extensions_over_Gamma_and_pi:Gamma}
  implies~\eqref{lem:extensions_over_Gamma_and_pi:pi}.  The
  implication~\eqref{lem:extensions_over_Gamma_and_pi:pi}
  $\implies$~\eqref{lem:extensions_over_Gamma_and_pi:Gamma} is proved by equivariant
  obstruction theory as follows.

  Note that $(X,\partial X)$ is a free $\Gamma$-$CW$-pair and $\dim(X) = d$. The
  $\Gamma$-$CW$-complex $X'$ is $(d-2)$-connected by
  Lemma~\ref{lem:characterization_of_slice-models}~%
\eqref{lem:characterization_of_slice-models:homology}. Hence we get from equivariant
  obstruction theory an exact sequence
  \[[X,X']^{\Gamma} \to [\partial X,X']^{\Gamma} \xrightarrow{o^{\Gamma}}
    H^d_{\Gamma}(X,\partial X;\pi_{d-1}(X')).
  \]
  This is explained for finite $\Gamma$ for instance in~\cite[pages~119 - 120]{Dieck(1987)},
  the condition that $\Gamma$ is finite is not needed at all. The
  construction is compatible with restriction.  So we get a commutative diagram with exact
  rows
  \[
    \xymatrix{[X,X']^{\Gamma} \ar[r] \ar[d]^{i^*} & [\partial X,X']^{\Gamma}
      \ar[r]^-{o^{\Gamma}} \ar[d]^{i^*} & H^d_{\Gamma}(X,\partial X;\pi_{d-1}(X'))
      \ar[d]^{i^*}
      \\
      [i^*X,i^*X']^{\pi} \ar[r] & [i^*\partial X,i^*X']^{\pi} \ar[r]^-{o^{\pi}} &
      H^d_{\pi}(i^*X,i^*\partial X;\pi_{d-1}(i^*X')).  }
  \]
  Hence it suffices to show that
  $i^* \colon H^d_{\Gamma}(X,\partial X;\pi_{d-1}(X')) \to H^d_{\pi}(i^*X,i^*\partial
  X;\pi_{d-1}(i^*X'))$ is injective. This will be done by a cohomological version of the
  transfer argument appearing in Subsection~\ref{subsec:Transfer}, which we explain next.
  Recall from the definitions
  \begin{eqnarray*}
    H^d_{\Gamma}(X,\partial X;\pi_{d-1}(X'))
    & = &
          H^d(\hom_{\IZ \Gamma}(C_*(X,\partial X),\pi_{d-1}(X')));
    \\
    H^d_{\Gamma}(X,\partial X;\pi_{d-1}(X'))
    & = &
          H^d(\hom_{\IZ \pi}(i^*C_*(X,\partial X),i^*\pi_{d-1}(X'))).
  \end{eqnarray*}
  The group $G$-acts on $\hom_{\IZ \pi}(i^*C_*(X,\partial X),i^*\pi_{d-1}(X'))$ in the
  obvious way.  We have
  \[
    \hom_{\IZ \pi}(i^*C_*(X,\partial X),i^*\pi_{d-1}(X'))^G = \hom_{\IZ
      \Gamma}(C_*(X,\partial X),\pi_{d-1}(X')).
  \]
  If we put $D^* = \hom_{\IZ \pi}(i^*C_*(X,\partial X),i^*\pi_{d-1}(X'))$, then
  $i^* \colon H^d_{\Gamma}(X,\partial X;\pi_{d-1}(X')) \to H^d_{\pi}(i^*X,i^*\partial
  X;\pi_{d-1}(i^*X'))$ can be identified with the map
  $H^d(j^*) \colon H^d((D^*)^G) \to H^d(D^*)$ for the inclusion $j^* \colon (D^*)^G \to D^*$.
  Multiplication with the norm element $N := \sum_{g \in G} g \in \IZ G$ defines a
  $\IZ$-chain map $t^* \colon D^* \to (D^*)^G$ such that
  $t^* \circ j^* = |G| \cdot \id_{(D^*)^G}$.  Hence $j^*$ is injective, if
  $|G| \cdot \id \colon H^d((D^*)^G) \to H^d((D^*)^G) $ is injective.  Therefore it
  suffices to show that
  $H^d((D^*)^G) = H^d(\hom_{\IZ \Gamma}(C_*(X,\partial X),\pi_{d-1}(X')))$ is torsionfree.
  
  This follows from the following string of isomorphisms
  \begin{eqnarray*}
    H^d(\hom_{\IZ \Gamma}(C_*(X,\partial X),\pi_{d-1}(X')))
    &\cong &
             H^d(\hom_{\IZ \Gamma}(C_*(X,\partial X),H_{d-1}(X')))
    \\
    &\cong &
             H^d(\hom_{\IZ \Gamma}(C_*(X,\partial X),H_{d-1}(\partial X')))
    \\
    &\cong &
             H_0(C_*(X) \otimes_{\IZ\Gamma} H_{d-1}(\partial X'))
    \\
    &\cong &
             \IZ \otimes_{\IZ \Gamma} H_{d-1}(\partial X') 
    \\
    &\cong &
             H_{d-1}(\partial X'/\Gamma)
    \\
    &\cong &
             \bigoplus_{F \in \calm}H_{d-1}(S'_F/F).
  \end{eqnarray*}
  The first isomorphism comes from the Hurewicz homomorphism
  $\pi_{d-1}(X') \xrightarrow{\cong} H_{d-1}(X')$, which is bijective, as $X'$ is
  $(d-2)$-connected by Lemma~\ref{lem:characterization_of_slice-models}.  The second
  isomorphism comes from the $\IZ\Gamma$-isomorphism
  $H_{d-1}(\partial X') \xrightarrow{\cong} H_{d-1}(X')$, which is bijective, since for
  $n \in \{(d-1),d\}$ we get
  \[
    H_n(X',\partial X') \cong H_n(\eub{\Gamma},\partial \eub{\Gamma}) \cong H_n(\eub{\Gamma}) = 0
  \]
  using the homotopy $\Gamma$-pushout~\eqref{canonical_homotopy_Gamma_pushout_for_a_slice_model}.
  The third   isomorphism is a consequence of the assumption that $(X/\Gamma,\partial X/\Gamma)$ is a
  Poincar\'e pair. The fourth and fifth isomorphism come from the fact that $X$ is a
  connected free $\Gamma$-$CW$-complex and the functor
  $- \otimes_{\IZ \Gamma} H_{d-1}(\partial X')$ is right exact. The last isomorphism is
  obvious.  Note that $H_{d-1}(S_F'/F)$ is infinite cyclic and hence torsionfree.
\end{proof}

\begin{remark}\label{rem:obstruction}
  Suppose that we are in the situation of  Lemma~\ref{lem:extensions_over_Gamma_and_pi}.
  Consider the following composite
  \[
    \alpha \colon [\partial X,\partial X']^{\Gamma} \xrightarrow{j_*} [\partial X,X']^{\Gamma}
    \xrightarrow{o^{\Gamma}} H^d_{\Gamma}(X,\partial X;\pi_{d-1}(X'))
    \xrightarrow{\cong} \bigoplus_{F \in \calm}H_{d-1}(S'_F/F),
  \]
  where the first map is given by composition with the inclusion
  $j \colon \partial X' \to X'$, the second  is given by the equivariant obstruction, and
  the third map is the isomorphism appearing in the proof of
  Lemma~\ref{lem:extensions_over_Gamma_and_pi}.  One may guess what this composition is
  for $\partial u$, namely, we expect
  \begin{equation}
    \alpha([\partial u]) = \mu(X',\partial X') - H_{d-1}(\partial u)(\mu(X,\partial X)).
    \label{putative_equality}
  \end{equation}
  This makes sense, since the existence of an extension of $j \circ \partial u$ to a
  $\Gamma$-map $X \to X'$ implies that
  $\mu(X',\partial X') - H_{d-1}(\partial u)(\mu(X,\partial X))$ vanishes. Since
  $(X,\partial X)$ is by assumption a Poincar\'e slice complement model, we conclude from
  Lemma~\ref{lem:necessary_condition_for_slice_model_to_be_Poincare} that
  $\mu(X,\partial X) = s$. This implies
  $\mu(X',\partial X') - H_{d-1}(\partial u)(\mu(X,\partial X)) = 0$ and hence the
  existence of the $\Gamma$-extension $u$ of the $\Gamma$-map $\partial u$
  follows from obstruction theory.

  However, it is not so easy to check equation~\eqref{putative_equality} and we will not do this here.
  Instead we will construct the desired
  $\pi$-extension of the $\pi$-map $i^*\partial u$ directly
  and use the equivariant obstruction theory only to reduce the
  problem from $\Gamma$ to $\pi$ by Lemma~\ref{lem:extensions_over_Gamma_and_pi}.
\end{remark}

In the next step we construct a specific model for $(i^*X,i^*\partial X)$.

Because of Lemma~\ref{lem:H_inside_a_Poincare_pair} we can assume that we have a special
model for $B\pi = H \cup_z \widehat{B\pi}$ for $B\pi$. Put
\[
  J := \coprod_{F \in \calm} G/\pr(F),
\]
where $\pr \colon \Gamma \to G$ is the projection. Define
\[
  \tau \colon J \to \calf
\]
to be the obvious projection, which sends the summand $G/\pr(F)$ belonging to $F \in \calm$
to $F$.

Choose for every $j \in J$ an embedded disk $D_j^H \subseteq H \setminus \partial H$ such
that the disks for different $j $ are disjoint.  Put $S_j^H = \partial D_j^H$. Note that
the superscript $H$ shall remind the reader that these disks and spheres lie in the interior of $H$. Let
\[
  \overline{Y} = B\pi \setminus \coprod_{j \in J} \inte(D_j^H)
\]
be obtained from $B\pi = H \cup_z \widehat{B\pi}$ by deleting the interiors of these
embedded disks $D_j^H$.  Define
\begin{eqnarray*}
  \overline{\partial Y} & = & \coprod_{j \in J} S_j^H;
  \\
  C(\overline{\partial Y}) & = & \coprod_{j \in J} D_j^H.
\end{eqnarray*}
Then $\overline{Y} \cap H$ is a compact smooth manifold, whose boundary is the disjoint
union of $\partial H$ and $\overline{\partial Y}$ and we have
\[
  B\pi = (\overline{Y} \cap H) \cup_{\partial H \amalg \overline{\partial Y}}
  (\widehat{B\pi} \amalg C(\overline{\partial Y})) = \overline{Y} \cup_{\overline{\partial
      Y}} C(\overline{\partial Y}).
\]
Let $Y$ and $\partial Y$ be the free $\pi$-$CW$-complexes obtained by taking the
preimage of $\overline{Y}$ and $\partial \overline{Y}$ under the universal covering
$E\pi \to B\pi$. Note that then we can identify $\overline{Y} = Y/\pi$ and
$\overline{\partial Y} = \partial Y/\pi$ and $\partial \overline{Y} = \coprod_{j \in J} \pi \times S_j^H$.

Since $\partial X'/\Gamma = \coprod_{F \in \calm} S'_F/F$ holds, the $\pi$-space $i^* \partial X'$
can be written as
\[
  i^*\partial X' = \coprod_{F \in \calm} \bigl(\coprod_{G/\pr(F)} \pi \times S'_F\bigr) =
  \coprod_{j \in J} \pi \times S_j',
\]
where $S'_j$ is a copy of $S_{\tau(j)}'$.  For $j \in J$ choose a map
$\partial v_j \colon S_j^H \to S_j$, whose degree is $m_{\tau(j)}(X',\partial X')$.  Let
$\partial v \colon \partial Y \to i^*\partial X'$ be the disjoint union of the $\pi$-maps
$\id_{\pi} \times \partial v_j \colon \pi \times S_j^H \to \pi \times S_j'$.  Since each
map $\partial v_j \colon S_j^H \to S_j'$ can be extended to a map $D_j^H \to D_j'$ by
coning, we get an extension of $\partial v \colon \partial Y \to \partial X'$ to a $\pi$-map
$C(\partial v) \colon C( \partial Y) \to C(i^*\partial X')$.

\begin{lemma}\label{lem:extending_partial_v}
  The $\pi$-map $\partial v \colon \partial Y \to i^*\partial X'$ extends to a $\pi$-map
  $v \colon Y \to i^*X'$.
\end{lemma}
\begin{proof} We begin with explaining that we can assume without loss of generality that
  each $S_F'$ is the standard sphere with its standard orientation and hence each $D_F'$
  is the standard disk. Namely, we can replace $(i^*X',i^*\partial X')$ by a
  $\pi$-homotopy equivalent pair $(X'',\partial X'')$ such that $\partial X''$ is a
  disjoint union of standard spheres, by the following construction.  Choose for every $j$
  an orientation preserving homotopy equivalence $\partial g_j \colon S_j' \to S_j''$ with a
  copy of  the standard sphere of dimension $(d-1)$ with its standard orientation as target.
  Define $\partial X'' = \coprod_{j \in J} \pi \times S''_j$ and let
  $\partial g\colon i^*\partial X' \to \partial X''$ be the $\pi$-homotopy equivalence
  given by $ \coprod_{j \in J} \id_{\pi} \times \partial g_j$. Define $X''$ and the $\pi$-map $g \colon i^*X' \to X''$
  by the   $\pi$-pushout
  \[
    \xymatrix{i^*\partial X' \ar[r]^{\partial g} \ar[d] & \partial X'' \ar[d]
      \\
      i^* X' \ar[r]_g & X''.  }
  \]
  Since $\partial g$ is a $\pi$-homotopy equivalence,
  $(g,\partial g) \colon (X',\partial X') \to (X'',\partial X'')$ is a $\pi$-homotopy
  equivalence of free $\pi$-$CW$-pairs.  Obviously it suffices to show that the map
  $\partial g \circ \partial v\colon \partial Y \to \partial X''$ extends to a $\pi$-map
  $Y \to X''$. Because we may replace $(i^*X',i^* \partial X')$ with $(X'',\partial X'')$,
  we can assume without loss of generality that each $S'_j$ is the $(d-1)$-dimensional
  standard sphere with its standard orientation and each $D'_j$ is the $d$-dimensional
  standard disk.

  Put $\overline{X} := X'/\pi$ and $\overline{\partial X} := \partial X'/\pi$.  Let
  $\overline{\partial v} \colon \overline{\partial Y} \to \overline{X'}$ be
  $\partial v/\pi$.  Since
  $\overline{Y} \cup_{\partial \overline{Y}} C(\partial \overline{Y})$ and
  $\overline{X'} \cup_{\partial \overline{X'}} C(\partial \overline{X'})$ are models for
  $B\pi$, the map
  \[
    C(\overline{\partial v}) \colon C(\overline{\partial Y}) = \coprod_{j \in J} D_j^H \to
    C(\overline{\partial X'}) = \coprod_{j \in J} D'_j
  \]
  extends to a homotopy equivalence
  \[
    f \colon \overline{Y} \cup_{\partial \overline{Y}} C(\partial \overline{Y})
    \xrightarrow{\simeq} \overline{X'} \cup_{\partial \overline{X'}} C(\partial
    \overline{X'})
  \]
  inducing the identity on $\pi$. Since the inclusion
  $\overline{X'} \to \overline{X'} \cup_{\coprod_{j \in J} S_j'} \coprod_{F} D_j'$ is
  $(d-1)$-connected, $\widehat{B\pi}$ is a ($d-2)$-dimensional $CW$-complex
  and $\partial H \to H$ is a cofibration,
  we can arrange  that $f(\widehat{B\pi})\subseteq \overline{X'}$ holds
  without altering $f$ on $H$ and the homotopy class of $f$.  In particular $f(\partial H) \subseteq X'$.

  Now we can change
  \[
    f|_{\overline{Y} \cap H} \colon \overline{Y} \cap H = H \setminus \coprod_{j \in J}
    \inte(D_j^H) \to \overline{X'} \cup_{\coprod_{j \in J} S_j'} \coprod_{j \in J} D'_j
  \]
  up to homotopy relative $\partial H$ such that it is transversal to each origine
  $0'_j \in D'_j$, since $D_F$ is a compact smooth manifold containing $0_F$ in interior
  and $\overline{Y} \cap H$ is a compact smooth manifold with boundary
  $\partial (\overline{Y} \cap H)$ such that $f(\partial (\overline{Y} \cap H))$ does not
  contain any of the points $o'_F$.  Furthermore we can arrange that for every $j \in J$
  the preimage $(f|_{\overline{Y} \cap H})^{-1}(D_j') = f^{-1}(D_j) \setminus D^H_F$ is a
  disjoint union of disks $\coprod_{i \in I_j} D_{j,i}^H$ for a finite set $I_j$ and
  $f|_{D[H]_{j,i}} \colon (D_{j,i}^H,S_{j,i}^H) \to (D_j',S_j')$ is a homeomorphism of
  pairs for $S_{j,i}^H = \partial D_{j,i}^H$. Let $\delta_{j,i} \in \{\pm 1\}$ be the
  local degree of the homeomorphism $f|_{S^H_{j,i}} \colon S^H_{j,i} \to S'_j$.  Let
  $\inte(D^H_j)$, $\inte(D^H_{j,i})$, and $\inte(D'_j)$ denote the interior of $D^H_j$,
  $D^H_{j,i}$ and $D'_j$. We abbreviate
  \begin{eqnarray*}
    Z'
    & := &
           \overline{X'} \cup_{\partial \overline{X'}} C(\partial \overline{X'})
           = \overline{X'} \cup_{\coprod_{j \in J} S_j'} \coprod_{j \in J} D_j';
    \\
    A
    & := &
           B\pi \setminus \bigl(\coprod_{j \in J} \inte(D^H_j) \amalg \coprod_{i \in I_j} \inte(D^H_{j,i})\bigr).
  \end{eqnarray*}
  Next we construct the following commutative diagram, where $p \colon E\pi \to B\pi$ and
  $p' \colon \widetilde{Z'} \to Z$ are the universal coverings
  \[
    \xymatrix@!C=15em{ H_d^{\pi}(E\pi;\IZ^v) \ar[r] \ar[d] &
      H_d^{\pi}(\widetilde{Z'};\IZ^v) \ar[d]
      \\
      H_d^{\pi}\bigl(E\pi, p^{-1}(A);\IZ^v\bigr) \ar[r] & H_d^{\pi}\bigl(\widetilde{Z'},
      p'^{-1}(Z' \setminus (\coprod_{j \in J} \inte(D_j')));\IZ^v\bigr)
      \\
      H_d\bigl(\coprod_{j \in J} (D^H_j,S^H_j) \amalg (\coprod_{i \in I_j} (D^H_{j,i},
      S^H_{j,i}))\bigr) \ar[u]^{\cong} \ar[r] & H_d\bigl(\coprod_{j \in J}
      (D'_j,S'_j)\bigr) \ar[u]_{\cong}
      \\
      \bigoplus_{j \in J} \left(H_d(D_j^H,S_j^H) \oplus \bigoplus_{i \in I_j}
        H_d(D_{j.i}^H,S_{j.i}^H) \right) \ar[r] \ar[d] \ar[u] & \bigoplus_{j \in J}
      H_d(D'_j,S'_j) \ar[d] \ar[u]
      \\
      \bigoplus_{j \in J} \left(H_{d-1}(S_j^H) \oplus \bigoplus_{i \in I_j}
        H_{d-1}(S_{j.i}^H) \right) \ar[r] & \bigoplus_{j \in J} H_{d-1}(S'_j) }
  \]
  The uppermost two vertical arrows are given by the obvious inclusions.  The vertical
  arrows pointing upwards are the isomorphisms are given by excision or by the disjoint
  union axiom.  The lower most vertical arrows are given by boundary homomorphisms.  All
  vertical arrows are induced by the homotopy equivalence
  $f \colon B\pi = \overline{Y} \cup_{\overline{\partial Y}} C(\overline{\partial Y}) \to
  Z' := \overline{X'} \cup_{\partial \overline{X'}} C(\partial \overline{X'})$.

  The fundamental class $[B\pi]$ is sent under the composite of the four vertical
  arrows (or their inverses) of the left column to the element in the left lower corner
  $\bigoplus_{j \in J} \left(H_{d-1}(S_j^H) \oplus \bigoplus_{i \in I_j}
    H_{d-1}(S_{j.i}^H) \right) $, which is given for each summand by the fundamental class
  of the corresponding sphere.  The fundamental class $[B\pi]$ is sent under the uppermost
  vertical arrow to the fundamental class $[Z']$. The fundamental class of $[Z']$ is sent
  under the under the composite of the four vertical arrows (or their inverses) of the
  right column to $(m_{\tau(j)}(X',\partial X') \cdot [S_j])_{j \in J}$. Given $j \in J$,
  the lowermost vertical arrow sends by construction the fundamental class
  $[S_j^H] \in H_{d-1}(S_j^H) $ to
  $m_{\tau(j)}(X',\partial X') \cdot [S'_j] \in H_{d-1}(S'_j)$ and by the definition of
  $\delta_{j,i}$ the fundamental class $[S_{j,i}^H] \in H_{d-1}(S_{j,i}^H)$ to
  $\delta_{j,i} \cdot [S_j']$ in $H_{d-1}(S'_j)$ for every $i \in F$. Since the diagram
  commutes, we conclude for every $j \in J$
  \[
    m_{\tau(j)}(X',\partial X') + \sum_{i \in I_j} \delta_{j,i} = m_{\tau(j)}(X',\partial
    X').
  \]
  Hence we get $\sum_{i \in I_j} \delta_{j,i} = 0$ for every $j \in J$.

  Next we show that we can change $f$ up to homotopy relative
  $\coprod_{j \in J} D_j^H \amalg \widehat{B\pi}$ such that each $I_j$ is empty.  We use
  induction over the cardinality of $\coprod_{j \in J} I_j$. The induction beginning
  $|\coprod_{j \in J} I_j| = 0 $ is trivial, the induction step done as follows.  Choose
  $j \in J$ with $I_j \not= \emptyset$. Since $\sum_{i \in I_j} \delta_{j,i} = 0$ and each
  element $\delta_{j,i}$ belongs to $\{\pm 1\}$, we can find $i_+$ and $i_- \in I_j$ with
  $\delta_{j,i_+} = 1$ and $\delta_{j,i_-} = -1$. Choose an embedded arc in $H$ joining a
  point $x_+ \in S^H_{j,i_+}$ to a point $x_- \in S^H_{j,i_-}$ such that the intersection
  of the arc with $\coprod_{j \in J} D^H_F\amalg (\coprod_{i \in I_j} D^H_{j,i})$ is
  $\{x_+,x_-\}$ and the arc meets $S^H_{j,i_+}$ and $S^H_{j,i_-}$ transversely. Then we
  can thicken this arc to a small tube $T$ in the obvious way such that the intersection
  of $T$ with $\coprod_{j\in J} D^H_j \amalg (\coprod_{i \in I_j} D^H_{j,i})$ is contained
  in small neighbourhoods of $x_+$ in $S^H_{j,i_+}$ and $x_-$ in $S^H_{j,i_+}$, which are
  diffeomorphic to $(d-1)$-dimensional discs. The union
  $D^H_{j,i_+} \cup T \cup D^H_{i_-,j}$ is diffeomorphic to a disk $D^d$. One can change
  $f$ up to homotopy on a small neighborhood of the tube such that $f$ is constant on the
  tube. Then $f$ induces a map $f_{D^d} \colon (D^d,S^{d-1}) \to (D^H_j,S^H_j)$ such that
  the degree of $f|_{S^{d-1}} \colon S^{d-1} \to S_j^H$ is
  $\delta_{j,i_+} + \delta_{j,i_.} = 0$.  We conclude the map
  $f|_{S^{d-1}} \colon S^{d-1} \to S_F^H$ is nullhomotopic.  Hence we can change $f$ up to
  homotopy relative $B\pi \setminus D^d$ such that $f(D^d)$ does not meet $0_j'$.  Thus we
  get rid of the points $x_+$ and $x_-$ and have made the cardinality of
  $\coprod_{j \in J} I_j$ smaller.  This finishes the proof that we can change $f$ up to
  homotopy  such that each $I_j$ is empty, or, equivalently, such that
  $f^{-1}(0'_F) = \{0^H_F\}$ holds, where $0'_j \in D'_j$ and $0^H_j \in D^H_j$ are the
  origines.

  Since the inclusion
  $\overline{X'} \to \overline{X'} \cup_{\coprod_{j \in J} S'_F} \coprod_{j \in J} D'_j
  \setminus \{0_F\}$ admits a retraction relative $\overline{X'}$ and
  $f(\coprod_{ j\in J } S_j^H) = \coprod_{j \in J} S_j$, we can change
  $f|_{\overline{Y}} \colon \overline{Y} \to \overline{X'} \cup_{\coprod_{j \in J} S^H_j}
  \coprod_{j \in J} D^H_j \setminus \{0^H_j\}$ relative $\coprod_{j \in J} S_j^H$ to a map
  $v \colon \overline{Y} \to \overline{X'}$. By construction $v$ extends $\partial v/\pi$.
  Hence by passing to the universal coverings, we obtain a $\pi$-map $v \colon Y \to i^*X'$
  extending $\partial v$.
\end{proof}

\begin{lemma}\label{lem:extensions_over_Gamma_from_partial_X_to_X}
  Suppose additionally  that $(X,\partial X)$ is a Poincar\'e slice complement model and
  $m_F(X,\partial X) = 1$ holds for every  $F \in \calm$.   Let $v_F \colon S_F \to S_F'$
  be the $F$-map uniquely determined up to $F$-homotopy by the property that it sends
  $[S_F/F]$ to $m_F(X',\partial X') \cdot [S_F'/F]$.  Let
  $\partial u \colon \partial X \to \partial X'$ be the $\Gamma$-map given by the disjoint
  union of the $\Gamma$-maps $\Gamma \times_F v_F$.

  Then there exists a $\Gamma$-map $u \colon X \to X'$ extending the $\Gamma$-map
  $\partial u \colon \partial X \to \partial X'$.
\end{lemma}
\begin{proof}
  Firstly we apply Lemma~\ref{lem:extending_partial_v} to $(i^*X,i^*\partial X)$ instead
  of $(i^*X,i^*\partial X)$.  Since $m_F(X',\partial X') = 1$ holds for every $F \in \calm$,
  the map   $\partial v_X \colon \partial Y \to i^*\partial X$ appearing in
  Lemma~\ref{lem:extending_partial_v} is a $\pi$-homotopy equivalence. 
  Moreover, by
  Lemma~\ref{lem:extending_partial_v} we get an extension of $\partial v_X$ to a
  $\Gamma$-map $v_X\colon Y \to X$. The same argument as it appears in
  Lemma~\ref{lem:maps_of_slice_models_homotopy_equivalence},
   but for $\pi$ instead of $\Gamma$, shows
  that $(v_X,\partial v_X) \colon (Y,\partial Y) \to (i^*X,i^*\partial X)$ is a $\pi$-homotopy
  equivalence.

  From Lemma~\ref{lem:extending_partial_v} 
  we obtain an extension of $\partial v \colon \partial Y \to \partial X'$
  to a $\pi$-map $v \colon Y \to i^*X'$.
  Since $\partial  u \circ \partial v_X$ and $\partial v$ are $\pi$-homotopic,
  the $\pi$-map $i^*\partial u \colon \partial X \to \partial X'$ extends to
  a $\pi$-map $u' \colon i^* X \to i^*X'$.  Finally we conclude from
  Lemma~\ref{lem:extensions_over_Gamma_and_pi} that the $\Gamma$-map
  $\partial u \colon \partial X \to \partial X'$ extends to $\Gamma$-map
  $u \colon X \to X'$.
\end{proof}

Now we are ready to give the proof of
Theorem~\ref{the:Homotopy_classification_of_slice_models}.  

\begin{proof}[Proof of Theorem~\ref{the:Homotopy_classification_of_slice_models}]
  The implication\eqref{the:homotopy_classification_of_slice_models:existence:maps}
  $\implies$~\eqref{the:homotopy_classification_of_slice_models:existence:m_F} follows from the
  definitions, Lemma~\ref{lem:partial_f_yields_partial_f_F}~\eqref{lem:partial_f_yields_partial_f_F:S_F_to_S_F_prime}
  and the commutative diagram
  \[
    \xymatrix{H^{\Gamma}_d(X,\partial X;\IZ^w) \ar[d]^{\cong}_{H^{\Gamma}_d(f,\partial
        f;\IZ^w)} \ar[r]^-{\partial} & H^{\Gamma}_{d-1}(\partial X;\IZ^w)
      \ar[d]_{\cong}^{H^{\Gamma}_d(\partial f;\IZ^w)}
      \\
      H^{\Gamma}_d(X',\partial X';\IZ^w) \ar[r]_-{\partial} & H^{\Gamma}_{d-1}(\partial
      X';\IZ^w) }
  \]
  The implication~\eqref{the:homotopy_classification_of_slice_models:existence:m_F}
  $\implies$~\eqref{the:homotopy_classification_of_slice_models:existence:maps} is proved as follows.
  After possibly changing the orientations of $S_F$ for $F \in \calm$ satisfying $|F| = 2$, we can find a 
  Poincar\'e slice complement model $(Y,\partial Y)$ satisfying $m_F(Y,\partial Y) = 1$
  for all $F \in \calm$ by  Theorem~\ref{the:constructing_slice_models} and
  Theorem~\ref{the:Checking_Poincare_duality}.
  Since we may change the orientations of $S_F'$ for $F \in \calm$ satisfying $|F| = 2$,
  we can asume without loss of generality that $m_F(X,\partial X) = m(X',\partial X')$ hold for every $F \in \calm$.
 From  Lemma~\ref{lem:extensions_over_Gamma_from_partial_X_to_X} we obtain $\Gamma$-maps of
  $\Gamma$-$CW$-pairs
  \begin{eqnarray*}
    (U,\partial u) \colon (Y,\partial Y) & \to & (X,\partial X);
    \\
    (U',\partial u') \colon (Y,\partial Y) & \to & (X',\partial X').
  \end{eqnarray*}
  
  Since $m_F(X,\partial X) = m(X',\partial X')$ hold for every $F \in \calm$ and there is
  for every $F \in \calm$ an orientation preserving $F$-homotopy equivalence
  $S_F \to S_F'$, there is a $\Gamma$-homotopy equivalence
  $\partial f \colon \partial X \to \partial X'$ such that $\partial f \circ \partial u$
  is $\Gamma$-homotopic to $\partial u'$ and $\partial f$ extends to a $\Gamma$-homotopy equivalence
  $C(\partial X) \to C(\partial X')$.
  Now define $Z$ and $Z'$ by the $\Gamma$-pushouts
  \[\xymatrix{\partial Y \ar[r]^{\partial u} \ar[d] & \partial X \ar[d]
      \\
      Y \ar[r] & Z}
    \quad \raisebox{-5mm}{\text{and}} \quad \xymatrix{\partial Y \ar[r]^{\partial u'}
      \ar[d] & \partial X' \ar[d]
      \\
      Y \ar[r] & Z'}
  \]
  The $\Gamma$-maps $(U,\partial u)$ and $(U',\partial u')$ yield $\Gamma$-maps
  $(V,\id_{\partial X}) \colon (Z,\partial X) \to (X,\partial X)$ and
  $(V',\id_{\partial X'}) \colon (Z',\partial X') \to (X',\partial X')$.  Note that 
  $ (Z,\partial X)$ and $(Z',\partial X')$ are slice complement models by
  Lemma~\ref{lem:slice-models_and_pushouts} and the canoncial isomorphisms
  $H_d^{\pi}(Y,\partial Y;\IZ^w) \xrightarrow{\cong} H_d^{\pi}(Z,\partial X;\IZ^w)$
  and
  $H_d^{\pi}(Y,\partial Y;\IZ^w) \xrightarrow{\cong} H_d^{\pi}(Z',\partial X';\IZ^w)$
  respect the fundamental classes.   Lemma~\ref{lem:maps_of_slice_models_homotopy_equivalence} implies that
  $(V,\id_{\partial X})$ and $(V',\id_{\partial X'})$ are $\Gamma$-homotopy equivalences of
  $\Gamma$-$CW$-pairs and respect the fundamental classes.
  Obviously the $\Gamma$-homotopy equivalence
  $\partial f \colon \partial X \to \partial X'$ satisfying
  $\partial f \circ \partial u \simeq_{\Gamma} \partial u'$ extends to a $\Gamma$-homotopy
  equivalence of $\Gamma$-$CW$-pairs
  $(f,\partial f) \colon (X,\partial X) \to (X',\partial X')$ such that $(f,\partial f) \circ  (V,\id_{\partial X})$
  and $(V',\id_{\partial X})$ are $\Gamma$-homotopic. 
  This finishes the proof of Theorem~\ref{the:Homotopy_classification_of_slice_models}.
\end{proof}

  \begin{theorem}[Uniqueness of Poincare slice complement models]\label{the:Uniqueness_of_Poincare_slice_models}
    Suppose that  Assumption~\ref{ass:general_assumptions} is satisfied.
    Let $(X,\partial X)$and $(X',\partial X')$ be two Poincar\'e slice
    models with respect to the  slice systems $\cals$ and
    $\cals'$.Then:

    \begin{enumerate}
    \item\label{the:Uniqueness_of_Poincare_slcie_models:slice_systems} The slice systems
      $\cals$ and $\cals'$ can be oriented in such a way that  $m_F(X,\partial X) = m_F(X',\partial X') = 1$ holds for $F \in \calm$
      and both satisfy condition (S). Moreover, $\cals$ and $\cals'$ are oriented homotopy equivalent
      in the sense that for every $F \in \calm$ there exists an orientation preserving
      $F$-homotopy equivalence $S_F \to S_F'$;

     \item\label{the:Uniqueness_of_Poincare_slice_models:slice_models}
    
    There exists a $\Gamma$-homotopy equivalence
    $(f,\partial f) \colon (X,\partial X) \to (X',\partial X')$ of free
    $\Gamma$-$CW$-pairs with the properties that the $\Gamma$-map $\partial f$ extends to
    a $\Gamma$-homotopy equivalence of $\Gamma$-$CW$-pairs
    $ C(\partial X) \to C(\partial X')$ and the isomorphism
    $H^{\Gamma}_d(X,\partial X;\IZ^w)\xrightarrow{\cong} H^{\Gamma}_d(X',\partial
    X';\IZ^w)$ induced by $(f,\partial f)$ sends $[X,\partial X]$ to $[X',\partial X']$.
  \end{enumerate}
\end{theorem}
\begin{proof}~\eqref{the:Uniqueness_of_Poincare_slcie_models:slice_systems}
   This follows from Lemma~\ref{lem:basics_about_free_slice_systems} and
  Lemma~\ref{lem:necessary_condition_for_slice_model_to_be_Poincare}.
    \\[1mm]~\eqref{the:Uniqueness_of_Poincare_slice_models:slice_models}
    This follows from assertion~\eqref{the:Uniqueness_of_Poincare_slcie_models:slice_systems}
    and Theorem~\ref{the:Homotopy_classification_of_slice_models}.
  \end{proof}


  \typeout{-------------------------- Section 9: Simple homotopy classification of slice
    models --------------------------}

  \section{Simple homotopy classification of slice complement models}%
\label{sec:Simple_homotopy_classification_of_slice_models}

 \begin{theorem}[Simple homotopy classification]\label{the:Simple_homotopy_classification}
  Suppose that  Assumption~\ref{ass:general_assumptions} is satisfied.
  Let $(X,\partial X)$ and $(X',\partial X')$ be Poincar\'e slice complement models with respect to
  the slice systems $\cals$ and $\cals'$. Suppose that $\cals$ and $\cals'$ satisfy condition (S).
  Assume  that the following conditions are satisfied:

  \begin{itemize}
  \item The Farrell-Jones Conjecture for $K$-theory holds for $\IZ \Gamma$;
  \item For all $F \in \calm$ the $2$-Sylow subgroup of $F$ is cyclic;
  \item The Poincar\'e structures on $(X,\partial X)$ and $(X',\partial X')$ are simple;
  \item For every $F \in \calf$ the $F$-homotopy equivalence $v_F \colon S_F \to S_F'$,
    which is uniquely determined by the property that it sends $[S_F]$ to $[S_F']$, is a
    simple $F$-homotopy equivalence.
  \end{itemize}

  Then the $\Gamma$-homotopy equivalence
  $(f,\partial f) \colon (X,\partial X) \to (X',\partial X')$ of
  Theorem~\ref{the:Uniqueness_of_Poincare_slice_models}~\eqref{the:Uniqueness_of_Poincare_slice_models:slice_models}
  is a simple homotopy equivalence of free $\Gamma$-$CW$-pairs, i.e., both $\partial f$
  and $f$ are simple $\Gamma$-homotopy equivalences.
\end{theorem}
\begin{proof}
  For $M \in \calm$, let $i(F) \colon F \to \Gamma$ be the inclusion and
  $i(F)_* \colon \Wh(F) \to \Wh(\Gamma)$ be the induced homomorphism on the Whitehead
  groups.  The map
  \begin{equation}
    \bigoplus_{F \in \calm} i(F)_* \colon \bigoplus_{F \in \calm} \Wh(F) \to \Wh(\Gamma)
    \label{iso__computing_wh(Gamma)}
  \end{equation}
  is bijective.  This follows by inspecting the proof
  of~\cite[Theorem~5.1~(d)]{Davis-Lueck(2003)}, which works also for $\Lambda = \IZ$ in
  the notation used there, or from~\cite[Theorem~5.1]{Davis-Lueck(2022_manifold_models)}.

  We get from the assumptions $\tau(v_F) = 0$ in $\Wh(F)$. Since we have equipped $\cals$
  and $\cals'$ with their canoncial orientations,
  Lemma~\ref{lem:partial_f_yields_partial_f_F}~\eqref{lem:partial_f_yields_partial_f_F:S_F_to_S_F_prime} implies
  $\partial f = \coprod_{F \in \calm} \id_{\Gamma} \times_F v_F$.  Hence we get
  \begin{equation}
    \tau(\partial f) = \sum_{F \in \calm} i(F)_*(\tau(\partial v_F)) = 0.
    \label{tau(partial_f)_is_0}
  \end{equation}
  Equip $\Wh(F)$ and $\Wh(\Gamma)$ with the involutions coming from the $w$-twisted
  involution on $\IZ \Gamma$ and the untwisted involution on $\IZ F$. Recall that
  $w|_F = 0$ holds by assumption.  Hence the isomorphism~\eqref{iso__computing_wh(Gamma)}
  is compatible with the involutions.
  
  Since $(X,\partial X)$ and $(X',\partial X')$ are simple Poincar\'e pairs by assumption,
  we conclude
  \begin{equation}
    \tau(f) + (-1)^d \ast(\tau(f)) = \tau(\partial f) = 0
    \label{formula_fortau(u)}
  \end{equation}
  from equation~\eqref{tau(partial_f)_is_0} and the diagram of $\IZ \Gamma$-chain
  homotopy equivalences, which commutes up to $\IZ \Gamma$-chain homotopy,
  \[
    \xymatrix@!C=8em{C^{d-*}(X) \ar[d]_{-\cap [X/\Gamma, \partial X /\Gamma]} &
      C^{d-*}(X') \ar[d]^{-\cap [X'/\Gamma, \partial X' /\Gamma]} \ar[l]_{C^{d-*}(f)}
      \\
      C_d(X,\partial X) \ar[r]_{C_*(f,\partial f)} & C_*(X',\partial X').  }
  \]
  Since $F$ acts freely on $S_F$ and $S_F$ is homotopy equivalent to the $(d-1)$-dimensional
  standard sphere, the cohomology of $F$ is periodic. This implies that the $p$-Sylow subgroup
  of $F$ is finite cyclic if $p$ is odd, see~\cite[Proposition~9.5 in Chapter VI on page~157]{Brown(1982)}.
  Since the $2$-Sylow subgroup is cyclic by assumption,
  $SK_1(\IZ F)$ vanishes, see Oliver~\cite[Theorem~14.2~(i) on page~330]{Oliver(1988)}.
  Hence $\Wh(F)$ is a finitely generated free abelian group and agrees with
  $\Wh'(F) := \Wh(F)/ \tors(\Wh(F))$. The involution on $\Wh'(F)$ and hence also on
  $\Wh(F)$ is trivial, see~\cite[Corollary~7.5 on page~182]{Oliver(1988)}.  Hence the
  involution on $\Wh(\Gamma)$ is trivial and $\Wh(\Gamma)$ is torsionfree because of the
  isomorphism~\eqref{iso__computing_wh(Gamma)}, which is compatible with the
  involutions. Since $d$ is even, the equation~\eqref{formula_fortau(u)} boils down to
  $2 \cdot \tau(f) = 0$.  Since $\Wh(\Gamma)$ is torsionfree, we conclude $\tau(f) = 0$.
\end{proof}


  \typeout{-------------------------- Section 10: The case of odd $d$ --------------------------}

  \section{The case of odd $d$}\label{sec:The_case_of_odd_d}

  Recall that from Section~\ref{subsec:Constructing_slice-models} on we have assumed that
  $d \ge 4$ and $d$ is even. We want to explain the case that $d$ is odd and $d \ge 3$ in
  this section.  Recall from Lemma~\ref{lem:basics_about_free_slice_systems} that then
  each element $F \in \calm$ is cyclic of order two and $F$ acts orientation reversing on
  $S_F$.  Let $H_{d-1}^F(EF;\IZ^-)$ be the homology of $EF$ with coefficients in the
  $\IZ F$-module $\IZ^-$, whose underlying abelian group is $\IZ$ and on which the
  generator of $F = \IZ/2$ acts by $-\id$.  Hence instead of
  Assumption~\ref{ass:general_assumptions} we will make in this section the following
  assumption:

  \begin{assumption}\label{ass:d_odd}\

\begin{itemize}

\item The natural number $d$ is odd and satisfies $d \ge 3$;
  \item The group $\Gamma$ satisfies  conditions (M) and (NM), see Notation~\ref{not:(M)_and_(NM)_intro};
\item There exists a finite $\Gamma$-$CW$-model for $\eub{\Gamma}$ of dimension $d$ such
  that its singular $\Gamma$-subspace $\eub{\Gamma}^{> 1}$ is
  $\coprod_{F\in\calm} \Gamma/F$.  (This condition is discussed and
  simplified in Theorem~\ref{the:Models_for_the_classifying_space_for_proper_Gamma-actions_intro}
  and Remark~\ref{rem:fin_dom_intro}  and implies conditions (M) and (NM), see Remark~\ref{Theorem_models_and_necessary_conditions}.)

\item There is a finite $d$-dimensional Poincar\'e $CW$-complex model for $B\pi$ with
  respect to the orientation homomorphisms $v \colon \pi \to \{\pm 1 \}$. We have made a
  choice of a fundamental class $[B\pi] \in H_d^{\pi}(E\pi;\IZ^v)$;

\item The homomorphism $w \colon \Gamma \to \{\pm 1\}$ of Notation~\ref{not:w} has the
  property that $w|_F$ is non-trivial for every $F \in \calm$;

\item The composite
  \[
    H_d^{\Gamma}(E\Gamma,\partial E\Gamma;\IZ^w) \xrightarrow{\partial}
    H^{\Gamma}_{d-1}(\partial E\Gamma;\IZ^w ) \xrightarrow{\cong} \bigoplus_{F \in \calm}
    H_{d-1}^F(EF;\IZ^-) \to H_{d-1}^F(EF;\IZ^-)
  \]
  of the boundary map, the inverse of the obvious isomorphism, and the projection to the
  summand of $F \in \calm$ is surjective for all $F \in \calm$.
  
\end{itemize}

\end{assumption}

Besides the simplification that each element $F \in \calm$ is cyclic of order two, it is
also convenient that we have to consider only one slice system as explained next.  Recall
that $S_F$ is homotopy equivalent to $S^{d-1}$. The antipodal action of $F = \IZ/2$ on
$S^{d-1}$ is free and reverses the orientation.  Equivariant obstruction theory implies
that there is an $F$-homotopy equivalence $S_F \xrightarrow{\simeq_F} S^{d-1}$,
see~\cite[Theorem~4.11 on page~126]{Dieck(1987)} or~\cite[Theorem~3.5 on page~139]{Lueck(1988)}.
Moreover, any  $F$-selfhomotopy equivalence $S^{d-1} \xrightarrow{\simeq_F} S^{d-1}$
is homotopic to an $F$-homeomorphisms, namely to the identity  or the antipodal selfmap.
This implies that we only have to consider only one slice system
$\cals^{\operatorname{st}} = \{S_F^{\operatorname{st}} \mid F \in \calm\}$, namely the
one, where each $S_F^{\operatorname{st}}$ is the standard $(d-1)$-dimensional sphere
$S^{d-1}$ with the antipodal $F = \IZ/2$-action.  An orientation for it is a choice of
fundamental class $[S^{\operatorname{st}}_F]$ in the infinite cyclic group
$H_{d-1}^F(S^{\operatorname{st}}_F;\IZ^-)$, which corresponds to a choice of fundamental
class $[S^{\operatorname{st}}_F]$ in the infinite cyclic group
$H_{d-1}(S^{\operatorname{st}}_F)$, since the canonical map
$H_{d-1}(S^{\operatorname{st}}_F) \to H_{d-1}^F(S^{\operatorname{st}}_F;\IZ^-)$
is an inclusion of infinite cyclic groups with index $2$. The third simplification is
that $\Wh_n(F)$ vanishes for $n \le 1$. Hence $\Wh_n(\Gamma)$ vanish for $n \le 1$,
see~\cite[Theorem~5.1~(d)]{Davis-Lueck(2003)} or~\cite[Theorem~5.1]{Davis-Lueck(2022_manifold_models)}.
This is interesting in view of Remark~\ref{rem:fin_dom_intro}.

The proofs of Theorem~\ref{the:existence_Poincare_slice_models} and
Theorem~\ref{the:Uniqueness_of_Poincare_slice_models} carry directly over (and actually
simplify) to the following theorems; one has to replace $H_{d-1}(BF)$ and $H_{d-1}(S_F/F)$
by $H_{d-1}^F(EF;\IZ^-)$ and $H^F_{d-1}(S_F,\IZ^-)$ everywhere.

  \begin{theorem}[Existence of Poincar\'e slice complement models in the odd dimensional case]%
\label{the:existence_Poincare_slice_models_in_the_odd_dimensional_case}
 Suppose that Assumption~\ref{ass:d_odd} is satisfied.
 Then  there exists a Poincar\'e slice complement model $(X,\partial X)$ with respect to
    $\cals^{\operatorname{st}}$. 
  \end{theorem}
     
  \begin{theorem}[Homotopy classification of Poincar\'e
    slice complement models in the odd dimensional case]%
\label{the:Homotopy_classification_of_Poincare_slice_models_in_the_odd_dimensional_case}
    Suppose that Assumption~\ref{ass:d_odd} is satisfied. Let $(X,\partial X)$ and
    $(X',\partial X')$ respectively be Poincar\'e slice complement models for $\eub{\Gamma}$
    with respect to  $\cals^{\operatorname{st}}$.

    Then there exists a simple $\Gamma$-homotopy equivalence
    $(f,\partial f) \colon (X,\partial X) \to (X',\partial X')$ of free
    $\Gamma$-$CW$-pairs such that $\partial f$ is the identity on
    $\coprod_{F \in \calm} \Gamma \times _F S_F$ and the isomorphism
    $H^{\Gamma}_d(X,\partial X;\IZ^w)\xrightarrow{\cong} H^{\Gamma}_d(X',\partial
    X';\IZ^w)$ induced by $(f,\partial f)$ sends $[X,\partial X]$ to $[X',\partial X']$.
\end{theorem}


\typeout{----------------------------- References ------------------------------}

\addcontentsline{toc<<}{section}{References} \bibliographystyle{abbrv}



\end{document}